\documentclass[10pt]{article}
\usepackage{amsfonts,latexsym}
\usepackage{amsmath}
\usepackage{amsthm}
\usepackage{amssymb}
\usepackage{mathrsfs}
\usepackage{hyperref}
\usepackage{color}
\usepackage[]{cite}
\usepackage{graphicx}
\usepackage{appendix}
\allowdisplaybreaks[4]
\usepackage{indentfirst}
\usepackage{geometry}
\newtheorem{theorem}{Theorem}[section]
\newtheorem{proposition}[theorem]{Proposition}
\newtheorem{definition}[theorem]{Definition}
\newtheorem{hypothesis}[theorem]{Hypothesis}
\newtheorem{condition}[theorem]{Condition}
\newtheorem{lemma}[theorem]{Lemma}
\newtheorem{remark}[theorem]{Remark}

\begin{document}
\title{\bf Pullback Measure Attractors, Zero-Noise Limits, and Moderate Deviations for 2D Stochastic Primitive Equations with Multiplicative L\'{e}vy Noise \footnote{The research is supported by the China Postdoctoral Science Foundation under Grant Number 2026M794838.}}

\author{{Jiangwei Zhang$^\text{a}\footnote{Corresponding author.}$,\quad Boling Guo$^\text{a}$, \quad  Juntao Wu$^\text{b}$
	}\\
	{ \small\textsl{$^\text{a}$ Institute of Applied Physics and Computational Mathematics, }}\\
	{ \small \textsl{Beijing 100088,  P.R. China}}\\
	{ \small\textsl{$^\text{b}$ School of Mathematics and Statistics, Wuhan University, }}\\ 
	{ \small \textsl{Wuhan, Hubei 430072, P.R. China}}
}
\footnotetext{
	\emph{E-mail addresses}: zjwmath@163.com (J. Zhang), gbl@iapcm.ac.cn (B. Guo), 00036371@whu.edu.cn (J. Wu).
}
\date{}

%%%%%%%%%%%%%%%%%%%%%%%%

\renewcommand{\theequation}{\arabic{section}.\arabic{equation}}
\numberwithin{equation}{section}

\maketitle

%%%%%%%%%%%%%%%%%%%%%
\begin{abstract} 
We study the long-term distributional dynamics and small-noise asymptotics of two-dimensional nonautonomous stochastic primitive equations driven by Gaussian and multiplicative L\'evy noise, together with moderate deviations for the purely jump model. For sufficiently small noise, uniform moment bounds and an exponentially weighted terminal estimate yield a pullback absorbing family and tightness, while a lower semicontinuous vertical-moment functional preserves admissibility under weak limits. We prove the existence and uniqueness of a pullback measure attractor in the weak topology of probability measures and establish its upper semicontinuity as both noise components vanish. For the jump-driven equation, we establish a moderate deviation principle in
$\mathcal D([0,T];H)\cap L^2(0,T;V)$
with speed $a^2(\epsilon)/\epsilon$. The proof combines continuity of the skeleton map with controlled stochastic convergence based on entropy bounds, truncation, martingale estimates, and direct vertical estimates, avoiding Lipschitz continuity of the vertical derivative of the jump coefficient.

	\medskip
	
\noindent \textbf{Keywords}: {Stochastic primitive equation, L\'{e}vy noise, Pullback measure attractor, Zero-noise limit, Moderate deviation.}

\noindent  \textbf{MSC 2020}: {Primary 35B40, 37L55, 60H15; Secondary 35B41, 35R60, 60F10.}
\end{abstract}

	\tableofcontents

\section{Introduction}
The primitive equations are the fundamental hydrostatic model for large-scale atmospheric and oceanic motion. They arise from the Navier--Stokes equations by exploiting the small aspect ratio of geophysical flows and replacing the vertical momentum equation by the hydrostatic balance. This reduction preserves the principal transport and diffusion mechanisms while producing a highly anisotropic system in which the vertical velocity is not an independent unknown.
 In the present paper, we consider the following nonautonomous 2D stochastic primitive equations on the rectangle $\mathcal M=[0,L]\times[-l,0]$:
\begin{equation}\label{PEeq-1}
	\begin{cases}
		 dv-\gamma\Delta v\,dt+v\partial_xv\,dt+w\partial_zv\,dt+\partial_xp\,dt
		=g(t)\,dt\\
		\qquad \qquad +\sqrt\epsilon\,\sigma(t,v)\,dW(t)
		+\epsilon\displaystyle\int_E h(v(t-),\xi)\widetilde N^{\epsilon^{-1}}(dt,d\xi),\quad t>\tau,\\
		\partial_zp=0,\\
		\partial_xv+\partial_zw=0.
	\end{cases}
\end{equation}
Here, $v=v(t,x,z)$ and $w=w(t,x,z)$ denote the horizontal and vertical velocities, respectively, and $p$ is the pressure. The term $g=g(t,x,z)$ is a time-dependent deterministic forcing, and $\epsilon\in(0,1)$ denotes the noise intensity. We write $\Delta=\partial^{2}_{x}+\partial^{2}_{z}$ for the Laplacian. 
The noise coefficients $\sigma$ and $h$ are measurable functions. $W$ is a two-sided cylindrical Wiener process, and $\widetilde{N}^{\epsilon^{-1}}$ is a compensated Poisson random measure. The precise definitions of these stochastic objects are given in the next section. We assume throughout that $W$ and $\widetilde{N}^{\epsilon^{-1}}$ are independent.

The hydrostatic relation implies that the pressure is independent of the vertical variable. More importantly, the incompressibility condition and the boundary conditions determine $w$ diagnostically by
\[
w(t,x,z)=-\int_{-l}^{z}\partial_xv(t,x,z')\,dz'.
\]
Consequently, the horizontal momentum equation contains the nonlocal anisotropic term
\[
-\left(\int_{-l}^{z}\partial_xv(t,x,z')\,dz'\right)\partial_zv(t,x,z).
\]
This term is the main structural distinction between the primitive equations and the standard 2D Navier--Stokes equation. In particular, the basic $H$-energy estimate is not sufficient for the continuity, compactness, and controlled convergence arguments needed below. Estimates for the vertical derivative are therefore an intrinsic part of the analysis, rather than an auxiliary regularity improvement that can be omitted without changing the proof.

The deterministic primitive equations have been studied extensively since the foundational formulations of Lions, Temam, and Wang; see \cite{Pedlosky-1987,Lions-Non-1992(1),Lions-Non-1992(2),Lions-Non-1992(3),Lions-Non-1992(4)}. Global well-posedness, regularity, and long-time dynamics were subsequently developed in, among many other works, \cite{Guill-DIE-2001,Hu-DCDS-2003,Petcu-CPAA-2004,Ju-DCDS-2007,Kobelkov-JMFM-2007,Cao-Titi-Anna-2007,Cao-CPAM-2016,Ju-JNS-2015,You-ZAMP-2019,WGH-3M-2023,You-JDE-2025}. Random forcing is natural in geophysical fluid models because unresolved scales, rapidly varying external inputs, and model uncertainty cannot be represented completely by a deterministic source. For Gaussian perturbations, the stochastic primitive equations have been investigated from the viewpoints of well-posedness, regularity, random attractors, invariant measures, and deviation asymptotics; see \cite{Glatt-DCDSB-2008,Guo-CMP-2009,Glatt-Holtz-AMO-2011,Debussche-Non-2012,Gao-CMS-2012,Glatt-JMP-2014,Dong-JDE-2017,ZRR-CMA-2019,Zhou-PhysicaD-2019,Brze-JDE-2021,Dong-Z-JDDE-2022,Slav-JTP-2022,Zhou-SD-2025}. In comparison, results for jump perturbations are more limited. Sun and Gao \cite{SunGao-2013} established well-posedness for 2D primitive equations driven by L\'{e}vy noise, and Ma and Zhang \cite{ZRR-2021} obtained a large deviation principle (LDP) for multiplicative L\'{e}vy perturbations. These results provide the finite-time stochastic foundation for the present work, but they do not by themselves yield the nonautonomous distributional dynamics, its zero-noise stability, or the moderate deviation asymptotics considered here.

Our first objective is to describe the long-time behavior of probability distributions generated by \eqref{PEeq-1}. Since $g$ and $\sigma$ may depend explicitly on time, the resulting Markov family is time-inhomogeneous and is described by a two-parameter evolution rather than an autonomous semigroup. Pullback measure attractors are designed precisely for this situation: they describe the compact family of asymptotic probability laws observed at a fixed terminal time as the initial time is sent to the remote past. The measure attractor framework originates in \cite{schmalfuss1991,schmalfuss1999,Mar-EJP-1998} and has recently been developed for nonautonomous stochastic partial differential equations in \cite{LDS-JDE-2024,LYR-QTDS-2025,LDS-QTDS-2024,LZH-AML,LDS-CNSNS-2025,WRH-CSF-2025,LZWH-BMMS-2025,BCS-JDE-2025,Zhang-arx-2026}. Applying this theory to \eqref{PEeq-1} requires substantially more than a direct invocation of an abstract result. One must establish a continuous measure evolution process, a pullback absorbing family, and pullback asymptotic compactness in a probability space with a fourth-moment constraint. Each of these steps depends on estimates compatible with the anisotropic convection and with the simultaneous presence of Gaussian and jump noise.

The analysis uses the existing well-posedness theorem for $\mathcal H$-valued initial data having finite fourth moments of the state and its vertical derivative. Accordingly, the transition family acts on the admissible law class $\mathcal P_{4,z}(H)$, whose elements are supported on $\mathcal H$ and satisfy the corresponding fourth-moment condition. This class provides the integrability needed to restart the equation, while compactness, attraction, and upper semicontinuity are considered in the ambient weak topology of the Polish space $\mathcal P(H)$.

The second objective concerns the stability of the long-time statistical dynamics when the stochastic perturbations disappear. Zero-noise limits for attractors are more delicate than finite-time convergence of individual solutions. Even when the stochastic and deterministic trajectories are close on a fixed interval, upper semicontinuity of attractors also requires estimates that are uniform with respect to the noise intensity and compatible with the pullback limit. Our proof follows a direct dynamical argument. We first derive convergence of the stochastic evolution to the deterministic evolution on every fixed time interval. We then use a common pullback absorbing family, invariance of the stochastic attractors, and the pullback attraction of the deterministic attractor. A triangle estimate at a suitably chosen pullback time separates the finite-time perturbation error from the deterministic attraction error. This yields upper semicontinuity of the stochastic pullback measure attractors toward the deterministic pullback measure attractor as the Gaussian and jump-noise intensities vanish simultaneously. The structure of this argument is related to the measure attractor method in \cite{MiLiZeng2024,WRH-CSF-2025}, but its implementation here depends on the primitive equation estimates developed in the preceding sections.

Our third objective is a moderate deviation principle (MDP) for the purely jump-driven primitive equation. Moderate deviations quantify fluctuations on a scale lying between the central limit regime and the large deviation regime. If $v^\epsilon$ denotes the jump-driven stochastic solution and $v^0$ the deterministic solution, we study
\begin{equation}\label{PEeq-MDPdetdec4.3}
\mathcal Y^\epsilon=\frac{v^\epsilon-v^0}{a(\epsilon)},
\end{equation}
where
\begin{equation}\label{PEeq-MDPdetdec4.4}
a(\epsilon)\to0,
\qquad
\frac{\epsilon}{a^2(\epsilon)}\to0
\quad\text{as }\epsilon\to0.
\end{equation}
Thus $\sqrt\epsilon\ll a(\epsilon)\ll1$. The weak convergence approach of Budhiraja, Dupuis, and Ganguly \cite{Budhiraja-AOP-2016} converts the exponential asymptotic problem into the convergence of equations driven by controlled Poisson random measures. This approach has been used for several stochastic fluid and jump systems; see \cite{DXZZ-2017-JFA,ZZZ-CIMS-2018,Liu-PoAN-2023,WZ-SIAM-2024,LIS-2025-arxiv}. For the primitive equations, however, the two key verification steps require different compactness and nonlinear estimates.

The first difficulty is the continuity of the skeleton map when the controls converge only weakly in $L^2(\nu_T)$. Weak convergence does not imply that the norm of the difference of the controls tends to zero, so a direct Gronwall estimate based on $\|q^\epsilon-q\|_{L^2(\nu_T)}$ is unavailable. We instead establish uniform bounds for the skeleton solutions, obtain strong compactness in $L^2(0,T;H)$ from spatial and temporal estimates, and identify every limit by pairing the weakly convergent controls with fixed test functions. The strong $L^2(0,T;H)$ convergence then makes the remaining control--solution pairing vanish in the difference energy identity, which upgrades the convergence to $C([0,T];H)\cap L^2(0,T;V)$.

The second difficulty is the convergence of the controlled stochastic fluctuations to the skeleton dynamics. The controlled equation contains a compensated jump martingale, a nonlinear state-dependent coefficient $h$, intensity perturbations controlled only through an entropy bound, and the anisotropic convection inherited from the primitive equations.  The proof requires uniform state and vertical derivative estimates, a bounded/tail decomposition of the controls, entropy estimates on the tail region, tightness in the strong path space, and identification of all nonlinear and controlled drift terms.  A key step is to test the small control remainder directly against $-\partial_{zz}\mathscr S^\epsilon$.  The resulting estimate relies only on the Lipschitz continuity of $h$ on $H$ and the deterministic $L^2$ bound for the truncated control; no Lipschitz assumption on $\partial_z h$ is needed.

It is useful to clarify the relation between the present work and the closest existing results. The Gaussian CLT and MDP established in \cite{ZRR-CMA-2019} do not involve controlled Poisson intensities or the entropy tail estimates needed for jump controls. The L\'{e}vy noise LDP in \cite{ZRR-2021} is developed at the large deviation scale and therefore does not yield the linearized skeleton that governs moderate deviations. A recent MDP for stochastic 2D primitive equations with only horizontal viscosity \cite{Sun-Wang-Liu-2026} treats a different anisotropic-viscosity model; in the present multiplicative L\'{e}vy setting, the weak convergence argument requires controlled Poisson intensities, entropy estimates for control truncation, and identification of the corresponding jump martingales. The MDP framework for abstract hydrodynamic systems driven by L\'evy noise, proposed in \cite{LIS-2025-arxiv}, does not eliminate these model-specific tasks, as the diagnostic vertical velocity necessitates additional estimates for both the state and its vertical derivative. Similarly, the measure-attractor framework in \cite{Zhang-arx-2026} does not directly establish the primitive-equation convection estimates, the terminal $V$-tightness, or the zero-noise comparison employed in the present work. In particular, the nonlocal term $W(v)\partial_zv$ must be controlled through the vertical regularity mechanism developed below rather than by the basic $H$-energy estimate alone.

The main results and the technical contributions of the paper are summarized below.

\medskip
\noindent\textbf{$(i)$ Uniform estimates and pullback measure attractors.}
We derive uniform second- and fourth-moment estimates for the stochastic solution, fourth-moment estimates for its vertical derivative, and an exponentially weighted estimate in $V$. The vertical derivative estimate controls the diagnostic convection term, while the weighted $V$-estimate supplies compactness at the terminal time. We prove continuity of the solution map in the relative $H$-topology of $\mathcal H$ and the corresponding Feller property for the induced two-parameter evolution of admissible laws. We then construct a closed $\mathfrak D$-pullback absorbing family, establish pullback asymptotic compactness through the compact embedding $V\hookrightarrow H$, and obtain a unique $\mathfrak D$-pullback measure attractor contained in $\mathcal P_{4,z}(H)$ and compact in the ambient topology $(\mathcal P(H),d_{\mathcal P(H)})$. These conclusions hold for every sufficiently small noise intensity $0<\epsilon<\epsilon_0$.

\medskip
\noindent\textbf{$(ii)$ Zero-noise upper semicontinuity.}
We prove finite-time convergence of the stochastic solutions to the deterministic solutions uniformly on bounded classes of admissible initial data. The estimates used to construct the pullback absorbing family are uniform for sufficiently small noise intensities. Combining these facts with the invariance of the attractors and the pullback attraction of the deterministic system gives
\[
\operatorname{dist}_{w,H}\bigl(\mathscr A_\epsilon(\tau),\mathscr A_0(\tau)\bigr)\longrightarrow 0
\quad\text{for every }\tau\in\mathbb R.
\]
The proof is carried out directly at the level of the measure evolution by combining invariance with the common pullback absorbing family.

\medskip
\noindent\textbf{$(iii)$ Moderate deviations.}
For the purely jump-driven equation, we first construct the controlled equation by stopped Galerkin approximation, prove nonexplosion and identify the martingale limit, and then establish pathwise uniqueness by a localized stochastic Gronwall argument.  This yields Borel solution maps for both the controlled fluctuation equation and the residual equation used after the Skorokhod representation.  We next prove compactness of the skeleton dynamics and continuity of the skeleton map under weak convergence of the controls.  Finally, a direct vertical energy estimate for the small control remainder closes the controlled stochastic convergence without imposing a Lipschitz condition on $\partial_z h$.  Consequently, the rescaled fluctuations satisfy a MDP in $\mathcal D([0,T];H)\cap L^2(0,T;V)$ with conventional speed $a^2(\epsilon)/\epsilon$. The good rate function is
\[
I(\eta)=\inf\left\{\frac12\|\psi\|_{L^2(\nu_T)}^2:\ (\eta,\psi)\text{ solves the skeleton equation}\right\}.
 \]
The argument treats separately weak control compactness, controlled well-posedness, measurable reconstruction, and the model-specific vertical convergence estimate.

Taken together, these results connect three levels of asymptotic analysis for a single geophysical fluid model: the pullback statistical dynamics at each sufficiently small fixed noise intensity, the stability of that dynamics in the zero-noise limit, and the probabilities of fluctuations on intermediate scales around the deterministic trajectory. The common analytical mechanism is the combination of vertical estimates specific to the primitive equations with measure evolution and weak convergence methods.

The remainder of the paper is organized as follows. Section~\ref{Pre-PEs2} introduces the functional framework, the stochastic setting, the assumptions on the coefficients, and the well-posedness results used later. Section~\ref{PMA-PEs3} establishes the uniform pullback estimates, the Feller property, the pullback absorbing family, pullback asymptotic compactness, and the existence and uniqueness of the pullback measure attractors. Section~\ref{USC-PEs4} proves the zero-noise upper semicontinuity of the pullback measure attractors. Section~\ref{MDP-PEs4} develops the skeleton equation, verifies the two weak convergence requirements for controlled Poisson random measures, and proves the MDP.

\section{Preliminaries}\label{Pre-PEs2}

For the stochastic primitive equations \eqref{PEeq-1}, the boundary $\partial \mathcal{M}$ is partitioned into the top boundary $\Gamma_u$, the bottom boundary $\Gamma_b$, and the lateral boundary $\Gamma_l$, defined by
$$
\Gamma_u:=\{(x,z)\in {\mathcal{M}}: z=0\},~~
\Gamma_b:=\{(x,z)\in {\mathcal{M}}: z=-l\},~~
\Gamma_l:=\{(x,z)\in {\mathcal{M}}: x\in \{0,L\}\}.
$$
The stochastic PEs \eqref{PEeq-1} are supplemented by the following boundary conditions
\begin{align}
	&\text{on }\Gamma_u:\,\,\,\, \partial_{z}v=0, \ w=0, \label{PEeq-2}\\
	&\text{on }\Gamma_b:\,\,\,\, \partial_{z}v=0,\  w=0, \label{PEeq-3}\\
	&\text{on }\Gamma_l:\,\,\,\,  v=0, \label{PEeq-4}
\end{align}
and the initial condition
\begin{equation}\label{PEeqin-5}
	v(\tau)=\zeta.
\end{equation}
Integrating the incompressibility relation in \eqref{PEeq-1} from $-l$ to $z$ and using $w=0$ on $\Gamma_b$, we obtain
\begin{equation}\label{PEeq-5}
	w(t,x,z)=-\int^{z}_{-l}\partial_x v (t,x,z')\,dz':=\mathcal{W}(v)(t,x,z).
\end{equation}
We normalize the viscosity coefficient by setting $\gamma=1$. The hydrostatic compatibility is built into the phase space and coefficient ranges: $\int_{-l}^{0}v(x,z)\,dz=0$ for a.e.\ $x\in(0,L)$, while $g(t)\in H$, $\sigma(t,v)\in\mathcal L_2(U;H)$, and $h(v,\xi)\in H$. Thus the vertical mean of $g(t)$ and $h(v,\xi)$, and that of $\sigma(t,v)u$ for every $u\in U$, vanishes.

With the above notation, equation \eqref{PEeq-1} with the initial and boundary conditions \eqref{PEeq-2}-\eqref{PEeqin-5} is equivalently written as follows:
\begin{equation}\label{PEeq-6}
	\begin{cases}
			\begin{aligned}
		dv-\Delta vdt+v\partial_{x} vdt&+\mathcal{W}(v)\partial_{z} vdt+\partial_{x} pdt
		 =g(t)dt\\
		& +\sqrt{\epsilon}\sigma(t,v)dW(t)+\epsilon\int_{E}h(v(t-),\xi)\widetilde{N}^{\epsilon^{-1}}(dt,d\xi),
			\end{aligned}\\
		\partial_{z}v|_{\Gamma_{u}}=0,~~
		\partial_{z}v|_{\Gamma_{b}}=0,~~
		v|_{\Gamma_{l}}=0,\\
		v(\tau)=\zeta,
	\end{cases}	
\end{equation}
%}

\subsection{Functional spaces}
For two separable Hilbert spaces $X$ and $Y$, let $\mathcal{L}(X,Y)$  denote the space of all bounded linear operators from $X$ to $Y$, and
$\mathcal{L}_2(X,Y)$ the space of Hilbert-Schmidt operators
from $X$ to $Y$. These spaces are endowed with the norms $\|\cdot\|_{\mathcal{L}(X,Y)}$ and $\|\cdot\|_{\mathcal{L}_2(X,Y)}$, respectively. 
For $1\leq p<\infty$, let $L^p(\mathcal M)$ denote the usual Lebesgue space with norm
\[
|\phi|_p=\left(\int_{\mathcal{M}}|\phi(x,z)|^p\,dx\,dz\right)^{1/p},
\qquad \phi\in L^p(\mathcal{M}).
\]
For $m\in\mathbb N$, let $(W^{m,p}(\mathcal M),\|\cdot\|_{m,p})$ be the usual Sobolev space. When $p=2$, we write $H^m(\mathcal M)=W^{m,2}(\mathcal M)$ and set
\[
H^m(\mathcal M)=\left\{\phi\in L^2(\mathcal M):\partial^\alpha\phi\in L^2(\mathcal M)\text{ for every }|\alpha|\le m\right\},
\qquad
\|\phi\|_m^2=\sum_{0\le |\alpha|\le m}|\partial^\alpha\phi|^2.
\]
Then $(H^{m}(\mathcal M),\|\cdot\|_m)$ is a Hilbert space. We denote by $(\cdot,\cdot)$ and $|\cdot|$ the inner product and norm of $L^2(\mathcal M)$, respectively.
Define the working spaces as follows:
\[
H = \left\{ v \in L^2(\mathcal{M}) : \int_{-l}^0 v \, dz = 0 \right\}
\]
and
\[
V = \left\{ v \in H^1(\mathcal{M}) : \int_{-l}^0 v \, dz = 0,\, v = 0 \text{ on } \Gamma_l \right\}.
\]
The inner product and norm on $H$ are those inherited from $L^2(\mathcal M)$.
For $v_1,v_2\in V$, we equip $V$ with the inner product
$
{\langle v_1,v_2 \rangle_1}=\int_{\mathcal{M}} \partial_{x} v_1 \partial_{x} v_2+ \partial_{z}v_1\partial_{z}v_2dxdz$
and the norm $\|\cdot\|=\sqrt{\langle \cdot,\cdot\rangle_1}$.
By \cite{Alessandrini-2008}, there exists $\textbf{c}_0>0$ such that the following Poincar\'{e} inequality holds:
$$
\sqrt{\textbf{c}_0}|v|\leq \|v\|, ~~~~\forall v\in V.
$$ 
Let $V'$, endowed with the norm $\|\cdot\|_{V'}$, be the dual space of $V$, so that the following embeddings are dense and continuous:
$
V\hookrightarrow H \equiv H'\hookrightarrow V'$. We
denote by $\langle \cdot,\cdot\rangle$ the duality between $V$ and $ V'$. 
Define the intermediate space
\[
\mathcal{H}=\{v\in H:\ \partial_z v\in H\},
\qquad |v|_{\mathcal H}=\sqrt{|v|^2+|\partial_zv|^2}.
\]

\begin{remark}
	The space $\mathcal H$ plays a dual role: it is the admissible initial state class supplied by the well-posedness theorem and it provides the vertical estimates needed for the anisotropic convection. Accordingly, the measure evolution and every attractor section lie in $\mathcal P_{4,z}(H)$, not in all of $\mathcal P_4(H)$ and not in a separately metrized space $\mathcal P_4(\mathcal H)$. The topology used for convergence, compactness, and attraction is nevertheless the weak topology of probability measures on the coarser state space $H$, namely the topology of $\mathcal P(H)$.
\end{remark}

Unless stated otherwise, the symbols $\left(\cdot,\cdot\right)$, $\langle \cdot,\cdot\rangle$, and $\langle\cdot,\cdot\rangle_{1}$ are interpreted according to context.
We write $\mathbb R^+:=[0,\infty)$ and $\mathbb R^\tau:=[\tau,\infty)$. 
Throughout the paper, $C$ denotes a generic positive constant whose value may change from line to line. When we need to emphasize dependence on certain parameters, we write $C_{c_1,c_2,\ldots}$, which may also change from line to line but depends only on $c_1,c_2,\ldots$.

\subsection{Reformulation of the problem}
We denote by the hydrostatic Helmholtz projection $P_H$ the orthogonal projection of $L^{2}(\mathcal{M})$ onto $H$. It is given explicitly by
$$
P_Hv=v-\frac{1}{l}\int_{-l}^0 vdz.
$$
Let
\begin{align*}
	\mathcal{V}:=\left\{v\in C^{\infty}(\overline{\mathcal{M}}):\partial_{z}v|_{\Gamma_u\cup \Gamma_b}=0, v|_{\Gamma_{l}}=0, \int^{0}_{-l} v dz=0 \right\},
\end{align*}
which is a dense subset of \( H \), \( \mathcal{H} \) and \( V \). We now interpret \eqref{PEeq-6} as an equation in $V'$. To this end, we define a Stokes-type operator $A:V\to V'$ by
$\langle Av_1,v_2\rangle = \langle v_1,v_2\rangle_1$.
Viewed as an unbounded operator on $H$, $A=-P_H\Delta$ has domain
\[
D(A)=\left\{v\in H^2(\mathcal M): \int_{-l}^0v\,dz=0,\ v=0\text{ on }\Gamma_l,\ \partial_zv=0\text{ on }\Gamma_u\cup\Gamma_b\right\};
\]
see \cite{Glatt-DCDSB-2008} for details. This realization is positive and self-adjoint and has compact inverse. Let $\{e_j\}_{j=1}^\infty$ be an orthonormal basis of $H$ consisting of eigenfunctions of $A$, with eigenvalues
$0<\lambda_1\leq\lambda_2\leq\cdots$ satisfying $\lambda_j\to\infty$ as $j\to\infty$. For every $v\in D(A)$,
\[
\int_{-l}^0 \partial_{xx} v \, dz = \partial_{xx} \int_{-l}^0 v \, dz = 0, \qquad 
\int_{-l}^0 \partial_{zz} v \, dz = \partial_z v \Big|_{z=-l}^{z=0} = 0.
\]
Consequently, for $v\in D(A)$,
\[
P_H \Delta v = P_H \partial_{xx} v + P_H \partial_{zz} v = \partial_{xx} v + \partial_{zz} v = \Delta v.
\]

For $v,\widetilde{v}\in V$, we define the nonlinear term $B(v, \widetilde{v}):V\times V \rightarrow V'$ as
\begin{eqnarray*}
B(v, \widetilde{v})=P_H(v\partial_x\widetilde{v}+\mathcal{W}(v)\partial_z \widetilde{v}),
\end{eqnarray*}
and define the associated trilinear form $b$ by 
$$
b(v,\widetilde{v}, \widehat{v}):=\langle B(v,\widetilde{v}), \widehat{v} \rangle = \int_{\mathcal{M}} \left(v\partial_x\widetilde{v}\widehat{v}+\mathcal{W}(v)\partial_z \widetilde{v}\widehat{v}\right) dxdz.
$$
We use the following standard estimates; see \cite{Glatt-DCDSB-2008} for details.
\begin{lemma}\label{PEeq-lem2.1}
 For every $v, \widetilde{v}, \widehat{v}\in V$, the trilinear form $b$ is continuous on $V\times V\times V$, and there exists a constant $C>0$ such that
 \begin{align}
&|b(v,\widetilde{v}, \widehat{v})|\leq C\left(|v|^{\frac{1}{2}}\|v\|^{\frac{1}{2}}\|\widetilde{v}\||\widehat{v}|^{\frac{1}{2}}\|\widehat{v}\|^{\frac{1}{2}}+|\partial_xv||\partial_z \widetilde{v}||\widehat{v}|^{\frac{1}{2}}\|\widehat{v}\|^{\frac{1}{2}}\right),\label{PEeq-2.11}\\
&\left| \langle B(\widetilde{v}, \widetilde{v}) - B(\widehat{v}, \widehat{v}), \widetilde{v} - \widehat{v} \rangle \right|
\leq C \|\widetilde{v}\| |\widetilde{v} - \widehat{v}| \|\widetilde{v} - \widehat{v}\|
+ C |\partial_z \widetilde{v}|  |\widetilde{v} - \widehat{v}|^{1/2} \|\widetilde{v} - \widehat{v}\|^{3/2}.\label{PEeq-2.103}
\end{align}
 Moreover, for all $v,\widetilde v,\widehat v\in V$,
 \[
 b(v,\widetilde v,\widetilde v)=0,
 \qquad b(v,\widetilde v,\widehat v)=-b(v,\widehat v,\widetilde v).
 \]
 If, in addition, $v\in D(A)$, then
\begin{align}\label{PEeq-2.13}
	\langle v\partial_xv+\mathcal W(v)\partial_zv,\partial_{zz}v\rangle=0.
\end{align}
The last identity is used at the Galerkin level before passage to the limit.
If \( v, \widetilde{v} \in V \) and \( \partial_z \widetilde{v} \in V \), then
\begin{equation}\label{PEeq-2.0013}
	\|B(v, \widetilde{v})\|_{V'} \leq C\bigl( (|v| + |\partial_z v|)|\partial_{x} \widetilde{v}| + |v| \cdot \|\partial_z \widetilde{v}\| + |v| \cdot |\partial_z \widetilde{v}|^{1/2}\|\partial_z \widetilde{v}\|^{1/2} \bigr).
\end{equation}

\end{lemma}

With this notation, the stochastic primitive equations \eqref{PEeq-6} take the abstract form
\begin{equation}\label{PEeq-2.15}
	\begin{cases}
    dv(t)+Av(t)dt+B(v(t),v(t))dt=g(t)dt+\sqrt{\epsilon}\sigma(t,v(t)) dW(t)\\
    \qquad \qquad  \qquad \qquad  \qquad \qquad  \qquad \quad +{\epsilon}\int_{E}h(v(t-),\xi)\widetilde{N}^{\epsilon^{-1}}(dt,d\xi),\\
     v(\tau)=\zeta,
  \end{cases}
\end{equation}

\subsection{Stochastic setting and assumptions}\label{PEeq-Pre2.3}

$\bullet~\textbf{Notation.}$ Let $X$ be a metric space and $E$ a locally compact Polish space. Denote by $\mathcal B(X)$ the Borel $\sigma$-field on $X$, and by $C_c(X)$ the space of continuous compactly supported functions. We equip $C([a,b];X)$ with the uniform topology and $\mathcal D([a,b];X)$ with the Skorokhod topology; the latter is Polish under the usual Skorokhod metric, see \cite[Chapter 2]{Mtivier-1988}. We use $L^2(a,b;X)$ for the space of square-integrable $X$-valued functions.

Let $\mathcal M_{FC}(E)$ denote the space of Borel measures $\vartheta$ on $E$ that are finite on compact sets. We equip $\mathcal M_{FC}(E)$ with the weakest topology for which
$$
\vartheta\longmapsto\int_E f(v)\,d\vartheta(v)
$$
is continuous for every $f\in C_c(E)$. With a compatible metric, $\mathcal M_{FC}(E)$ is Polish; see, for example, \cite{Budhiraja-Poinca-2011}.

$\bullet~\textbf{Stochastic basis for the long-time dynamics.}$
For Sections~\ref{PMA-PEs3} and \ref{USC-PEs4}, let
$(\Omega,\mathscr F,\{\mathscr F_t\}_{t\in\mathbb R},\mathbb P)$ be a filtered probability space satisfying the usual conditions. It carries a two-sided cylindrical Wiener process
\[
W(t)=\sum_{k\geq1}r_k w_k(t),\qquad t\in\mathbb R,
\]
where $\{r_k\}_{k\geq1}$ is an orthonormal basis of the separable Hilbert space $U$ and $\{w_k\}_{k\geq1}$ is a family of independent two-sided standard Brownian motions. It also carries, for every $\epsilon>0$, a Poisson random measure $N^{\epsilon^{-1}}$ on $\mathbb R\times E$ with intensity $\epsilon^{-1}dt\,\nu(d\xi)$, and
\[
\widetilde N^{\epsilon^{-1}}(dt,d\xi)
=N^{\epsilon^{-1}}(dt,d\xi)-\epsilon^{-1}dt\nu(d\xi).
\]
The Wiener process and the Poisson random measure are assumed to be independent. All expectations in the long-time and zero-noise sections are taken on this stochastic basis.

$\bullet~\textbf{Canonical controlled Poisson random measures on a finite interval.}$
The following construction is used only in Section~\ref{MDP-PEs4}. Fix $T\in (0,\infty)$, set $E_T=[0,T]\times E$, \( Y=E \times \mathbb{R}^+ \) and \( Y_T := [0, T] \times Y \), and let $\nu_T=\lambda_T\otimes\nu$. Put
$\mathfrak M=\mathcal M_{FC}(E_T)$ and
$\overline{\mathfrak M}=\mathcal M_{FC}(Y_T)$. A Poisson random measure $\mathcal N$ on $E_T$ with intensity $\nu_T$ is an $\mathfrak M$-valued random variable such that, for every $K\in\mathcal B(E_T)$ with $\nu_T(K)<\infty$, the random variable $\mathcal N(K)$ has a Poisson distribution with mean \(\nu_{T}(K)\), and
for pairwise disjoint sets \(K_1, \cdots, K_k \in \mathcal{B}(E_T)\), the random variables \(\mathcal{N}(K_1), \dots, \mathcal{N}(K_k)\) are independent (see e.g. \cite{Ikeda-1981}).
Let \(\mathbb{P}\) denote the unique probability measure induced by \(\mathcal{N}\) on \(\bigl(\mathfrak{M}, \mathcal{B}(\mathfrak{M})\bigr)\) under which the canonical map \(N: \mathfrak{M} \to \mathfrak{M}\) defined by \(N(m) := m\) is a Poisson random measure with intensity measure \(\nu_{T}\).
For $\kappa>0$, let $\mathbb P_\kappa$ denote the probability measure on \((\mathfrak{M}, \mathcal{B}(\mathfrak{M}))\) under which \(N\) is a Poisson random measure with intensity \(\kappa \nu_{T}\). The corresponding expectation operators are denoted by \(\mathbb{E}\) and \(\mathbb{E}_\kappa\), respectively.

Let $\overline{\mathbb P}$ be the unique probability measure on \( (\overline{\mathfrak{M}}, \mathcal{B}(\overline{\mathfrak{M}})) \) under which the canonical map  
\( \overline{N} : \overline{\mathfrak{M}} \to \overline{\mathfrak{M}} \), defined by \( \overline{N}(m) := m \), is a Poisson random measure with intensity measure  
\( \bar{\nu}_T = \lambda_T \otimes \nu \otimes \lambda_\infty \), where \( \lambda_\infty \) is the Lebesgue measure on \(\mathbb{R}^+ \).  
We denote the corresponding expectation by $\overline{\mathbb E}$.  
We denote the corresponding compensated Poisson random measure by $\widetilde N$. For every \(t \in [0, T]\), we define  
\[\mathscr{F}_t := \sigma \left( \{ \widetilde{N}((0, s] \times A) : 0 \leq s \leq t, A \in \mathcal{B}(Y) \} \right),\]
and denote by $\overline{\mathscr{F}}_t$ the completion of \(\mathscr{F}_t\) under \(\overline{\mathbb{P}}\). Let $\overline{\mathfrak P}$ be the predictable $\sigma$-field on $[0,T]\times\overline{\mathfrak M}$ associated with the filtration \(\{\overline{\mathscr{F}}_t\}_{t\in [0,T]}\) on $(\overline{\mathfrak M},\mathcal B(\overline{\mathfrak M}))$. Denote by $\overline{\mathfrak R}_+$ (resp. $\overline{\mathfrak R}$) the class of $(\mathcal B(E)\otimes\overline{\mathfrak P})/\mathcal B(\mathbb R^+)$-measurable mappings $E_T\times\overline{\mathfrak M}\to\mathbb R^+$ (resp. $(\mathcal B(E)\otimes\overline{\mathfrak P})/\mathcal B(\mathbb R)$-measurable mappings $E_T\times\overline{\mathfrak M}\to\mathbb R$). For \(\phi \in \overline{\mathfrak{R}}_+\), define a counting process \(N^\phi\) on \(E_T\) by  
\begin{align}\label{Possine}
	N^\phi((0, t] \times \mho) := \int_{[0, t] \times \mho \times \mathbb{R}^+} 1_{[0, \phi(s, x)]}(r) \, \overline{N}(ds, dx, dr), \quad t\in [0,T], ~~ \mho \in \mathcal{B}(E).
\end{align}
The random measure $N^\phi$ has predictable compensator $\phi(s,x)ds\nu(dx)$.  Its compensated version is therefore defined by
\[
\widetilde N^\phi(ds,dx):=N^\phi(ds,dx)-\phi(s,x)ds\nu(dx),
\]
in the sense of integration against predictable integrands.  This definition, rather than a second thinning of the compensated canonical measure, is used throughout Section~\ref{MDP-PEs4}.  When $\phi\equiv\epsilon^{-1}$, we write $N^{\epsilon^{-1}}$ and $\widetilde N^{\epsilon^{-1}}$.  Then $N^{\epsilon^{-1}}$ is a Poisson random measure on $E_T$ with intensity $\epsilon^{-1}\nu_T$ and
\[
\widetilde N^{\epsilon^{-1}}((0,t]\times\mho)
=N^{\epsilon^{-1}}((0,t]\times\mho)-\epsilon^{-1}t\nu(\mho),
\qquad \nu(\mho)<\infty.
\]
Moreover, $N^{\epsilon^{-1}}$ under $\overline{\mathbb P}$ has the same law as the canonical measure $N$ under $\mathbb P_{\epsilon^{-1}}$.

For a Banach space $X$, $L^2(\Omega;X)$ denotes the space of strongly measurable square-integrable $X$-valued random variables, equipped with the norm
\[
\|\phi\|_{L^2(\Omega;X)}=\left(\mathbb E\|\phi(\omega)\|_X^2\right)^{1/2}.
\]

Throughout the paper, we impose the following standing assumptions on the nonlinear terms $\sigma$ and $h$.
\begin{hypothesis}\label{PEeqAssum-2.2}
Suppose that $\sigma:\mathbb R\times H\to\mathcal L_2(U;H)$ is Borel measurable and, for every $t\in\mathbb R$, satisfies
	\begin{description}
		\item[(A.1)] $\|\sigma(t,v)\|^2_{\mathcal{L}_2({U}; H)}\leq C_{\sigma}(1+|v|^2),\quad \forall v \in {H}$
		\item[(A.2)] $\|\sigma(t,v_1)-\sigma(t,v_2)\|^2_{\mathcal{L}_2({U}; H)}\leq L_{\sigma}|v_1-v_2|^2,\quad \forall v_1, v_2 \in H$
		\item[(A.3)] for every $v\in\mathcal H$, $\sigma(t,v)\in\mathcal L_2(U;\mathcal H)$ and $\|\partial_z\sigma(t,v)\|^2_{\mathcal L_2(U;H)}\leq\widetilde C_\sigma(1+|\partial_zv|^2)$, where $[\partial_z\sigma(t,v)]u:=\partial_z[\sigma(t,v)u]$ for $u\in U$.
	\end{description}
where $C_{\sigma}$, $L_{\sigma}$, and $\widetilde{C}_{\sigma}$ are positive constants.
\end{hypothesis}

%{\color{red}
Let \( L^2(\nu_T) \) be the space of all \( (\mathcal{B}([0, T]) \otimes \mathcal{B}(E)) / \mathcal{B}(\mathbb{R}) \)-measurable functions \( q \) satisfying 
\[\|q\|_{L^2(\nu_T)}^2 := \int_0^T \int_{E} |q(t, \xi)|^2 \nu(d\xi ) dt < +\infty.\]
Then \( (L^2(\nu_T), \|\cdot\|_{L^2(\nu_T)}) \) is a Hilbert space. For $p\geq2$, $L^p(\nu_T)$ is defined analogously.
For $r>0$, let $\mathbb B_\nu(r)$ be the closed ball of radius \( r \) centered at zero in \( L^2(\nu_T) \). We always endow $\mathbb B_\nu(r)$ with the weak topology of $L^2(\nu_T)$; hence it is weakly compact.
%}
\begin{hypothesis}\label{PEeqAssum-2.3}
		Suppose that $h: H \times E\rightarrow H$ is measurable and satisfies
			 \begin{description}
			 	
				\item[(B.1)] $|h(v,\xi)|\leq C_{h}(\xi)(1+|v|),\quad \forall v \in {H}, \xi \in E,$
				
				\item[(B.2)] $|h( v_1,\xi)-h(v_2,\xi)| \leq L_{h}(\xi)|v_1-v_2|,\quad \forall v_1, v_2 \in H, \xi \in E,$
				
				\item[(B.3)] for every $v\in\mathcal H$, $h(v,\xi)\in\mathcal H$ and $|\partial_z h(v,\xi)|\leq\widetilde C_h(\xi)(1+|\partial_zv|)$ for $\nu$-a.e.\ $\xi\in E$.
			\end{description}
	where $C_h,\widetilde C_h\in L^2(\nu)\cap L^4(\nu)$ and $L_h\in L^2(\nu)$. The fourth-order integrability of the growth coefficients is used in the pullback fourth-moment estimates and in the Burkholder-Davis-Gundy (BDG) estimates.
%	where $C_{h}\in L^{p}(\nu)$, $L_{h}\in L^2(\nu)$ and $\widetilde{C}_{h} \in L^{p}(\nu)$ $(p\geq 2)$.
\end{hypothesis}

\subsection{Large deviation principle and compactness theorem}
We recall the definition of an LDP. Let $\mathfrak U$ be a Polish space.
\begin{definition}[Rate function]
A function $I:\mathfrak U\to[0,\infty]$ is called a \emph{rate function} if it is lower semicontinuous. It is called a \emph{good rate function} if, in addition, every level set
\[
\{x\in\mathfrak U:I(x)\le M\},\qquad M<\infty,
\]
is compact in $\mathfrak U$.
\end{definition}

\begin{definition}[Large deviation principle]
	 Let \( I \) be a good rate function on \( \mathfrak{U} \). Given a collection \( \{b(\epsilon)\}_{\epsilon > 0} \) of positive reals, a family \( \{\mathcal{Y}^\epsilon\}_{\epsilon > 0} \) of \( \mathfrak{U} \)-valued random elements is said to satisfy an LDP on \( \mathfrak{U} \) with logarithmic scale \( b(\epsilon) \) (equivalently, with conventional speed \(1/b(\epsilon)\)) and rate function \( I \) if the following two bounds hold:
	 
	 $(i)$ \textit{(Upper bound)} For each closed subset \( F \) of \( \mathfrak{U} \),
	 
	 \[
	 \limsup_{\epsilon \to 0} b(\epsilon) \log \mathbb{P}(\mathcal{Y}^\epsilon \in F) \leq - \inf_{x \in F} I(x).
	 \]
	 
	 $(ii)$ \textit{(Lower bound)} For each open subset \( O \) of \( \mathfrak{U} \),
	 
	 \[
	 \liminf_{\epsilon \to 0} b(\epsilon) \log \mathbb{P}(\mathcal{Y}^\epsilon \in O) \geq - \inf_{x \in O} I(x).
	 \]
	
\end{definition}

We use the following compactness criterion, which is a variant of the results in \cite[Chapter I, Section 5]{Lions69} and \cite[Section 13.3]{Temam79}; see also \cite[Theorem 2.1]{Flandoli-Gatarek-95}. 
%Let \(\mathfrak{X} \) be a Banach space with norm \( \|\cdot\|_{\mathfrak{X}} \). 
Given \( p > 1 \), \( \alpha \in (0, 1) \), let \( W^{\alpha, p}([0,T]; Y) \) be the Sobolev space of all \( v \in L^p([0,T]; Y) \) such that
\[
\int_0^T \int_0^T \frac{\|v(t) - v(s)\|_{Y}^p}{|t - s|^{1+\alpha p}} dt ds < \infty,
\]
endowed with the norm
\[
\|v\|_{W^{\alpha, p}([0,T]; Y)}^p = \int_0^T \|v(t)\|_{Y}^p  dt + \int_0^T \int_0^T \frac{\|v(t) - v(s)\|_{Y}^p}{|t - s|^{1+\alpha p}}  dt ds.
\]

\begin{lemma}\label{PEeq-MDPCpctness-4.1}
	Let \( B_0 \subset B \subset B_1 \) be Banach spaces, with \(B_0\) and \(B_1\) reflexive and \(B_0\) compactly embedded in \(B\). Let \( p \in (1, \infty) \) and \( \alpha \in (0, 1) \) be given. Let \( \mathfrak X \) be the space
	\[
	\mathfrak{X} = L^p([0,T]; B_0) \cap W^{\alpha, p}([0,T]; B_1)
	\]
	endowed with the natural norm. Then the embedding of \( \mathfrak{X} \) in \( L^p([0,T]; B) \) is compact.
\end{lemma}

\subsection{Existence and uniqueness of solutions}
We use the following notion of solution for \eqref{PEeq-2.15}; see \cite{SunGao-2013} for the detailed proof.

\begin{definition}\label{PEeqDef-2.4}
 An $H$-valued $\mathscr{F}_t$-adapted c\`{a}dl\`{a}g  stochastic process  $v$ is called a solution of (\ref{PEeq-2.15}) if the following conditions hold: 
$$
v  \in \mathcal{D}([\tau,{\tau+T}];H)\cap  L^2(\tau,{\tau+T};V),~~ \mathbb{P}\text{-a.s.;}
$$
and, for every $t\in [\tau,{\tau+T}]$ and $\psi\in D(A)$, 
\begin{align*}
	&~\langle v(t),\psi\rangle+\int^{t}_{\tau}\Big(\langle v(s), A \psi\rangle+\langle B(v(s),v(s)), \psi\rangle \Big)ds\\
	&=\langle \zeta,\psi\rangle+\int^{t}_{\tau}\langle g(s),\psi\rangle ds+\sqrt{\epsilon}\int^{t}_{\tau}\langle \sigma(s,v(s)) dW(s),\psi\rangle+\epsilon\int^{t}_{\tau}\int_{E}\langle h(v(s-),\xi),\psi\rangle\widetilde{N}^{\epsilon^{-1}}(ds,d\xi), ~~ \mathbb{P}\text{-a.s.}
\end{align*}
\end{definition}

We recall the following global well-posedness result, with further details provided in \cite[Theorem~3.2]{SunGao-2013} (or \cite[Theorem~2.3]{ZRR-2021}).
\begin{theorem}\label{PEeqthe-2.5}
Suppose Hypotheses \ref{PEeqAssum-2.2} and \ref{PEeqAssum-2.3} hold. Let $\zeta$ be an $\mathscr F_\tau$-measurable, $\mathcal H$-valued random variable such that
$\mathbb E\left[|\zeta|^4+|\partial_z\zeta|^4\right]<\infty$, and assume that
$g,\partial_zg\in L^4([\tau,\tau+T];H)$.
Then problem \eqref{PEeq-2.15} admits a unique global solution $v$ in the sense of Definition  \ref{PEeqDef-2.4}.
 Furthermore,  
 there exist $\epsilon_0\in (0,1)$ and a positive constant $C:=C(T,\|g\|_{L^4([\tau,{\tau+T}];H)},\|\partial_{z}g\|_{L^4([\tau,{\tau+T}];H)})$ such that for every $\epsilon\in (0,\epsilon_0)$,
\begin{align}\label{ee-5}
	\mathbb{E}\left[\sup_{\tau\leq t\leq \tau+T}|v(t)|^{2p}+\left(\int^{\tau+T}_{\tau}\|v(t)\|^{2}dt\right)^{p}\right]\leq C\left(1+\mathbb{E}\left[|\zeta|^{2p}\right]\right),
\end{align}
and
\begin{eqnarray}\label{ee-6}
	\mathbb{E}\left[\sup_{\tau\leq t\leq \tau+T}|\partial_zv(t)|^{2p}+\left(\int^{\tau+T}_{\tau}\|\partial_zv(t)\|^{2}dt\right)^{p}\right]\leq C\left(1+\mathbb{E}\left[|\partial_z\zeta|^{2p}\right]\right),
\end{eqnarray}	
where $p\in \{1,2\}$. 
\end{theorem}

\begin{remark}
	In the pullback and zero-noise sections, Theorem~\ref{PEeqthe-2.5} is applied to $\mathcal H$-valued initial variables satisfying
	$\mathbb E\left[|\zeta|^4+|\partial_z\zeta|^4\right]<\infty$, with $0<\epsilon<\epsilon_0$. The resulting evolution is therefore considered on $\mathcal P_{4,z}(H)$.
\end{remark}

\begin{remark}
	Hypotheses~\ref{PEeqAssum-2.2} and \ref{PEeqAssum-2.3} impose global Lipschitz conditions on the noise coefficients $\sigma$ and $h$. The same arguments can be extended to locally Lipschitz coefficients by truncation; see \cite{Bre-JEMS-2023}.
\end{remark}

\section{Pullback measure attractors}\label{PMA-PEs3}
We now establish the existence and uniqueness of pullback measure attractors for \eqref{PEeq-2.15}. We first introduce the measure-theoretic notation used below.
Let $X$ be a separable Banach space with norm $\|\cdot\|_X$, and let $C_b(X)$ be the space of all bounded continuous real-valued functions on $X$, equipped with the supremum norm $\|\phi\|_{C_b} = \sup_{x \in X} |\phi(x)|$. 
Denote by $L_b(X)$ the space of bounded Lipschitz functions on $X$, i.e., those $\phi \in C_b(X)$ satisfying
\begin{align*}
	\text{Lip}(\phi) := \sup_{\substack{x_1, x_2 \in X \\ x_1 \neq x_2}} \frac{|\phi(x_1) - \phi(x_2)|}{\|x_1 - x_2\|_X} < \infty.
\end{align*}
The space $L_b(X)$ is equipped with the norm $\|\phi\|_{L_b} = \|\phi\|_{C_b} + \text{Lip}(\phi)$. Let $\mathcal{P}(X)$ denote the set of probability measures on $(X,\mathcal{B}(X))$, and define the metric on $\mathcal{P}(X)$ by
\begin{align*}
	d_{\mathcal{P}(X)}(\mu_1, \mu_2) = \sup_{\substack{\phi \in L_b(X) \\ \|\phi\|_{L_b} \leq 1}} |(\phi, \mu_1) - (\phi, \mu_2)|, \quad \forall \mu_1, \mu_2 \in \mathcal{P}(X),
\end{align*}
where $(\phi, \mu) = \int_X \phi(x) \, \mu({d}x)$ for any $\mu \in \mathcal{P}(X)$ and $\phi \in C_b(X)$.
Recall that a sequence of probability measures $\{\mu_n\}_{n=1}^\infty \subset \mathcal{P}(X)$ is weakly convergent to $\mu \in \mathcal{P}(X)$ if for every $\phi \in C_b(X)$,
\begin{align*}
	\lim_{n \to \infty} (\phi, \mu_n) = (\phi, \mu).
\end{align*}
Note that $(\mathcal{P}(X), d_{\mathcal{P}(X)})$ forms a Polish space. In particular, a sequence $\{\mu_n\}_{n=1}^\infty \subset \mathcal{P}(X)$ converges to $\mu$ in $(\mathcal{P}(X), d_{\mathcal{P}(X)})$ if and only if $\{\mu_n\}_{n=1}^\infty$ converges weakly to $\mu$.

For every $p \geq 1$, we denote the subspace $\mathcal{P}_p(X)$ of $\mathcal{P}(X)$ by
\begin{align*}
	\mathcal{P}_p(X) = \left\{ \mu \in \mathcal{P}(X) : \int_X \|x\|_X^p \, \mu({d}x) < \infty \right\}.
\end{align*}
Then $(\mathcal{P}_p(X), d_{\mathcal{P}(X)})$ is a metric space but is not complete in general, because weak convergence alone does not preserve finiteness of the $p$-th moment. In the present paper we therefore regard $\mathcal P(H)$, endowed with $d_{\mathcal P(H)}$, as the ambient Polish space. The evolution is defined only on the admissible law class $\mathcal P_{4,z}(H)$ introduced below. The moment lower semicontinuity lemma and the closed absorbing family will ensure that every weak limit occurring in the pullback construction remains admissible. Thus no completeness of $\mathcal P_{4,z}(H)$ is assumed.

A  subset $\mathcal{S}\subseteq {\mathcal P}_p \left(X\right)$ is
bounded  if there exists $r>0$ such that
$\mathcal{S}\subseteq \mathbb{B}_{\mathcal{P}_{p}(X)}(r)$. If $\mathcal{S}$ is bounded in ${\mathcal P}_p(X)$,
then we set
$$
\|\mathcal{S}\|_{{\mathcal P}_p (X)}
=
\sup_{
	{\mu} \in \mathcal{S}
} \left(\int_{X} \|x\|_{X}^p {\mu}
( {dx}  )\right)^{{1}/{p}}<\infty.
$$
Given $r > 0$, we define the closed ball $\mathbb{B}_{\mathcal{P}_p(X)}(r)$ as follows:
\begin{align*}
	\mathbb{B}_{\mathcal{P}_p(X)}(r) = \left\{ \mu \in \mathcal{P}_p(X) : \left( \int_X \|x\|_X^p \, \mu({d}x) \right)^{1/p} \leq r \right\}.
\end{align*}
For any nonempty subsets $Y,Z\subset\mathcal P(X)$, we use the one-sided Hausdorff semidistance generated by weak convergence on $X$,
\begin{align}\label{weak-Hausdorff-distance}
	\operatorname{dist}_{w,X}(Y,Z)
	:=\sup_{\mu\in Y}\inf_{\nu\in Z}d_{\mathcal P(X)}(\mu,\nu).
\end{align}
For a point $\mu\in\mathcal P(X)$ we also write
\[
d_{\mathcal P(X)}(\mu,Z):=\inf_{\nu\in Z}d_{\mathcal P(X)}(\mu,\nu).
\]
The subscript $w$ emphasizes that this is a weak measure semidistance; it is not a Wasserstein distance and does not require convergence of moments.

Observe that for every $r>0$, $\mathbb{B}_{\mathcal{P}_{p}(X)}(r)$ is a closed subset of $\mathcal{P}(X)$ with respect to the metric $d_{\mathcal{P}(X)}$. Consequently, the space $\left(\mathbb{B}_{\mathcal{P}_{p}(X)}(r),d_{\mathcal{P}(X)}\right)$ is complete for every $r>0$.

\begin{hypothesis}\label{PEeqAssum-2.4}
	For every $t\in\mathbb R$ and $v\in V$, assume that
	$\sigma(t,v)\in\mathcal L_2(U;V)$ and that $h(v,\xi)\in V$ for $\nu$-a.e.\ $\xi\in E$.
	Moreover, suppose that there exists a constant $\alpha_1>0$ such that
	\begin{align*}
		\|\sigma(t,v)\|_{\mathcal L_2(U;V)}^2
		+\int_E\|h(v,\xi)\|_V^2\nu(d\xi)
		\le\alpha_1(1+\|v\|^2),
		\qquad t\in\mathbb R,\ v\in V.
	\end{align*}
\end{hypothesis}

\subsection{Uniform pullback moment estimates of solutions}\label{UL-D-unver3.1}
In this subsection, we establish uniform pullback moment estimates for solutions of problem \eqref{PEeq-2.15}; these estimates will facilitate the subsequent proof of the existence and uniqueness of pullback measure attractors. We write $\mathscr L(\vartheta)$ for the law of a random variable $\vartheta$.
Assume that there exists a constant $\kappa>0$ such that
\begin{align}\label{LSWs4.1}
	\frac{\textbf{c}_0}{8}-\kappa\geq 0,
\end{align}
and for this $\kappa$, we assume that for any $\tau \in \mathbb{R}$,
\begin{align}\label{PETNs-3.01}
	\int_{-\infty}^{\tau} e^{\kappa (s-\tau)}(|g(s)|^{4}+|\partial_{z}g(s)|^{4})ds<\infty.
\end{align}

% Let $\mathfrak{D}$ be the collection of families of bounded nonempty subsets of $\mathcal{P}_4(H)$, which is given by
%\begin{align*}
%	\mathfrak{D}=\big\{&D=\left\{D(\tau)\subseteq \mathcal{P}_4(H): D(\tau) \text{ is a bounded nonempty subset of } \mathcal{P}_4(H), \tau\in \mathbb{R}\right\}:\\
%	&\lim_{\tau\rightarrow -\infty} e^{\kappa \tau}\|D(\tau)\|^4_{\mathcal{P}_4(H)}=0\big\}.
%\end{align*} 
Define the admissible class
\[
\mathcal P_{4,z}(H):=
\left\{\mu\in\mathcal P_4(H):
\mu(\mathcal H)=1,~~
\int_H|\partial_z\zeta|^4\,\mu(d\zeta)<\infty
\right\}.
\]
Here and below the functional $\zeta\mapsto|\partial_z\zeta|^4$ is understood to be $+\infty$ outside $\mathcal H$. Let $\mathfrak D$ denote the collection of all families
$D=\{D(\tau):\tau\in\mathbb R\}$ of bounded nonempty subsets of $\mathcal P_{4,z}(H)$ such that
\[
\lim_{\tau\to-\infty}e^{\kappa\tau}
\sup_{\mu\in D(\tau)}
\int_H|\zeta|^4\,\mu(d\zeta)=0
\]
and
\[
\lim_{\tau\to-\infty}e^{\kappa\tau}
\sup_{\mu\in D(\tau)}
\int_H|\partial_z\zeta|^4\,\mu(d\zeta)=0.
\]
Thus the admissible universe contains precisely the moment information that is required by Theorem~\ref{PEeqthe-2.5} and by the pullback vertical derivative estimates.

Following the standard definition of a $\mathfrak D$-pullback
measure attractor in \cite[Definition~2.6]{LDS-JDE-2024},
we use the following adapted formulation. The evolution is defined
on the admissible class $\mathcal P_{4,z}(H)$, whereas compactness
and attraction are understood in the ambient weak topology of $\mathcal P(H)$.
	
	\begin{definition}\label{def-admissible-pullback-attractor}
		Let
		\[
		S(t,\tau):\mathcal P_{4,z}(H)\longrightarrow
		\mathcal P_{4,z}(H),\qquad t\geq\tau,
		\]
		be a measure evolution process satisfying
		\[
		S(\tau,\tau)=I,\qquad
		S(t,s)S(s,\tau)=S(t,\tau),
		\quad t\geq s\geq\tau.
		\]
		A family
		$\mathscr A=\{\mathscr A(\tau):\tau\in\mathbb R\}
		\in\mathfrak D$
		is called a $\mathfrak D$-pullback measure attractor for $S$ if:
		\begin{enumerate}
			\item[(1)] for every $\tau\in\mathbb R$, the set
			$\mathscr A(\tau)$ is nonempty and compact in
			$(\mathcal P(H),d_{\mathcal P(H)})$;
			
			\item[(2)] $\mathscr A$ is invariant, namely,
			\[
			S(t,\tau)\mathscr A(\tau)=\mathscr A(t),
			\qquad t\geq\tau;
			\]
			
			\item[(3)] $\mathscr A$ pullback attracts every $D\in\mathfrak D$,
			that is, for every $\tau\in\mathbb R$,
			\[
			\lim_{r\to\infty}
			\operatorname{dist}_{w,H}
			\bigl(
			S(\tau,\tau-r)D(\tau-r),
			\mathscr A(\tau)
			\bigr)=0.
			\]
		\end{enumerate}
		Here compactness and attraction are understood in the weak topology
		of probability measures on $H$, while membership in $\mathfrak D$
		retains the fourth-order and vertical-derivative moments required
		to restart the equation.
	\end{definition}

We derive the long-time fourth-moment estimates for problem \eqref{PEeq-2.15}.
\begin{lemma}\label{lemPETNs-3.1}
	Suppose Hypotheses \ref{PEeqAssum-2.2}, \ref{PEeqAssum-2.3} and \eqref{PETNs-3.01} hold. Then there exists $\epsilon_0\in(0,1)$ such that, for every $\epsilon\in(0,\epsilon_0)$, $\tau\in\mathbb R$, and $\zeta\in L^4(\Omega,\mathscr F_\tau;\mathcal H)$ with $\mathbb E[|\partial_z\zeta|^4]<\infty$, the solution of \eqref{PEeq-2.15} satisfies 
	\begin{align}\label{PETNs-3.2}
		\begin{split}
			&~~\mathbb{E}\left[|v(t,\tau,\zeta)|^{4}\right]+\int_{\tau}^t e^{\kappa (s-t)}\mathbb{E}\left[|v(s,\tau,\zeta)|^{2}\|v(s,\tau,\zeta)\|^2\right]ds\\
			&\leq \rho_1 \left(e^{-\kappa(t-\tau)}\mathbb{E}\left[|\zeta|^{4}\right]+ \int_\tau^te^{-\kappa(t-s)}|g(s)|^{4}ds+1\right),
		\end{split}
	\end{align}
	where $\rho_1$ is a positive constant independent of $\zeta$, $\tau$ and $t$.
\end{lemma}  
\begin{proof}
	Write $v(t):=v(t,\tau,\zeta)$ for the solution of \eqref{PEeq-2.15} with initial data $\zeta$.
	For $p\in\{2,4\}$, applying It\^{o}'s formula with jumps to $e^{\kappa t}|v(t)|^p$ and using the cancellation $b(v,v,v)=0$ gives
%	Applying It\^{o}'s formula to the process $e^{\kappa t}|v(t)|^p$, we obtain from \eqref{PEeq-2.15} that for any $p\geq 2$,
	\begin{align}\label{PETNs-3.3}
		&\,\, e^{\kappa t}|v(t)|^{p}+\int_{\tau}^{t}\left(pe^{\kappa s}|v(s)|^{p-2}\|v(s)\|^2-\kappa e^{\kappa s} |v(s)|^p\right)ds\notag \\
		&=e^{\kappa \tau}|\zeta|^{p}+p\int_{\tau}^{t}e^{\kappa s}|v(s)|^{p-2}\langle g(s),v(s)\rangle ds+p\sqrt{\epsilon} \int_{\tau}^{t}e^{\kappa s}|v(s)|^{p-2}\langle v(s),\sigma(s,v(s))\rangle dW(s)\notag \\
		&+\frac{p}{2}\epsilon \int_{\tau}^{t}e^{\kappa s} |v(s)|^{p-2}\|\sigma(s,v(s))\|^2_{\mathcal{L}_2(U;H)}ds
		+\frac{p(p-2)}{2}\epsilon \int_{\tau}^{t}e^{\kappa s} |v(s)|^{p-4}\langle v(s),\sigma (s,v(s))\rangle^2ds \notag \\
		&+p{\epsilon}\int_{\tau}^{t}e^{\kappa s}\int_{E}|v(s)|^{p-2}\langle v(s-),h(v(s-),\xi)\rangle\widetilde{N}^{\epsilon^{-1}}(ds,d\xi)\notag \\
		&+\int_{\tau}^{t}e^{\kappa s} \int_{E}\left(|v(s-)+{\epsilon}h(v(s-),\xi)|^{p}-|v(s-)|^p-p{\epsilon}\langle v(s-),h(v(s-),\xi)\rangle\right)N^{\epsilon^{-1}}(ds,d\xi). 
	\end{align}
%	where we used
%	the equality \eqref{PEeq-2.13}.
	By H\"{o}lder's inequality we get
	\begin{align}\label{PETNs-3.4}
		p\int_{\tau}^{t}e^{\kappa s}|v(s)|^{p-2}\langle g(s),v(s)\rangle ds\leq \frac{p\textbf{c}_0}{4}\int_{\tau}^{t}e^{\kappa s}|v(s)|^pds+\left(\frac{4(p-1)}{\textbf{c}_0p}\right)^{p-1}\int_{\tau}^{t}e^{\kappa s}|g(s)|^{p}ds.
	\end{align}
	Combining \eqref{PETNs-3.3} and \eqref{PETNs-3.4} and taking the expectation, we obtain
	\begin{align}\label{PETNs-3.6}
		&~~e^{\kappa t}\mathbb{E}\left[|v(t)|^{p}\right]+\frac{p}{2}\int_{\tau}^{t}e^{\kappa s}\mathbb{E}\left[|v(s)|^{p-2}\|v(s)\|^2\right]ds+(\frac{p\textbf{c}_0}{4}-\kappa)\int_{\tau}^{t}e^{\kappa s}\mathbb{E}\left[|v(s)|^{p}\right]ds \leq e^{\kappa \tau}\mathbb{E}\left[|\zeta|^{p}\right]\notag\\
		&+\left(\frac{4(p-1)}{\textbf{c}_0p}\right)^{p-1}\int_{\tau}^{t}e^{\kappa s}|g(s)|^{p}ds+\underbrace{\frac{p(p-1)}{2}\epsilon \int_{\tau}^{t}e^{\kappa s} \mathbb{E}\left[|v(s)|^{p-2}\|\sigma(s,v(s))\|^2_{\mathcal{L}_2(U,H)}\right]ds}_{\mathscr{J}_1(t)}\notag\\
		&+\underbrace{\mathbb{E}\left[\int_{\tau}^{t}e^{\kappa s} \int_{E}\left(|v(s-)+{\epsilon}h(v(s-),\xi)|^{p}-|v(s-)|^p-p{\epsilon}\langle v(s-),h(v(s-),\xi)\rangle\right)N^{\epsilon^{-1}}(ds,d\xi)\right]}_{\mathscr{J}_2(t)}.
	\end{align}
	
	Taylor's formula gives a constant $\wp_p>0$ such that for $p\geq 2$,
	\begin{align}\label{LSWs3.meanin24}
		\left||x+y|^{p}-|x|^{p}-{p}|x|^{{p}-2}\langle x,y\rangle\right|\leq \wp_p\left(|x|^{p-2}|y|^2+|y|^{p}\right), \quad \forall x,y \in H.
	\end{align}
	Since $\epsilon\in(0,1)$, the last term on the right-hand side of \eqref{PETNs-3.6} satisfies
	\begin{align*}
		\mathscr{J}_2(t)&\leq \wp_p\int_{\tau}^{t}e^{\kappa s}\mathbb{E}\left[\int_{E} \left({\epsilon}^2|v(s-)|^{p-2}|h(v(s-),\xi)|^{2}+{\epsilon}^p|h(v(s-),\xi)|^{p}\right) \epsilon^{-1}\nu(d\xi)\right] ds \\
		&\leq \underbrace{\wp_p {\epsilon}\int_{\tau}^{t}e^{\kappa s}\mathbb{E}\left[\int_{E} \left(|v(s-)|^{p-2}|h(v(s-),\xi)|^{2}+|h(v(s-),\xi)|^{p}\right) \nu(d\xi)\right] ds}_{\mathscr{J}_3(t)}.
	\end{align*}
	Define the constants
	\begin{align}\label{ep-0_0}
	\begin{split}
		C_p^{\text{noise}}:=2^{p-1} \wp_p \int_E |C_h(\xi)|^p &\nu(d\xi) + \bigl(p(p-1) + 2\wp_p\bigr) \left( C_\sigma + 2 \int_E |C_h(\xi)|^2 \nu(d\xi) \right), \quad p\in \{2,4\},\\
		&\epsilon_{0}^1:=\min_{p\in \{2,4\}} \left\{1,{\frac{p\textbf{c}_0}{8C_p^{\text{noise}}}}\right\}>0.
	\end{split}
\end{align} 	
	For $p\in\{2,4\}$, we have
	\[
	|v|^{p-2}(1+|v|^2)\le2(1+|v|^p),\qquad
	|v|^{p-2}(1+|v|)^2\le4(1+|v|^p),
	\]
	and $(1+|v|)^p\le2^{p-1}(1+|v|^p)$. Hence, Hypotheses \textbf{(A.1)} and \textbf{(B.1}), together with the definition of $C_p^{\rm noise}$, imply that for $\epsilon\in(0,\epsilon_{0}^1)$
%	By the hyptothesis \textbf{(A.1)} and \textbf{(B.1)}, we obatin that there exists 
%	\begin{align}\label{ep-0_0}
%		\epsilon_{0}^1=\min \left\{1,{\frac{p\textbf{c}_0}{8\left(2^{p-1} \wp_p \int_E |C_h(\xi)|^p \nu(d\xi) + \bigl(p(p-1) + 2\wp_p\bigr) \left( C_\sigma + 2 \int_E |C_h(\xi)|^2 \nu(d\xi) \right)\right)}}\right\}>0
%	\end{align} 
%	such that for ${\epsilon}\in (0,\epsilon_{0}^1)$,
	\begin{align}\label{PETNs-300.7}
		&~\mathscr{J}_1(t)+\mathscr{J}_3(t)\leq \wp_p\epsilon\int_{\tau}^{t}e^{\kappa s} \mathbb{E}\left[\int_{E}|h(v(s-),\xi)|^{p}\nu(d\xi)\right]ds+ \left(\frac{p(p-1)}{2}+\wp_p\right)\epsilon\\ \notag
		& \cdot\int_{\tau}^{t}e^{\kappa s} \mathbb{E}\left[|v(s)|^{p-2}\left(\|\sigma(s,v(s))\|^2_{\mathcal{L}_2(U;H)}+\int_{E}|h(v(s-),\xi)|^{2}\nu(d\xi)\right)\right]ds\\ \notag
		&\leq 2^{p-1}\wp_p\epsilon\int_{E} |C_h(\xi)|^p\nu(d\xi)\int_{\tau}^{t}e^{\kappa s}\mathbb{E}\left[(1+|v(s)|^{p})\right]ds\\ \notag
		&+\left(\frac{p(p-1)}{2}+\wp_p\right)\epsilon\left(C_{\sigma}+2\int_{E} |C_h(\xi)|^2\nu(d\xi)\right) \int_{\tau}^{t}e^{\kappa s} \mathbb{E}\left[|v(s)|^{p-2}+|v(s)|^{p}\right]ds  \\ \notag
		&\leq 2^{p-1}\wp_p\epsilon\int_{E} |C_h(\xi)|^p\nu(d\xi)\int_{\tau}^{t}e^{\kappa s}\mathbb{E}\left[(1+|v(s)|^{p})\right]ds\\ \notag
		&+\left(\frac{p(p-1)}{2}+\wp_p\right)\epsilon\left(C_{\sigma}+2\int_{E} |C_h(\xi)|^2\nu(d\xi)\right) \int_{\tau}^{t}e^{\kappa s} \mathbb{E}\left[|v(s)|^{p-2}+|v(s)|^{p}\right]ds  \\ \notag
		&\leq \epsilon_0^1 C_p^{\text{noise}} \int_{\tau}^{t}e^{\kappa s} \mathbb{E}\left[(1+|v(s)|^{p})\right]ds
		\leq \frac{p\textbf{c}_0}{8}\int_{\tau}^{t}e^{\kappa s} \mathbb{E}\left[|v(s)|^{p}\right]ds+\frac{p\textbf{c}_0}{8\kappa}e^{\kappa t},
	\end{align}
	Combining this estimate with \eqref{PETNs-3.6} and using \eqref{LSWs4.1}, we may take $\epsilon_{0}:=\epsilon_{0}^1$ so that, for 	${\epsilon}\in (0,\epsilon_0)$,
	\begin{align}\label{PETNs-3.7}
		\begin{split}
			&~~\mathbb{E}\left[|v(t)|^{p}\right]+\frac{p}{2}\int_\tau^t e^{\kappa (s-t)}	\mathbb{E}\left[|v(s)|^{p-2}\|v(s)\|^2\right]ds\\
			&\leq e^{-\kappa (t-\tau)}\mathbb{E}\left[|\zeta|^{p}\right]
			+\left(\frac{4(p-1)}{\textbf{c}_0p}\right)^{p-1}\int_{\tau}^{t}e^{\kappa (s-t)}|g(s)|^{p}ds
			+\frac{p\textbf{c}_0}{8\kappa}.
		\end{split}
	\end{align}
	Taking $p=4$ in \eqref{PETNs-3.7} yields the desired estimate. This completes the proof.
\end{proof}

\begin{lemma}\label{UPe4-lem4.1}
	Suppose Hypotheses \ref{PEeqAssum-2.2}, \ref{PEeqAssum-2.3} and \eqref{PETNs-3.01} hold. Then, for every $\tau \in \mathbb{R}$ and $D = \{D(t), t \in \mathbb{R}\} \in \mathfrak{D}$, there exist ${T} := {T}(\tau, D) > 1$ and  $\epsilon_{0}\in (0,1]$ such that for all $\epsilon\in (0,\epsilon_{0})$ and $t \geq {T}$, the solution $v$ of \eqref{PEeq-2.15} satisfies
	\begin{align}\label{UPe4-4.2}
		&~~\mathbb{E}\left[|v(\tau, \tau - t, \zeta)|^4\right] + \int_{\tau - t}^\tau e^{-\kappa(\tau - s)}\mathbb{E}\left[|v(s, \tau - t, \zeta)|^2\|v(s, \tau - t, \zeta)\|^2\right]{d}s\\ \notag
		&\leq \mathcal{R}_1 + \mathcal{R}_1\int_{-\infty}^\tau e^{-\kappa(\tau - s)}|g(s)|^{4} {d}s,
	\end{align}
	and
	\begin{align}\label{UPe4-4.3}
		\begin{split}
			\int_{\tau - 1}^\tau \mathbb{E}\left[\|v(s, \tau - t, \zeta)\|^2\right]{d}s\leq \mathcal{R}_2 + \mathcal{R}_2\int_{-\infty}^\tau e^{-\kappa(\tau - s)}|g(s)|^{4} {d}s,
		\end{split}
	\end{align}
	where $\zeta\in L^4(\Omega,\mathscr F_{\tau-t};\mathcal H)$ with $\mathscr{L}(\zeta) \in D(\tau - t)$, and $\mathcal{R}_i$ $(i=1,2)$ are positive constants independent of $\tau$, ${\epsilon}$ and $D$.
\end{lemma}
\begin{proof}
	Estimate \eqref{UPe4-4.2} follows directly from Lemma~\ref{lemPETNs-3.1} by taking $\mathcal{R}_1:=2\rho_1$. Moreover, \eqref{PETNs-3.7} gives
	\begin{align}\label{PETNs-3.02}
		\begin{split}
			\int_{\tau-t}^{\tau} e^{-\kappa (\tau-s)}\mathbb{E}\left[\|v(s, \tau - t, \zeta)\|^2\right]ds\leq \rho_2 \left(e^{-\kappa t}\mathbb{E}\left[|\zeta|^{2}\right]+ \int_{-\infty}^\tau e^{-\kappa(\tau-s)}|g(s)|^{2}ds+1\right),
		\end{split}
	\end{align}
 where $\rho_2:=\max\left\{1,\frac{2}{\textbf{c}_0},\frac{\textbf{c}_0}{4\kappa}\right\}$.
	Thus, \eqref{PETNs-3.02} and Young's inequality yield, for all $t\geq1$,
	\begin{align*}%\label{PETNs-3.002}
%		\begin{split}
			&~~\int_{\tau - 1}^\tau \mathbb{E}\left[\|v(s, \tau - t, \zeta)\|^2\right]{d}s\leq e^{\kappa}\int_{\tau-t}^{\tau} e^{-\kappa (\tau - s)}\mathbb{E}\left[\|v(s, \tau - t, \zeta)\|^2\right]ds\\
			&\leq \rho_2e^{\kappa} \left(e^{-\kappa t}\mathbb{E}\left[|\zeta|^{4}\right]+ \int_{-\infty}^\tau e^{-\kappa(\tau-s)}|g(s)|^{4}ds+\frac{1}{4\kappa}+\frac{1}{4}e^{-\kappa t}+1\right),
%		\end{split}
	\end{align*}
	Taking $\mathcal{R}_2:=\rho_2e^{\kappa}\left(\frac{1}{2\kappa}+\frac{5}{2}\right)$ yields \eqref{UPe4-4.3}. This completes the proof.
\end{proof}

We next derive uniform moment estimates for the vertical derivative of solutions to \eqref{PEeq-2.15}.
\begin{lemma}\label{UPe4-lem4.2}
	Suppose Hypotheses \ref{PEeqAssum-2.2}, \ref{PEeqAssum-2.3} and \eqref{PETNs-3.01} hold. Then, for every $\tau \in \mathbb{R}$ and $D = \{D(t), t \in \mathbb{R}\} \in \mathfrak{D}$, there exist ${T} := {T}(\tau, D) > 0$ and  $\epsilon_{0}\in (0,1)$ such that for all $\epsilon\in (0,\epsilon_{0})$ and $t \geq {T}$, the solution $v$ of \eqref{PEeq-2.15} satisfies
	\begin{align}\label{UPe4-4.5}
		&~~\mathbb{E}\left[|\partial_{z}{v}(\tau, \tau - t, \zeta)|^4\right] + \int_{\tau - t}^\tau e^{-\kappa(\tau - s)}\mathbb{E}\left[|\partial_{z}v(s, \tau - t, \zeta)|^2\|\partial_{z}v(s, \tau - t, \zeta)\|^2\right]{d}s\\ \notag
		&\leq \mathcal{R}_3 + \mathcal{R}_3\int_{-\infty}^\tau e^{-\kappa(\tau - s)}|\partial_{z}g(s)|^{4} {d}s,
	\end{align}
%	and
%	\begin{align}\label{UPe4-4.003}
%		\begin{split}
%			&~~\int_{\tau - 1}^\tau \mathbb{E}\left[\|v(s, \tau - t, \zeta)\|^2+\|{v}(s, \tau - t, {v}_{\tau-t})\|^2\right]{d}s\\
%			&\leq \mathcal{R}_2 + \mathcal{R}_2\int_{-\infty}^\tau e^{-\gamma(\tau - s)}(|G(s)|^{4}+|Q(s)|^{4}) {d}s,
%		\end{split}
%	\end{align}
	where $\zeta\in L^4(\Omega,\mathscr F_{\tau-t};\mathcal H)$ with $\mathscr{L}(\zeta) \in D(\tau - t)$, and $\mathcal{R}_3$ is a positive constant independent of $\tau$, $\epsilon$ and $D$.
\end{lemma}
\begin{proof}
	Applying It\^{o}'s formula to $|\partial_{z}{v}(t)|^p$ for $p\in\{2,4\}$ and using \eqref{PEeq-2.15} yields, for $t\geq \tau$,
	\begin{align}\label{UPe4-4.6}
		&~~|\partial_{z}{v}(t)|^p+p\int_{\tau}^t |\partial_{z}{v}(s)|^{p-2}\|\partial_{z}{v}(s)\|^2ds \notag \\
		&\leq|\partial_{z}\zeta|^p+p\int_{\tau}^t |\partial_{z}{v}(s)|^{p-2}\left|\langle \partial_{z}{v}(s),\partial_{z}g(s) \rangle\right| ds+\underbrace{p\int_{\tau}^t |\partial_{z}{v}(s)|^{p-2} |\langle \partial_{zz}{v}(s),B(v(s),{v}(s)) \rangle|ds}_{=0 \text{ by  \eqref{PEeq-2.13}}}\notag \\
		&+ p\sqrt{\epsilon}\int_{\tau}^t |\partial_{z}{v}(s)|^{p-2}
		\langle \partial_{z}{v}(s),\partial_{z}\sigma (s,v(s)) dW(s)\rangle+\frac{p(p-1)}{2}\epsilon \int_{\tau}^t |\partial_{z}{v}(s)|^{p-2}\|\partial_{z}\sigma(s,v(s))\|_{\mathcal{L}_2(U,H)}^2ds\notag \\
		&+p{\epsilon}\int_{\tau}^t \int_{E} |\partial_{z}{v}(s)|^{p-2}\langle \partial_{z}{v}(s-),\partial_{z}h({v}(s-),\xi)\rangle \widetilde{N}^{\epsilon^{-1}}(ds,d\xi)\notag \\
		&+\int_{\tau}^{t} \int_{E}\left(|\partial_{z}{v}(s-)+{\epsilon}\partial_{z}h({v}(s-),\xi)|^{p}-|\partial_{z}{v}(s-)|^p-p{\epsilon}\langle \partial_{z}{v}(s-),\partial_{z}h({v}(s-),\xi)\rangle\right)N^{\epsilon^{-1}}(ds,d\xi).
	\end{align}
	Taking the expectation of \eqref{UPe4-4.6} and then setting $p=4$ yields
\begin{align}\label{UPe4-4.7}
	&~~\mathbb{E}\left[|\partial_{z}{v}(t)|^4\right]+4\int_{\tau}^t \mathbb{E}\left[|\partial_{z}{v}(s)|^{2}\|\partial_{z}{v}(s)\|^2\right]ds \notag \\
	&\leq \mathbb{E}\left[|\partial_{z}\zeta|^4\right]+4\int_{\tau}^t \mathbb{E}\left[|\partial_{z}{v}(s)|^{2}\left|\langle \partial_{z}{v}(s),\partial_{z}g(s) \rangle\right|\right] ds+{6\epsilon \int_{\tau}^t \mathbb{E}\left[|\partial_{z}{v}(s)|^{2}\|\partial_{z}\sigma(s,v(s))\|_{\mathcal{L}_2(U,H)}^2\right]ds}\notag \\
	&+\mathbb{E}\left[\int_{\tau}^{t} \int_{E}\left(|\partial_{z}v(s-)+{\epsilon}\partial_{z}h(v(s-),\xi)|^{4}-|\partial_{z}v(s-)|^4-4{\epsilon}\langle \partial_{z}v(s-),\partial_{z}h(v(s-),\xi)\rangle\right)N^{\epsilon^{-1}}(ds,d\xi)\right]\notag \\
	&=:\mathbb{E}\left[|\partial_{z}\zeta|^{4}\right]+\sum_{j=4}^{6}\mathscr{J}_j(t).
\end{align}
Using integration by parts, H\"{o}lder's inequality, and Young's inequality, we have
	\begin{align}\label{UPe4-4.8}
	\begin{split}
			\mathscr{J}_4(t)&\leq  4\int_{\tau}^t \mathbb{E}\left[|\partial_{z}{v}(s)|^{3}|\partial_{z}g(s)|\right] ds\leq \frac{\textbf{c}_0}{2}\int_{\tau}^{t}\mathbb E\left[|\partial_{z} {v}(s)|^{4}\right]ds+\frac{64}{\textbf{c}_0}\int_{\tau}^{t}|\partial_{z}g(s)|^{4}ds.
	\end{split}
	\end{align} 
	Set
	\[
	\mathfrak c_z:=8\wp_4\int_E|\widetilde C_h(\xi)|^4\,\nu(d\xi)
	+(12+2\wp_4)\left(\widetilde C_\sigma
	+2\int_E|\widetilde C_h(\xi)|^2\,\nu(d\xi)\right).
	\]
As in \eqref{ep-0_0}, since $\epsilon\in(0,1)$, estimate \eqref{LSWs3.meanin24} together with conditions \textbf{(A.3)} and \textbf{(B.3)} gives
\begin{align}\label{UPe4-4.9}
\mathscr{J}_5(t)+\mathscr{J}_6(t)&\leq 6\widetilde{C}_{\sigma}\epsilon \int_{\tau}^t \mathbb{E}\left[|\partial_{z}{v}(s)|^{2}\left(1+|\partial_zv(s)|^2\right)\right]ds\notag \\
&\quad+\wp_4\epsilon \int_{\tau}^{t} \mathbb{E}\left[\int_{E} \left(|\partial_{z}v(s-)|^{2}||\partial_{z}h(v(s-),\xi)|^2+|\partial_{z}h(v(s-),\xi)|^4\right)\nu(d\xi)\right]ds\notag \\
&\leq 8\wp_4\epsilon\int_{E} |\widetilde{C}_h(\xi)|^4\nu(d\xi)\int_{\tau}^{t}\mathbb{E}\left[(1+|\partial_zv(s)|^{4})\right]ds\\ \notag
&\quad+\left(6+\wp_4\right)\epsilon\left(\widetilde{C}_{\sigma}+2\int_{E} |\widetilde{C}_h(\xi)|^2\nu(d\xi)\right) \int_{\tau}^{t} \mathbb{E}\left[|\partial_zv(s)|^{2}+|\partial_zv(s)|^{4}\right]ds  \\ \notag
&\leq \mathfrak c_z
\epsilon\int_{\tau}^{t}\mathbb{E}\left[(1+|\partial_zv(s)|^{4})\right]ds.
\end{align}

Combining \eqref{UPe4-4.7}--\eqref{UPe4-4.9}, define
\begin{align}\label{ep-0_1}
	\epsilon_{0}^2=\min \left\{1,{\frac{\textbf{c}_0}{2\mathfrak c_z}}\right\}>0
\end{align}
Then, for $\epsilon\in(0,\epsilon_0^2)$,
\begin{align*}
	\frac{d}{dt}\mathbb{E}\left[|\partial_{z}{v}(t)|^4\right]+\left(
\frac{\textbf{c}_0}{8}-\kappa\right) \mathbb{E}\left[|\partial_{z}{v}(t)|^{4}\right]+\kappa \mathbb{E}\left[|\partial_{z}{v}(t)|^{4}\right] + \mathbb{E}\left[|\partial_{z}{v}(t)|^{2}\|\partial_{z}{v}(t)\|^2\right]\leq \frac{64}{\textbf{c}_0}|\partial_{z}g(t)|^{4}+\frac{\textbf{c}_0}{2},
\end{align*} 
Together with \eqref{LSWs4.1}, this yields
\begin{align}\label{UPe4-4.20}
	\frac{d}{dt}\mathbb{E}\left[|\partial_{z}{v}(t)|^4\right]+\kappa \mathbb{E}\left[|\partial_{z}{v}(t)|^{4}\right]+ \mathbb{E}\left[|\partial_{z}{v}(t)|^{2}\|\partial_{z}{v}(t)\|^2\right]\leq \frac{64}{\textbf{c}_0}|\partial_{z}g(t)|^{4}+\frac{\textbf{c}_0}{2}.
\end{align}
Integrating from $\tau-t$ to $\tau$ and applying Gronwall's inequality gives
\begin{align}\label{VE3.21}
	\begin{split}
		&~~\mathbb{E}\left[|\partial_{z}{v}(\tau,\tau-t,\zeta)|^4\right]+\int_{\tau - t}^{\tau} e^{\kappa(s-\tau)}\mathbb{E}\left[|\partial_{z}{v}(s,\tau-t,\zeta)|^{2}\|\partial_{z}{v}(s,\tau-t,\zeta)\|^2\right]ds\\
%		&\leq e^{-\kappa t}\mathbb{E}\left[|\partial_{z}\zeta|^{4}\right]+\frac{64}{\textbf{c}_0}\int_{\tau - t}^{\tau}e^{\kappa(s-\tau)}|\partial_{z}g(s)|^{4}ds+\frac{\textbf{c}_0}{2\kappa}\\
		&\leq e^{-\kappa t}\mathbb{E}\left[|\partial_{z}\zeta|^{4}\right]+\frac{64}{\textbf{c}_0}\int_{-\infty}^{\tau}e^{\kappa(s-\tau)}|\partial_{z}g(s)|^{4}ds+\frac{\textbf{c}_0}{2\kappa}.
	\end{split}
\end{align}
Let $\mathcal{R}_3=\max\left\{\frac{64}{\textbf{c}_0},\frac{\textbf{c}_0}{\kappa}\right\}$ and $\epsilon_{0}=\min\left\{\epsilon_{0}^1,\epsilon_{0}^2\right\}$, where $\epsilon_{0}^1$ is defined in \eqref{ep-0_0}. By the vertical-moment tempering condition in the definition of $\mathfrak D$, there exists $T=T(\tau,D)>0$  such that
\[
e^{-\kappa t}\sup_{\mu\in D(\tau-t)}\int_{H}|\partial_z \zeta|^4\,\mu(d\zeta)\leq \frac{c_0}{2\kappa},\qquad t\geq T.
\]
Consequently,
the estimate \eqref{UPe4-4.5} holds
for any $\mathscr L(\zeta)\in D(\tau-t)$, $t\ge T$, and $\epsilon\in(0,\epsilon_0)$. This completes the proof.
\end{proof}

We next establish a uniform higher-regularity estimate for solutions of \eqref{PEeq-2.15}.
\begin{lemma}\label{UPe4-lem3.5}
	Suppose Hypotheses \ref{PEeqAssum-2.2}, \ref{PEeqAssum-2.3}, \ref{PEeqAssum-2.4} and \eqref{PETNs-3.01} hold. For every $r\leq t$, set
	\begin{align}\label{UPe400-3.22}
		\Theta(t,r,\zeta)=e^{-\alpha_2\int_{r}^{t}(|{v}(s, r, \zeta)|^2\|{v}(s, r, \zeta)\|^2+|\partial_{z}{v}(s, r, \zeta)|^2\|\partial_{z}{v}(s, r, \zeta)\|^2)ds},
	\end{align}
	where $\alpha_2>0$ is chosen below so as to absorb the nonlinear convection terms.
	 Then, for every $\tau \in \mathbb{R}$ and $D \in \mathfrak{D}$, there exist ${T} := {T}(\tau, D) > 1$ and  $\epsilon_{0}\in (0,1)$ such that, for all $\epsilon\in (0,\epsilon_{0})$ and $t \geq {T}$, the solution $v$ of \eqref{PEeq-2.15} satisfies
	\begin{align}\label{UPe4-3.22}
		\mathbb{E}\left[\Theta(\tau,\tau-t,\zeta)\|{v}(\tau, \tau - t, \zeta)\|^2\right] \leq \mathcal{R}_4 + \mathcal{R}_4\int_{-\infty}^\tau e^{-\kappa(\tau - s)}|g(s)|^{4} {d}s,
	\end{align}
	where $\zeta\in L^4(\Omega,\mathscr F_{\tau-t};\mathcal H)$ with $\mathscr{L}(\zeta) \in D(\tau - t)$, and $\mathcal{R}_4$ is a positive constant independent of $\tau$, $\epsilon$ and $D$. 
	
\end{lemma}
\begin{proof}
Let $P_m$ denote the spectral projection associated with $A$, and let $v_m$ be the corresponding finite-dimensional Galerkin approximation. We begin by carrying out the analysis for the stopped process $v_m(\cdot\wedge \tau_{m,R})$, where the stopping time is defined as
\[
\tau_{m,R}:=\inf\left\{s\ge \tau-t:\ \|v_m(s)\|^2+\int_{\tau-t}^s |Av_m(r)|^2\,dr\ge R\right\}\wedge \tau.
\]
The finite-dimensional jump It\^{o} formula applies directly to $\|v_m\|^2$. The unweighted identity below is included solely for the purpose of deriving estimates for the convection term and the noise terms. The terminal estimate itself, however, is obtained prior to taking expectations, by applying the product formula to the stopped weighted energy $\Theta_m\|v_m\|^2$. This order of operations is essential, since $\Theta_m$ is random and depends on the entire Galerkin trajectory. To simplify the notation in the auxiliary estimates that follow, we suppress the subscripts $m$ and $R$. Once the weighted estimate is established, we pass to the limit by first letting $R\to\infty$ and then $m\to\infty$. The limiting procedure is justified by the Galerkin convergence provided in Theorem~\ref{PEeqthe-2.5}, together with the weak lower semicontinuity of the $V$-norm and Fatou's lemma.

	Applying It\^{o}'s formula to \eqref{PEeq-2.15} yields, for any $\tau\in \mathbb{R}$, $t>1$ and $\varsigma\in (\tau-1,\tau)$,
	\begin{align}\label{UPe4-3.24}
		&~\|{v}(\tau,\tau-t,\zeta)\|^2+2\int_{\varsigma}^{\tau} |A{v}(s,\tau-t,\zeta)|^{2}ds=\|v(\varsigma,\tau-t,\zeta)\|^2 \notag \\
		&-2\int_{\varsigma}^{\tau}  \langle B(v(s,\tau-t,\zeta),{v}(s,\tau-t,\zeta)),Av(s,\tau-t,\zeta)\rangle ds+2\int_{\varsigma}^{\tau}  \langle g(s),Av(s,\tau-t,\zeta)\rangle ds \notag \\
		&+2\sqrt{\epsilon} \int_{\varsigma}^{\tau} 
		\langle A{v}(s,\tau-t,\zeta),\sigma (s,v(s,\tau-t,\zeta)) dW(s)\rangle+\epsilon \int_{\varsigma}^{\tau}  \|\sigma(s,v(s,\tau-t,\zeta))\|_{\mathcal{L}_2(U,V)}^2ds\notag \\
		&+2{\epsilon}\int_{\varsigma}^{\tau} \int_{E} \langle A{v}(s-,\tau-t,\zeta),h({v}(s-,\tau-t,\zeta),\xi)\rangle \widetilde{N}^{\epsilon^{-1}}(ds,d\xi)\notag \\
		&+{\epsilon^2}\int_{\varsigma}^{\tau} \int_{E}\|h({v}(s-,\tau-t,\zeta),\xi)\|^2N^{\epsilon^{-1}}(ds,d\xi).
	\end{align}
 A direct application of Young's inequality gives
	\begin{align}\label{UPe4-3.25}
		2\int_{\varsigma}^{\tau}  \mathbb{E}\left[\langle g(s),Av(s,\tau-t,\zeta)\rangle\right] ds
		\leq 2\int_{\varsigma}^{\tau} |g(s)|^2ds+ \frac{1}{2} \int_{\varsigma}^{\tau} \mathbb{E}\left[|A{v}(s,\tau-t,\zeta)|^{2}\right]ds.
	\end{align}
	By Hypothesis \ref{PEeqAssum-2.4} we have
	\begin{align*}
		&~~\epsilon \int_{\varsigma}^{\tau}  \mathbb{E}\left[\|\sigma(s,v(s,\tau-t,\zeta))\|_{\mathcal{L}_2(U,V)}^2\right]ds+{\epsilon^2}\mathbb{E}\left[\int_{\varsigma}^{\tau} \int_{E}\|h({v}(s-,\tau-t,\zeta),\xi)\|^2N^{\epsilon^{-1}}(ds,d\xi)\right]\notag \\
		&\leq \epsilon \int_{\varsigma}^{\tau} \mathbb{E}\left[\|\sigma(s,v(s,\tau-t,\zeta))\|_{\mathcal{L}_2(U,V)}^2+\int_{E}\|h({v}(s-,\tau-t,\zeta),\xi)\|^2\nu(d\xi)
		\right]ds\notag \\
		&\leq \epsilon\alpha_1 \int_{\varsigma}^{\tau} \mathbb{E}\left[1+\|{v}(s-,\tau-t,\zeta)\|^2\right]ds,
	\end{align*}
  Hence, with $\epsilon_0^3=\min \left\{1,\frac{\textbf{c}_0}{2\alpha_1}\right\}$, for every $\epsilon\in (0,\epsilon_{0}^3)$,
	\begin{align}\label{UPe4-3.26}
		&~~\epsilon \int_{\varsigma}^{\tau}  \mathbb{E}\left[\|\sigma(s,v(s,\tau-t,\zeta))\|_{\mathcal{L}_2(U,V)}^2\right]ds+{\epsilon^2}\mathbb{E}\left[\int_{\varsigma}^{\tau} \int_{E}\|h({v}(s-,\tau-t,\zeta),\xi)\|^2N^{\epsilon^{-1}}(ds,d\xi)\right]\notag \\
		&\leq \frac{\textbf{c}_0}{2}+\frac{1}{2}\int_{\varsigma}^{\tau} \mathbb{E}\left[|A{v}(s,\tau-t,\zeta)|^{2}\right]ds.
	\end{align}
		It remains to deal with the second term on the right-hand side of \eqref{UPe4-3.24}.
		For the diagnostic vertical velocity
		$\mathcal W(v)(x,z)=-\int_{-l}^{z}\partial_xv(x,z')\,dz'$, the hydrostatic
		constraint and the boundary conditions give
		$\mathcal W(v)|_{z=-l,0}=0$. Hence the Ladyzhenskaya and Minkowski
		inequalities yield, for $v\in D(A)$,
		\begin{equation}\label{correct-W-estimate}
		 \|\mathcal W(v)\|_{L^4(\mathcal M)}
		 \leq C_l|\partial_xv|^{1/2}
		 \bigl(|\partial_xv|+|\partial_{xx}v|\bigr)^{1/2}
		 \leq C_l\|v\|^{1/2}|Av|^{1/2}.
		\end{equation}
		Consequently, by Ladyzhenskaya's, H\"{o}lder's, and Young's inequalities we have
		{\small
		\begin{align}\label{UPe4-3.27}
		&-2\mathbb{E}\left[\int_{\varsigma}^{\tau}  \langle B(v(s,\tau-t,\zeta),{v}(s,\tau-t,\zeta)),Av(s,\tau-t,\zeta)\rangle ds\right]\notag \\
		&\leq 2\mathbb{E}\left[\int_{\varsigma}^{\tau}  \left|\langle B(v(s,\tau-t,\zeta),{v}(s,\tau-t,\zeta)),Av(s,\tau-t,\zeta)\rangle\right| ds\right]\notag \\
		&\leq  2\int_{\varsigma}^{\tau}\mathbb{E}\left[ 
		\int_{\mathcal{M}}\left|v(s,\tau-t,\zeta) \partial_{x}v(s,\tau-t,\zeta) Av(s,\tau-t,\zeta)\right|dxdz \right]ds\notag\\
		&+2\int_{\varsigma}^{\tau}\mathbb{E}\left[\int_{\mathcal{M}}\left|\left(\int_{-l}^{z}\partial_{x} v(s,x,z')dz'\right)\partial_{z}v(s,\tau-t,\zeta) Av(s,\tau-t,\zeta) \right| dxdz \right]ds\notag\\
		&\leq C\int_{\varsigma}^{\tau}\mathbb{E}\left[|v(s,\tau-t,\zeta)|_{4}|\partial_{x}v(s,\tau-t,\zeta)|_{4}|Av(s,\tau-t,\zeta)| \right]ds\notag\\
			&+C_{l}\int_{\varsigma}^{\tau}\mathbb{E}\left[\|\mathcal W(v(s,\tau-t,\zeta))\|_{L^4(\mathcal M)} |\partial_{z}v(s,\tau-t,\zeta)|_4|Av(s,\tau-t,\zeta)| \right]ds\notag\\
		&\leq C\int_{\varsigma}^{\tau}\mathbb{E}\left[|v(s,\tau-t,\zeta)|^{1/2}\|v(s,\tau-t,\zeta)\|^{1/2}|\partial_{x}v(s,\tau-t,\zeta)|^{1/2}\|\partial_{x}v(s,\tau-t,\zeta)\|^{1/2}|Av(s,\tau-t,\zeta)| \right]ds\notag\\
		&+C_{l}\int_{\varsigma}^{\tau}\mathbb{E}\left[|\partial_{x}v(s,\tau-t,\zeta)|^{1/2}\|\partial_{x}v(s,\tau-t,\zeta)\|^{1/2} |\partial_{z}v(s,\tau-t,\zeta)|^{1/2}\|\partial_{z}v(s,\tau-t,\zeta)\|^{1/2}|Av(s,\tau-t,\zeta)| \right]ds\notag\\
		&\leq C_{l}\int_{\varsigma}^{\tau}\mathbb{E}\left[|v(s,\tau-t,\zeta)|^{1/2}\|v(s,\tau-t,\zeta)\||Av(s,\tau-t,\zeta)|^{3/2} \right]ds\notag\\
		&+C_{l}\int_{\varsigma}^{\tau}\mathbb{E}\left[\|v(s,\tau-t,\zeta)\|^{1/2} |\partial_{z}v(s,\tau-t,\zeta)|^{1/2}\|\partial_{z}v(s,\tau-t,\zeta)\|^{1/2}|Av(s,\tau-t,\zeta)|^{3/2} \right]ds\notag\\
		&\leq \frac{1}{2}\int_{\varsigma}^{\tau}\mathbb{E}\left[|Av(s,\tau-t,\zeta)|^{2} \right]ds
		+C_{l}\int_{\varsigma}^{\tau}\mathbb{E}\left[|v(s,\tau-t,\zeta)|^{2}\|v(s,\tau-t,\zeta)\|^{4} \right]ds\notag\\
		&+C_{l}\int_{\varsigma}^{\tau}\mathbb{E}\left[\|v(s,\tau-t,\zeta)\|^{2}|\partial_{z}v(s,\tau-t,\zeta)|^{2}\|\partial_{z}v(s,\tau-t,\zeta)\|^{2} \right]ds.
		\end{align}
}
	We now return to the stopped Galerkin equation and set
	\[
	\Xi_m(s):=|v_m(s)|^2\|v_m(s)\|^2
	+|\partial_zv_m(s)|^2\|\partial_zv_m(s)\|^2,
	\qquad
	\Theta_m(s):=\exp\left\{-\alpha_2\int_{\tau-t}^{s}\Xi_m(r)\,dr\right\}.
	\]
	Then $\Theta_m$ is a continuous adapted process of finite variation and
	$d\Theta_m=-\alpha_2\Theta_m\Xi_m\,ds$. Hence its quadratic covariation with
	$\|v_m\|^2$ vanishes. Applying the product formula on
	$[\varsigma,\tau\wedge\tau_{m,R}]$ gives
	\begin{align}\label{UPe4-weighted-product}
		&~d\bigl(\Theta_m\|v_m\|^2\bigr)
		+2\Theta_m|Av_m|^2ds
		+\alpha_2\Theta_m\Xi_m\|v_m\|^2ds\notag\\
		&=-2\Theta_m\langle B(v_m,v_m),Av_m\rangle ds
		+2\Theta_m\langle g,Av_m\rangle ds\notag\\
		&+\Theta_m\,d\mathcal M_m(s)
		+\epsilon\Theta_m\left(\|\sigma(s,v_m)\|_{\mathcal L_2(U,V)}^2
		+\int_E\|h(v_m,\xi)\|_V^2\nu(d\xi)\right)ds,
	\end{align}
	where $\mathcal M_m$ is the sum of the stopped Wiener martingale and the stopped compensated Poisson martingales, including the compensated part of the jump quadratic-variation term. Since $0<\Theta_m\le1$, it is a bounded predictable multiplier, so the stochastic integral in \eqref{UPe4-weighted-product} has zero expectation.

	The last Young estimate in \eqref{UPe4-3.27} is obtained pointwise before expectation and gives
	\[
	2|\langle B(v_m,v_m),Av_m\rangle|
	\le \frac12|Av_m|^2+C_l\Xi_m\|v_m\|^2.
	\]
	Moreover,
	\[
	2\Theta_m|\langle g,Av_m\rangle|
	\le \frac12\Theta_m|Av_m|^2+2\Theta_m|g|^2.
	\]
	Finally, Hypothesis~\ref{PEeqAssum-2.4} and the spectral coercivity
	$\|w\|^2\le C_A|Aw|^2$ for some $C_A>0$ and every $w\in D(A)$ imply, after decreasing
	$\epsilon_0^3>0$ if necessary,
	\begin{align*}
	&\epsilon\,\mathbb E\int_{\varsigma}^{\tau\wedge\tau_{m,R}}
	\Theta_m(s)\left(\|\sigma(s,v_m)\|_{\mathcal L_2(U,V)}^2
	+\int_E\|h(v_m,\xi)\|_V^2\nu(d\xi)\right)ds\\
	&\qquad\le \frac12\mathbb E\int_{\varsigma}^{\tau\wedge\tau_{m,R}}
	\Theta_m(s)|Av_m(s)|^2ds
	+C\int_{\varsigma}^{\tau}\mathbb E[\Theta_m(s)]ds,
	\end{align*}
	uniformly in $m$ and $R$. Choose $\alpha_2>C_l$. Taking expectations in
	\eqref{UPe4-weighted-product}, using the preceding three estimates, and then letting
	$R\to\infty$ and $m\to\infty$, we obtain, for every
	$\varsigma\in(\tau-1,\tau)$,
	\begin{align}\label{UPe4-3.28}
		&~~\mathbb{E}\left[\Theta(\tau,\tau-t,\zeta)\|{v}(\tau,\tau-t,\zeta)\|^2\right]\notag\\
		& \leq \mathbb{E}\left[\Theta(\varsigma,\tau-t,\zeta)\|v(\varsigma,\tau-t,\zeta)\|^2\right]
		+2\int_{\varsigma}^{\tau} \mathbb{E}\left[\Theta(s,\tau-t,\zeta)|g(s)|^2\right]ds
		+C\int_{\varsigma}^{\tau} \mathbb{E}\left[\Theta(s,\tau-t,\zeta)\right]ds\notag\\
		& \leq \mathbb{E}\left[\Theta(\varsigma,\tau-t,\zeta)\|v(\varsigma,\tau-t,\zeta)\|^2\right]
		+2\int_{\varsigma}^{\tau}|g(s)|^4ds+C.
	\end{align}
	Integrating \eqref{UPe4-3.28} with respect to $\varsigma$ from $\tau-1$ to $\tau$ and using \eqref{UPe4-4.3}, we obtain 
	\begin{align*}
		&~~\mathbb{E}\left[\Theta(\tau,\tau-t,\zeta)\|{v}(\tau,\tau-t,\zeta)\|^2\right]\notag\\
		& \leq \int_{\tau-1}^{\tau}\mathbb{E}\left[\Theta(\varsigma,\tau-t,\zeta)\|v(\varsigma,\tau-t,\zeta)\|^2\right]d\varsigma+2\int_{\tau-1}^{\tau}|g(s)|^4ds+C\notag\\
		&\leq \int_{\tau-1}^{\tau}\mathbb{E}\left[\|v(\varsigma,\tau-t,\zeta)\|^2\right]d\varsigma+2e^{\kappa}\int_{\tau-1}^{\tau}e^{-\kappa(\tau-s)}|g(s)|^4ds+C\notag\\
		&\leq \left(2e^{\kappa}+\mathcal{R}_2\right)\int_{-\infty}^{\tau}e^{-\kappa(\tau-s)}|g(s)|^4ds+\left(C+\mathcal{R}_2\right).
	\end{align*}
Thus, let $\mathcal{R}_4=\max\left\{2e^{\kappa}+\mathcal{R}_2,C+\mathcal{R}_2\right\}$ and $\epsilon_{0}=\min \left\{\epsilon_{0}^1,\epsilon_{0}^2,\epsilon_{0}^3\right\}$. Then the desired conclusion \eqref{UPe4-3.22} holds for every $\mathscr L(\zeta)\in D(\tau-t)$, $t\ge T>1$, and $\epsilon\in(0,\epsilon_0)$. 

More precisely, let $v_m$ be the Galerkin solutions used in the construction of $v$ and let $\Theta_m$ be the corresponding exponential weights.  The preceding calculation gives
\[
\sup_m\mathbb E\!\left[\Theta_m(\tau,\tau-t,\zeta_m)\|v_m(\tau)\|^2\right]\le C(\tau).
\]
After localization, the Galerkin compactness and identification step used in the well-posedness construction yields, along a subsequence (not relabeled),
\[
v_m\to v \text{ in probability in }L^2([\tau-t,\tau];H),\qquad
\partial_zv_m\to\partial_zv \text{ in probability in }L^2([\tau-t,\tau];H),
\]
and convergence in probability in $\mathcal D([\tau-t,\tau];H)$. Since the Poisson compensator is diffuse in time, $v$ has no jump at the deterministic time $\tau$ a.s.; hence the Skorokhod convergence gives $v_m(\tau)\to v(\tau)$ in probability in $H$. Consequently $\Theta_m(\tau,\tau-t,\zeta_m)\to\Theta(\tau,\tau-t,\zeta)$ in probability. Set
\[
\hat{Y}_m:=\Theta_m(\tau,\tau-t,\zeta_m)^{1/2}v_m(\tau).
\]
The terminal estimate makes $\{\hat{Y}_m\}$ bounded in $L^2(\Omega;V)$, so a subsequence converges weakly to some $\hat{Y}$ there. Since $0<\Theta_m\le1$, the preceding convergences also give $\hat{Y}_m\to\Theta(\tau,\tau-t,\zeta)^{1/2}v(\tau)$ in probability in $H$. Hence
\[
\hat{Y}=\Theta(\tau,\tau-t,\zeta)^{1/2}v(\tau).
\]
Weak lower semicontinuity in $L^2(\Omega;V)$ proves \eqref{UPe4-3.22}.  %Since the exponential weight is strictly positive almost surely, this also selects a $V$-valued representative of the terminal state.  Thus the terminal laws are tight in $H$ through the compact embedding $V\hookrightarrow H$.
This completes the proof.
\end{proof}

\subsection{Feller property and the measure evolution process}
%This section is devoted to proving the existence and uniqueness of $\mathfrak{D}$-pullback measure attractors for \eqref{PEeq-2.15} in $\mathcal{P}_4(H)$. 
The next result provides the continuous-dependence property needed to pass from the stochastic equation to an evolution process of probability laws.
Let $\mathcal{B}_b(H)$ be the set of all bounded Borel functions from $H$ to $\mathbb{R}$. For $\tau\in\mathbb R$, $t\ge\tau$, $\varphi\in\mathcal B_b(H)$, and $\zeta\in\mathcal H$, define the transition operator $p_{\tau,t}$ by
$$
(p_{\tau,t} \varphi)(\zeta)=\mathbb{E}\left[\varphi(v(t,\tau,\zeta))\right].
$$
For $\chi\in \mathcal{B}(H)$ and $\zeta\in\mathcal H$, set  
$$
P(\tau,\zeta;t,\chi)=(p_{\tau,t} \mathbf{1}_\chi)(\zeta)=\mathbb{P}\left(\left\{\omega\in \Omega: v(t,\tau,\zeta)\in \chi\right\}\right), 
$$
where $\mathbf{1}_\chi$ is the characteristic function of $\chi$. 
%Thus $p_{\tau,t}\varphi$ is a bounded Borel function on $\mathcal H$ endowed with the relative $H$-topology; the Feller property below asserts continuity for $\varphi\in C_b(H)$.

\begin{lemma}\label{UPe400-fm-3.6}
	Suppose Hypotheses \ref{PEeqAssum-2.2}, \ref{PEeqAssum-2.3}, \ref{PEeqAssum-2.4} and \eqref{PETNs-3.01} hold. Then the family $\{p_{\tau,t}\}_{t\geq \tau}$ has the following properties: 
	
		$(i)$ $\{p_{\tau,t}\}_{t\geq \tau}$ satisfies the Feller property, i.e., for any $t\geq \tau$, if $\phi\in C_b({H})$, the map $p_{\tau,t}\phi$ is bounded and continuous on $\mathcal H$ endowed with the relative $H$-topology;
	
	$(ii)$ the family $\{v(t,\tau,\zeta): t\geq \tau, \zeta\in\mathcal H\}$ is Markov; in particular, $p_{\tau,t}=p_{\tau,r}p_{r,t}$ for any $\tau\leq  r \leq t$. 
%	$(i)$ {$\{p_{\tau,t}\}_{t\geq \tau}$ satisfies the Feller property, i.e., $\zeta^n\to\zeta$ in $H$, with $\zeta^n,\zeta\in\mathcal H$, implies $(p_{\tau,t}\phi)(\zeta^n)\to(p_{\tau,t}\phi)(\zeta)$ for every $\phi\in C_b(H)$};
%	
%	$(ii)$ the family $\{v(t,\tau,\zeta): t\geq \tau, {\zeta\in \mathcal H}\}$ is Markov, particularly, $p_{\tau,t}=p_{\tau,r}p_{r,t}$ for any $\tau\leq  r \leq t$.
\end{lemma}
\begin{proof}
%	Since the proof of property $(ii)$ is standard (see pp. 167-170 of \cite{Peszat-2007-Levy} for similar argument), we only prove property $(i)$. For this purpose, we proceed in three steps. Fix $t_0\geq \tau$, and let $\zeta^n \rightarrow \zeta$ in $L^2(\Omega,H)$ as $n\rightarrow \infty$.
We first verify property $(ii)$. Pathwise uniqueness and the Galerkin construction in Theorem~\ref{PEeqthe-2.5} imply that, for every $r\ge\tau$, the solution on $[r,t]$ is a measurable function of its value at time $r$ and of the increments
\[
W(s)-W(r),\qquad N^{\epsilon^{-1}}((r,s]\times\cdot),
\qquad r\le s\le t.
\]
These increments are independent of $\mathscr F_r$. Consequently, for every $\phi \in C_b(H)$,
\[
\mathbb E\!\left[\phi(v(t,\tau,\zeta))\mid\mathscr F_r\right]
=(p_{r,t}\phi)(v(r,\tau,\zeta))\quad\mathbb P\text{-a.s.}
\]
Taking expectations and using the tower property gives
\[
(p_{\tau,t}\phi)(\zeta)=(p_{\tau,r}p_{r,t}\phi)(\zeta),
\]
which proves the inhomogeneous Markov and Chapman-Kolmogorov properties. See also \cite[pp.~167--170]{Peszat-2007-Levy}.

%We now prove property $(i)$. Fix $t_0\geq \tau$, and let $\zeta^n,\zeta\in\mathcal H$ be deterministic initial values such that $\zeta^n\to\zeta$ in $H$.
We now prove property $(i)$. Fix $t_0\geq \tau$, and let $\zeta^n,\zeta\in\mathcal H$ such that $\zeta^n\to\zeta$ in $H$.
	
	\textbf{Step 1.}  Define the stopping time $\tau_{R}=\widehat{\tau}_{R}\wedge \widetilde{\tau}_{R}$, where
	$$
	\widehat{\tau}_{R}:=\inf \left\{t\in [\tau,t_0]: |v(t,\tau,\zeta^n)|^2>R \text{~~or~~} |v (t,\tau,\zeta)|^2>R \text{~~or~~} \int_{\tau}^t \|v (s,\tau,\zeta)\|^2 ds>R^2\right\},
	$$
	and
		$$
	\widetilde{\tau}_{R}:=\inf \left\{t\in [\tau,t_0]:  |\partial_{z}v (t,\tau,\zeta)|^2>R \text{~~or~~} \int_{\tau}^t \|\partial_{z}v (s,\tau,\zeta)\|^2 ds>R^2\right\}.
	$$
If either defining set is empty, the corresponding stopping time is set equal to $t_0$. Thus, for every $t\le\tau_R$,
	\begin{align}\label{UPe400-3.209}
		\begin{split}
			|v(t,\tau,\zeta^n)|^2\leq R,~~~~ |v (t,\tau,&\zeta)|^2\leq R,~~~~ \int_{\tau}^t \|v (s,\tau,\zeta)\|^2 ds\leq R^2,\\
			|\partial_{z}v (t,\tau,\zeta)|^2\leq R,&~~~~ \int_{\tau}^t \|\partial_{z}v (s,\tau,\zeta)\|^2 ds\leq R^2.
		\end{split}
	\end{align}
	We set 
	\begin{align*}%\label{UPe400f-3.22}
		y(t):=e^{-\alpha_3\int_{\tau}^{t}(\|{v}(s, \tau, \zeta)\|^2+|\partial_{z}{v}(s, \tau, \zeta)|^{4})ds}, \quad \forall t\in [\tau,t_0],
	\end{align*}
	where $\alpha_3>0$ is chosen sufficiently large.
Set $\bar{v}_n=v(\cdot,\tau,\zeta^n)-v(\cdot,\tau,\zeta)$.  The bilinear difference is oriented as
\[
B(v^n,v^n)-B(v,v)=B(v^n,\bar{v}_n)+B(\bar{v}_n,v).
\]
The first term cancels after pairing with $\bar{v}_n$.  The primitive-equation estimate \eqref{PEeq-2.103} and Young's inequality give
\[
2|\langle B(\bar{v}_n,v),\bar{v}_n\rangle|
\le \frac12\|\bar{v}_n\|^2
+C\bigl(\|v\|^2+|\partial_zv|^4\bigr)|\bar{v}_n|^2.
\]
This orientation is important: the exponential weight depends only on the fixed reference solution $v(\cdot,\tau,\zeta)$ and not on any uniform bound for $\partial_z\zeta^n$.  Applying It\^{o}'s formula with jumps to $y(t)|\bar{v}_n(t)|^2$, using \textbf{(A.2)} and \textbf{(B.2)}, and then applying the BDG and Young inequalities yields, after stopping at $\tau_R$,
\begin{align}\label{UPe400d-3.30}
	\mathbb{E}\left[\sup_{\tau\le s\le t_0}y(s\wedge\tau_R)|\bar{v}_n(s\wedge\tau_R)|^2
	+\int_\tau^{t_0\wedge\tau_R}y(s)\|\bar{v}_n(s)\|^2ds\right]
	\le C_{t_0}|\zeta^n-\zeta|^2.
\end{align}
The constant $C_{t_0}>0$ is independent of $n$ and $R$.  This is the continuous dependence estimate required for the relative $H$-Feller property.
	On $[\tau,\tau_R]$, Poincar\'{e}'s inequality and
	\eqref{UPe400-3.209} give
	\begin{align}\label{UPe-adf3.3032}
		\int_\tau^{t_0\wedge\tau_R}|\partial_zv(s)|^4\,ds
		\le \sup_{\tau\le s\le\tau_R}|\partial_zv(s)|^2
		\int_\tau^{\tau_R}|\partial_zv(s)|^2\,ds\le \frac{R}{\textbf{c}_0}\int_\tau^{\tau_R}
		\|\partial_zv(s)\|^2\,ds
		\le \frac{R^3}{\textbf{c}_0}.
	\end{align}
	
%	The stopping time $\tau_R$ is introduced here because the quantity 
%	$\sup_{\tau \leq s \leq t_0} |\partial_z v(s,\tau,\zeta)|^2$ is not 
%	known to be finite almost surely on the entire interval $[\tau, t_0]$ 
%	without an additional uniform bound. The estimates in (3.31) are 
%	therefore valid only on the stopped interval $[\tau, \tau_R]$. 
%	After removing the stopping time via Chebyshev's inequality and 
%	the uniform moment bounds in (3.34)--(3.35), the resulting probability 
%	estimates yield the desired tightness of the terminal laws. 
	
	Since $t\mapsto y(t)$ is decreasing, \eqref{UPe400-3.209} and \eqref{UPe-adf3.3032} imply
	\begin{align*}
		\inf_{\tau\leq t\leq t_0}y(t\wedge \tau_R)&=e^{-\alpha_3\int_{\tau}^{t_0\wedge {\tau_R}}(\|{v}(s, \tau, \zeta)\|^2+|\partial_{z}{v}(s, \tau, \zeta)|^{4})ds}\\
		&\geq e^{-\alpha_3\left(\int_{\tau}^{{\tau_R}}\|{v}(s, \tau, \zeta)\|^2ds+\frac{R}{\textbf{c}_0}\int_{\tau}^{{\tau_R}}\|\partial_{z}{v}(s, \tau, \zeta)\|^{2}ds\right)}\geq e^{-\frac{\alpha_3}{\textbf{c}_0}R^2(\textbf{c}_0+R)},
	\end{align*}
	Combining this with \eqref{UPe400d-3.30} yields
	\begin{align}\label{UPe-f3.3032}
		\mathbb{E}\left[\sup_{\tau\leq t\leq t_0}|v(t\wedge \tau_R,\tau,\zeta^n)-v(t\wedge \tau_R,\tau,\zeta)|^2\right]\leq C_{\textbf{c}_0,t_0,R}|\zeta^n-\zeta|^2.
	\end{align}
	
	\textbf{Step 2.}
	 We prove that the family $\{\mathscr{L}(v(t_0,\tau,\zeta^n)): n\in \mathbb{N}_0:=\mathbb{N} \cup \{0\}\}$ of distribution laws is tight in $H$.
	 It suffices to show that, for fixed $t_0\geq\tau$, $v(t_0,\tau,\zeta^n)\to v(t_0,\tau,\zeta)$ in probability.
	 For $R_1>0$,
	 \begin{align}\label{UPe-f3.33}
	 	\begin{split}
	 		&~~\mathbb{P}\left(|v(t_0,\tau, \zeta^n)-v(t_0, \tau, \zeta)|>R_1\right)
	 		\leq  \mathbb{P}\left(\sup_{\tau\leq r\leq t_0}|v(r,\tau, \zeta^n)-v(r, \tau, \zeta)|>R_1\right)\\
	 		&\leq  \mathbb{P}\left(\left\{\sup_{\tau\leq r\leq t_0}|v(r\wedge \tau_{R},\tau, \zeta^n)-v(r\wedge \tau_{R}, \tau, \zeta)|>R_1\right\}\right)+\mathbb{P}\left(\tau_{R}<t_0\right)\\
	 		&=:\text{Term 1}+\text{Term 2}.
	 	\end{split}
	 \end{align}
	 
	 Equations~\eqref{ee-5}--\eqref{ee-6} yield a constant $C>0$ such that, for all $n\in\mathbb N_0$,
	 \begin{align}\label{UPe-f34}
	 	\mathbb{E}\left[\sup_{\tau\leq r\leq t_0}{| {v(r,\tau, \zeta^n)} |^2} \right]+\mathbb{E}\left[\sup_{\tau\leq r\leq t_0}{| {v(r,\tau,\zeta)} |^2} \right]+\mathbb{E} \left[\int_{\tau}^{t_0}  {\| {v(s,\tau,\zeta)} \|^2} ds\right]\leq C,
	 \end{align} 
	 and
	 \begin{align}\label{UPe-f35}
	 	\mathbb{E}\left[\sup_{\tau\leq r\leq t_0}{| {\partial_{z}v(r,\tau, \zeta)} |^2} \right]+\mathbb{E} \left[\int_{\tau}^{t_0}  {\| {\partial_{z} v(s,\tau,\zeta)} \|^2} ds\right]\leq C.
	 \end{align} 
	  
	  Combining \eqref{UPe-f3.3032} with Chebyshev's inequality yields
	 \begin{align}\label{UPe-f36}
	 	\begin{split}
	 		\text{Term 1}
	 		&\leq \frac{1}{R_1^2}\mathbb{E}\left[\sup_{\tau\leq r\leq t_0}|v(r\wedge \tau_{R},\tau, \zeta^n)-v(r\wedge \tau_{R}, \tau, \zeta)|^2\right]\leq \frac{C_{\textbf{c}_0,t_0,R}}{R_1^2}\mathbb{E}|\zeta^n-\zeta|^2,
	 	\end{split}
	 \end{align}
	 while \eqref{UPe-f34}--\eqref{UPe-f35} give 
	 \begin{align}\label{UPe-f37}
	 		\text{Term 2}&\leq \mathbb{P}\left(\sup_{\tau\leq r\leq t_0}\int_{\tau}^r \|v(s,\tau,\zeta)\|^2ds>R^2\right)+\mathbb{P}\left(\sup_{\tau\leq r\leq t_0}|v(r,\tau,\zeta^n)|^2>R\right)\notag \\
	 		&\quad+\mathbb{P}\left(\sup_{\tau\leq r\leq t_0}|v(r,\tau,\zeta)|^2>R\right)+\mathbb{P}\left(\sup_{\tau\leq r\leq t_0}\int_\tau^r \|\partial_{z}v(s,\tau,\zeta)\|^2ds>R^2\right)\notag\\
	 		&\quad+\mathbb{P}\left(\sup_{\tau\leq r\leq t_0}|\partial_{z}v(r,\tau,\zeta)|^2>R\right)\notag \\
	 		&\leq \frac{1}{R^2}\bigg(\mathbb{E}\left[\int_\tau^{t_0} \left(\|v(s,\tau,\zeta)\|^2+\|\partial_{z}v(s,\tau,\zeta)\|^2\right)ds\right]\notag\\
	 		&\quad +\mathbb{E}\left[\sup_{\tau\leq r\leq t_0}\left(|v(r,\tau,\zeta^n)|^2+|v(r,\tau,\zeta)|^2+|\partial_{z}v(r,\tau,\zeta)|^2\right)\right]\notag \\
	 		&\leq \frac{C}{R^2}+\frac{C}{R}.
	 \end{align}
	 By \eqref{UPe-f3.33}, \eqref{UPe-f36} and \eqref{UPe-f37}, we deduce
	 \begin{align*}%\label{5.14}
	 	\mathbb{P}\left(|v(t_0,\tau, \zeta^n)-v(t_0, \tau, \zeta)|>R_1\right)
	 		\leq \frac{C_{\textbf{c}_0,t_0,R}}{R_1^2}|\zeta^n-\zeta|^2+\frac{C}{R}+\frac{C}{R^2}.
	 \end{align*}
For fixed $R_1>0$, first choose $R$ sufficiently large and then let $n\to\infty$. Thus
$v(t_0,\tau,\zeta^n)\to v(t_0,\tau,\zeta)$ in probability.
	 
	 \textbf{Step 3.}  For $\varphi\in C_b(H)$, we need to prove the following convergence
	 \begin{align}\label{lemma5.4(5.35)}
	 	\lim_{n\rightarrow\infty} \mathbb{ E}\left[\varphi(v(t_0,\tau,\zeta^n))\right]=\mathbb{ E}\left[\varphi(v(t_0,\tau,\zeta))\right].
	 \end{align}
	 By \textbf{Step 2}, $v(t_0,\tau,\zeta^n) \to v(t_0,\tau,\zeta)$ in probability in $H$. In particular, the family of laws $\{\mathscr L(v(t_0,\tau,\zeta^n)):n\in\mathbb N\}\cup\{\mathscr L(v(t_0,\tau,\zeta))\}$ is tight. Fix $\varepsilon>0$. There exists a compact set $\mathcal{K}^{\varepsilon}\subset H$ such that
	 \begin{equation}\label{PEeq-Feller-step3-tight}
	 	\sup_{n\in\mathbb N}\mathbb P(v(t_0,\tau,\zeta^n)\notin \mathcal{K}^{\varepsilon})
	 	+\mathbb P(v(t_0,\tau,\zeta)\notin \mathcal{K}^{\varepsilon})<\varepsilon.
	 \end{equation}
	 Since $\varphi\in C_b(H)$, its restriction to $\mathcal{K}^{\varepsilon}$ is uniformly continuous. Therefore, there exists $\delta=\delta(\varepsilon,\mathcal{K}^{\varepsilon})>0$ such that, for all $v_1, v_2\in \mathcal{K}^{\varepsilon}$ with $|v_1-v_2|<\delta$, 
	 \begin{equation}\label{PEeq-Feller-step3-uc}
	 	|\varphi(v_1)-\varphi(v_2)|<\varepsilon.
	 \end{equation}
	 Writing
	 $\widehat{\textbf{c}}=\|\varphi\|_{C_b(H)}$ and splitting the expectation according to the events in \eqref{PEeq-Feller-step3-tight}-\eqref{PEeq-Feller-step3-uc}, we obtain
	 \begin{align*}
	 	\mathbb E\left[|\varphi(v(t_0,\tau,\zeta^n))-\varphi(v(t_0,\tau,\zeta))|\right]
	 	&\le \varepsilon
	 	+2\widehat{\textbf{c}}\Bigl[
	 	\mathbb P(v(t_0,\tau,\zeta^n)\notin \mathcal{K}^{\varepsilon})\\
	 	&\quad +\mathbb P(v(t_0,\tau,\zeta)\notin \mathcal{K}^{\varepsilon})
	 	+\mathbb P(|v(t_0,\tau,\zeta^n)-v(t_0,\tau,\zeta)|\ge\delta)
	 	\Bigr].
	 \end{align*}
	 The last probability tends to zero because $v(t_0,\tau,\zeta^n)\to v(t_0,\tau,\zeta)$ in probability. Consequently,
	 \[
	 \limsup_{n\to\infty}\mathbb E\left[|\varphi(v(t_0,\tau,\zeta^n))-\varphi(v(t_0,\tau,\zeta))|\right]
	 \le \varepsilon+2\widehat{\textbf{c}}\varepsilon.
	 \]
	 Letting $\varepsilon\downarrow0$ gives
	 \[
	 \lim_{n\to\infty}\mathbb E\bigl[\varphi(v(t_0,\tau,\zeta^n))\bigr]
	 =\mathbb E\bigl[\varphi(v(t_0,\tau,\zeta))\bigr],
	 \]
	 which implies \eqref{lemma5.4(5.35)}. Thus $p_{\tau,t_0}\varphi$ is continuous on $\mathcal H$ endowed with the $H$-topology, and property $(i)$ follows.
	  This completes the proof.
\end{proof}

For any $t\geq \tau$, we define the dual operator $p^*_{\tau,t}$ of the transition operator $p_{\tau,t}$ by
\begin{align}\label{dual-Pes}
p^*_{\tau,t}\mu(\cdot)
=\int_{\mathcal H}P(\tau,\zeta;t,\cdot)\,\mu(d\zeta),
\qquad \mu\in\mathcal P_{4,z}(H).
\end{align}
By Theorem~\ref{PEeqthe-2.5}, $p^*_{\tau,t}\mu$ again belongs to $\mathcal P_{4,z}(H)$. We set
\begin{align}\label{SLRSSs3.19}
S_\epsilon(t,\tau)\mu=p^*_{\tau,t}\mu,
\qquad \mu\in\mathcal P_{4,z}(H).
\end{align}

The evolution is therefore not asserted on all of $\mathcal P_4(H)$. It is an evolution of admissible laws in $\mathcal P_{4,z}(H)$, while convergence, attraction and compactness are measured by $d_{\mathcal P(H)}$.
\begin{lemma}\label{UPe4-lem-NADY}
	Suppose Hypotheses \ref{PEeqAssum-2.2}, \ref{PEeqAssum-2.3}, \ref{PEeqAssum-2.4} and \eqref{PETNs-3.01} hold. Then $\{S_\epsilon(t,\tau):t\geq\tau\}$ is a continuous two-parameter evolution process on $\mathcal P_{4,z}(H)$. More precisely, for every $\tau\in\mathbb R$: 
	
	$(i)$  $S_{\epsilon}(\tau,\tau)=I_{\mathcal{P}_{4,z}({H})}$;
	
	$(ii)$ $S_{\epsilon}(t,\tau)=S_{\epsilon}(t,s)S_{\epsilon}(s,\tau)$ for any $t\geq s \geq \tau$;
	
	$(iii)$ $S_{\epsilon}(t,\tau): \mathcal{P}_{4,z}({H})\rightarrow \mathcal{P}_{4,z}({H})$ is continuous for any $t\geq \tau$. 
\end{lemma}   
\begin{proof}
	The identity property $(i)$ follows from $v(\tau,\tau,\zeta)=\zeta$. The process property $(ii)$ follows from the inhomogeneous Markov property of solutions of \eqref{PEeq-2.15} in Lemma \ref{UPe400-fm-3.6}. Specifically, for any $t\geq s \geq \tau$,
	$$
	S_{\epsilon}(t,s)S_{\epsilon}(s,\tau) =p^*_{s,t}p^*_{\tau,s}=\left(p_{\tau,s} p_{s,t}\right)^*=p_{\tau,t}^*=S_{\epsilon}(t,\tau).
	$$
	We now verify $(iii)$. For $\mu\in\mathcal P_{4,z}(H)$, the Tonelli's theorem, the uniform estimates in Lemma \ref{lemPETNs-3.1} and \eqref{VE3.21} give
	\begin{align*}%\label{PEeq-S-map-moment}
		&~~\int_H \left(|\zeta|^4+|\partial_{z}\zeta|^4\right)\,\left(S_\epsilon(t,\tau)\mu\right)(d\zeta)
		=\int_{H}\mathbb E\bigl[|v(t,\tau,\zeta)|^4+|\partial_{z}v(t,\tau,\zeta)|^4\bigr] \,\mu(d\zeta)\notag\\
		&\le C\left(e^{-\kappa(t-\tau)}\int_H \left(|\zeta|^4+|\partial_{z}\zeta|^4\right)\,\mu(d\zeta)+\int_\tau^te^{-\kappa(t-s)}\left(|g(s)|^{4}+|\partial_{z}g(s)|^{4}\right)ds+1\right)<\infty.
	\end{align*}
	Thus, $S_\epsilon(t,\tau)$ maps $\mathcal P_{4,z}(H)$ into itself.
	It remains to prove continuity. Let $\mu_n,\mu\in\mathcal P_{4,z}(H)$ and assume that $\mu_n\to\mu$ in $d_{\mathcal P(H)}$. By the Skorokhod representation theorem, on an auxiliary probability space there exist random variables $\zeta_n,\zeta$ with laws $\mu_n,\mu$ such that $\zeta_n\to\zeta$ almost surely in $H$. Since all these laws belong to $\mathcal P_{4,z}(H)$, the variables are $\mathcal H$-valued almost surely. For every $\varphi\in C_b(H)$, Lemma~\ref{UPe400-fm-3.6} implies, samplewise,
	\[
	(p_{\tau,t}\varphi)(\zeta_n)\longrightarrow(p_{\tau,t}\varphi)(\zeta),
	\]
	because $p_{\tau,t}\varphi$ is continuous on $\mathcal H$ with the relative $H$-topology. Since this function is bounded, dominated convergence gives
	\[
	\int_H\varphi\,d(S_\epsilon(t,\tau)\mu_n)
	=\mathbb E[\varphi(v(t,\tau,\zeta_n))]
	\rightarrow
	\mathbb E[\varphi(v(t,\tau,\zeta))]
	=\int_H\varphi\,d(S_\epsilon(t,\tau)\mu).
	\]
	Thus $S_\epsilon(t,\tau)\mu_n\to S_\epsilon(t,\tau)\mu$ in $d_{\mathcal P(H)}$. Notice that this argument uses only the relative $H$-Feller property and does not require convergence of the vertical derivative moments. This completes the proof.
\end{proof}

We refer to $\{S_\epsilon(t,\tau):t\ge\tau\}$ as the measure evolution process on the admissible law class $\mathcal P_{4,z}(H)$.

\subsection{Existence of pullback measure attractors}

This section proves existence and uniqueness of a $\mathfrak D$-pullback measure attractor for the evolution of admissible laws on $\mathcal P_{4,z}(H)$, with compactness and attraction measured in the ambient topology of $\mathcal P(H)$. We first construct a closed $\mathfrak D$-pullback absorbing family for $S_\epsilon$.

\begin{lemma}\label{lem-moment-lsc}
	Define the extended functionals on $H$ by
	\[
	\Phi_0(u)=|u|^4,
	\qquad
	\Phi_z(u)=
	\begin{cases}
		|\partial_zu|^4,&u\in\mathcal H,\\
		+\infty,&u\notin\mathcal H.
	\end{cases}
	\]
	Then $\Phi_0$ and $\Phi_z$ are lower semicontinuous on $H$. Consequently, if $\mu_n\to\mu$ weakly in $\mathcal P(H)$, then
	\[
	\int_H\Phi_i(u)\,\mu(du)
	\le\liminf_{n\to\infty}\int_H\Phi_i(u)\,\mu_n(du),
	\qquad i\in\{0,z\}.
	\]
	In particular, every weak limit of a sequence having uniform bounds for both displayed moments belongs to $\mathcal P_{4,z}(H)$.
\end{lemma}
\begin{proof}
	Only the assertion for $\Phi_z$ needs verification. Suppose $u_n\to u$ in $H$ and $\liminf_n\Phi_z(u_n)<\infty$. Passing to a subsequence, $\{\partial_zu_n\}$ is bounded in $H$, and hence $\partial_zu_n\rightharpoonup w$ in $H$ along a further subsequence. Since the distributional vertical derivative is a closed operator, $u\in\mathcal H$ and $\partial_zu=w$. Weak lower semicontinuity of the $H$-norm gives $|\partial_zu|^4\le\liminf_n|\partial_zu_n|^4$. The integral inequalities follow from the Portmanteau theorem.
\end{proof}

\begin{lemma}\label{UPe4-lemABs}
	Suppose Hypotheses \ref{PEeqAssum-2.2}, \ref{PEeqAssum-2.3}, \ref{PEeqAssum-2.4} and \eqref{PETNs-3.01} hold. For $\tau\in\mathbb R$, set
	\[
	\mathscr M_0(\tau)
	=\mathcal R_1\left(
	1+\int_{-\infty}^{\tau}e^{-\kappa(\tau-s)}|g(s)|^4\,ds
	\right),
	\]
	and
	\[
	\mathscr M_z(\tau)
	=\mathcal R_3\left(
	1+\int_{-\infty}^{\tau}e^{-\kappa(\tau-s)}
	|\partial_zg(s)|^4\,ds
	\right).
	\]
	Define
	\[
	\mathscr K(\tau)
	=
	\left\{
	\mu\in\mathcal P_{4,z}(H):
	\int_H|\zeta|^4\,\mu(d\zeta)\le\mathscr M_0(\tau),\
	\int_H|\partial_z\zeta|^4\,\mu(d\zeta)\le\mathscr M_z(\tau)
	\right\}.
	\]
	Then $\mathscr K=\{\mathscr K(\tau):\tau\in\mathbb R\}$ belongs to $\mathfrak D$ and is a closed $\mathfrak D$-pullback absorbing family for $S_\epsilon$, uniformly for $\epsilon\in(0,\epsilon_0)$.
\end{lemma}
\begin{proof}
	By Lemma~\ref{lem-moment-lsc}, both moment constraints defining $\mathscr K(\tau)$ are closed under weak convergence in $\mathcal P(H)$. Hence $\mathscr K(\tau)$ is closed in $d_{\mathcal P(H)}$ and every one of its elements is admissible.
	
	Let $D\in\mathfrak D$. Lemma~\ref{UPe4-lem4.1} and the $H$-moment tempering condition give, for sufficiently large pullback time,
	\[
	\sup_{\mu\in D(\tau-t)}
	\int_H\mathbb E[|v(\tau,\tau-t,\zeta)|^4]\,\mu(d\zeta)
	\le\mathscr M_0(\tau).
	\]
	Similarly, Lemma~\ref{UPe4-lem4.2} together with the vertical-moment tempering condition yields
	\[
	\sup_{\mu\in D(\tau-t)}
	\int_H\mathbb E[|\partial_zv(\tau,\tau-t,\zeta)|^4]\,\mu(d\zeta)
	\le\mathscr M_z(\tau).
	\]
Consequently,
	\[
	S_\epsilon(\tau,\tau-t)D(\tau-t)\subset\mathscr K(\tau)
	\]
	for all sufficiently large $t$, uniformly in $\epsilon\in(0,\epsilon_0)$.
	
	Finally, assumption \eqref{PETNs-3.01} implies
	\[
	e^{\kappa\tau}\mathscr M_0(\tau)\to0
	\text{~ and ~}
	e^{\kappa\tau}\mathscr M_z(\tau)\to0
	\quad\text{as ~}\tau\to-\infty.
	\]
	Thus $\mathscr K\in\mathfrak D$.
\end{proof}

\begin{remark}\label{rem:common-absorbing}
	The estimates in Lemmas~\ref{UPe4-lem4.1}, \ref{UPe4-lem4.2}, and~\ref{UPe4-lem3.5} remain valid for $\epsilon=0$ after deleting the stochastic terms, with the same constants. Hence
	$\mathscr K=\{\mathscr K(\tau):\tau\in\mathbb R\}$ is a common closed $\mathfrak D$-pullback absorbing family for all $\epsilon\in[0,\epsilon_0)$. This uniformity is used in Section~\ref{USC-PEs4}.
\end{remark}

Next, we verify the $\mathfrak D$-pullback asymptotic compactness of $S_\epsilon$ in the ambient weak topology of $\mathcal P(H)$ and show that all limit laws remain in $\mathcal P_{4,z}(H)$.
\begin{lemma}\label{UPe4-lemcompatight}
	Suppose Hypotheses \ref{PEeqAssum-2.2}, \ref{PEeqAssum-2.3}, \ref{PEeqAssum-2.4} and \eqref{PETNs-3.01} hold. Then, for every $\epsilon\in(0,\epsilon_0)$, the measure evolution process $S_\epsilon$ on $\mathcal P_{4,z}(H)$ is $\mathfrak{D}$-pullback asymptotically compact in the ambient topology $(\mathcal P(H),d_{\mathcal P(H)})$. More precisely, for every $\tau\in\mathbb R$, $t_n \to +\infty$, and ${\mu}_n\in D(\tau-t_n)$ with $D\in \mathfrak{D}$, the sequence $\{S_\epsilon(\tau,\tau-t_n){\mu}_n\}_{n=1}^\infty$ has a subsequence converging in $\mathcal P(H)$, and every such limit belongs to $\mathcal P_{4,z}(H)$.
\end{lemma}
\begin{proof}
	By Prokhorov's theorem, it suffices to prove that the sequence of distributions 
	$\{\mathscr{L}(v(\tau,\tau-t_n,\zeta^n))\}_{n=1}^\infty$ is tight, where $\zeta^n$ is the initial data at initial time $\tau-t_n$.
	
	First, Lemma~\ref{UPe4-lem3.5} yields $\mathcal N_1=\mathcal N_1(\tau,D)\in\mathbb N$ such that, for all $n\geq\mathcal N_1$,
	\begin{align}\label{NSE-disghjmea5.8}
		\mathbb{E}\left[\Theta(\tau,\tau-2,v(\tau-2,\tau-t_n,\zeta^n))\| v(\tau,\tau-2,v(\tau-2,\tau-t_n,\zeta^n))\|^2\right]\leq \mathscr{M}_1,
	\end{align}
	where $\mathscr{M}_1$ depends only on $\tau$, and is independent of $\epsilon$, $n$ and ${D}$. 
	
	Moreover, Lemmas~\ref{UPe4-lem4.1} and~\ref{UPe4-lem4.2} yield
	 $\mathcal{N}_2=\mathcal{N}_2(\tau,D)\in \mathbb{N}$ and $\mathscr{M}_2:=\mathscr{M}_2(\tau)>0$ independent of $\epsilon$, $n$ and ${D}$ such that for all $n\geq \mathcal{N}_2$,
	 {\small
	\begin{align}\label{NSE-disghjmea5.9}
		\begin{split}
			&\int_{\tau-2}^\tau \mathbb{E}\left[|v(s,\tau-2,v(\tau-2,\tau-t_n,\zeta^n))|^2\|v(s,\tau-2,v(\tau-2,\tau-t_n,\zeta^n))\|^2\right]ds\\
			&+\int_{\tau-2}^\tau \mathbb{E}\left[|\partial_{z}v(s,\tau-2,v(\tau-2,\tau-t_n,\zeta^n))|^2\|\partial_{z}v(s,\tau-2,v(\tau-2,\tau-t_n,\zeta^n))\|^2\right]ds\\
			&\leq e^{2\kappa}\int_{\tau-2}^\tau e^{-\kappa(\tau-s)}\mathbb{E}\left[|v(s,\tau-2,v(\tau-2,\tau-t_n,\zeta^n))|^2\|v(s,\tau-2,v(\tau-2,\tau-t_n,\zeta^n))\|^2\right]ds\\
			&+e^{2\kappa}\int_{\tau-2}^\tau e^{-\kappa(\tau-s)}\mathbb{E}\left[|\partial_{z}v(s,\tau-2,v(\tau-2,\tau-t_n,\zeta^n))|^2\|\partial_{z}v(s,\tau-2,v(\tau-2,\tau-t_n,\zeta^n))\|^2\right]ds\\
			&\leq \mathscr{M}_2.
		\end{split}
	\end{align}}
	
%	For any $R>0$, let $\mathbb{B}_{V}(R):=\{u:\|u\|_{V}\leq R\}$. Since $V$ is compactly embedded in $H$, the set $\mathbb{B}_{V}(R)$ is compact in $H$.
	Markov's inequality together with \eqref{NSE-disghjmea5.8}--\eqref{NSE-disghjmea5.9} shows that, with $\mathcal N_3=\max\{\mathcal N_1,\mathcal N_2\}$, for every $n\geq\mathcal N_3$ and $R>2$, 
	\begin{align*}%\label{ojodfkpdfkp}
		%\begin{split}
		&~~\mathbb{P}\left(\|v(\tau,\tau-t_n,\zeta^n)\|^2>R\right)=\mathbb{P}\left(\|v(\tau,\tau-2,v(\tau-2,\tau-t_n,\zeta^n))\|^2>{R}\right)\notag \\
		&\leq  \mathbb{P}\left(\Theta (\tau,\tau-2,v(\tau-2, \tau-t_n, \zeta^n))\|v(\tau,\tau-2,v(\tau-2,\tau-t_n,\zeta^n))\|^2>{R}^{1/2}\right)\notag \\
		&+\mathbb{P}\left(\Theta^{-1}(\tau,\tau-2,v(\tau-2, \tau-t_n, \zeta^n))>{R}^{1/2}\right)\notag \\
		&\leq \frac{\mathbb{E}\left(\Theta(\tau,\tau-2,v(\tau-2, \tau-t_n, \zeta^n))\|v(\tau,\tau-2,v(\tau-2,\tau-t_n,\zeta^n))\|^2\right)}{R^{1/2}}\notag \\
		&+\mathbb{P}\bigg(\int^\tau_{\tau-2}\big(|v(s,\tau-2,v(\tau-2,\tau-t_n,\zeta^n))|^2\|v(s,\tau-2,v(\tau-2,\tau-t_n,\zeta^n))\|^2\notag \\
		&+|\partial_{z}v(s,\tau-2,v(\tau-2,\tau-t_n,\zeta^n))|^2\|\partial_{z}v(s,\tau-2,v(\tau-2,\tau-t_n,\zeta^n))\|^2\big)ds>\frac{\ln{R}}{2\alpha_2}\bigg)\notag \\
		&\leq \frac{\mathbb{E}\left[\Theta(\tau,\tau-2,v(\tau-2, \tau-t_n, \zeta^n))\|v(\tau,\tau-2,v(\tau-2, \tau-t_n, \zeta^n))\|_V^2\right]}{R^{1/2}}\notag \\
		& +\frac{2\alpha_2}{\ln{R}}\bigg(\int_{\tau-2}^\tau \mathbb{E}\left[|v(s,\tau-2,v(\tau-2,\tau-t_n,\zeta^n))|^2\|v(s,\tau-2,v(\tau-2,\tau-t_n,\zeta^n))\|^2\right]ds\notag \\
		&+\int_{\tau-2}^\tau \mathbb{E}\left[|\partial_{z}v(s,\tau-2,v(\tau-2,\tau-t_n,\zeta^n))|^2\|\partial_{z}v(s,\tau-2,v(\tau-2,\tau-t_n,\zeta^n))\|^2\right]ds\bigg)\notag \\
		&\leq \frac{\mathscr{M}_1}{R^{1/2}}+\frac{2\alpha_2\mathscr{M}_2}{\ln{R}},
		%\end{split}
	\end{align*}
	Hence, for each fixed $\tau\in\mathbb R$ and every $n\geq\mathcal N_3$,
	$$
	\lim\limits_{R\to \infty}\mathbb{P}\left(\|v(\tau,\tau-t_n,\zeta^n)\|^2>R\right)=0.
	$$
	
	Therefore, for any $\tau\in\mathbb R$ and $\varepsilon>0$, there exists $\widetilde R=\widetilde R(\tau,\varepsilon)>2$ such that for every
	$\zeta^n\in L^4(\Omega,\mathscr F_{\tau-t_n};\mathcal H)$ with $\mathscr L(\zeta^n)\in D(\tau-t_n)$ and $n\ge\mathcal N_3$,
	$$
	\mathbb{P}\left(\|v(\tau,\tau-t_n,\zeta^n)\|^2>\widetilde{R}\right)< \varepsilon.
	$$
	Since the closed $V$-ball is compact in $H$, the preceding estimate proves tightness of the terminal laws
	\[
	\hat{\mu}_n:=S_\epsilon(\tau,\tau-t_n)\mu_n.
	\]
	Let $\hat{\mu}_{n_k}\to\eta$ weakly in $\mathcal P(H)$. The pullback estimates in Lemmas~\ref{UPe4-lem4.1} and~\ref{UPe4-lem4.2} also give, after discarding finitely many indices,
	\[
	\sup_k\int_H\bigl(|u|^4+\Phi_z(u)\bigr)\hat{\mu}_{n_k}(du)<\infty.
	\]
	Lemma~\ref{lem-moment-lsc} therefore implies $\eta\in\mathcal P_{4,z}(H)\subset\mathcal P_4(H)$. Thus the limit remains in the domain of the measure evolution, as required for the subsequent $\omega$-limit construction.
	 This completes the proof.
\end{proof}

The following theorem is proved directly in the ambient Polish space $\mathcal P(H)$. This avoids assuming that the admissible class $\mathcal P_{4,z}(H)$ is complete in the weak topology.
\begin{theorem}\label{sDHD-th6.10}
	Suppose Hypotheses \ref{PEeqAssum-2.2}, \ref{PEeqAssum-2.3}, \ref{PEeqAssum-2.4} and \eqref{PETNs-3.01} hold. Then, for every $\epsilon\in(0,\epsilon_0)$, the measure evolution process $S_\epsilon$ on $\mathcal P_{4,z}(H)$ has a unique $\mathfrak D$-pullback measure attractor $\mathscr A_\epsilon$. For every $\tau\in\mathbb R$,
	\begin{align}\label{attractor-omega-formula}
		\mathscr A_\epsilon(\tau)
		=\bigcap_{s\ge0}
		\overline{\bigcup_{t\ge s}
			S_\epsilon(\tau,\tau-t)\mathscr K(\tau-t)}.
	\end{align}
	Moreover, $\mathscr A_\epsilon(\tau)\subset\mathscr K(\tau)$; in particular, every attractor section is a compact subset of $\mathcal P_{4,z}(H)$ in the ambient weak topology.
\end{theorem}
\begin{proof}
	Fix $\tau\in\mathbb R$ and write
	\[
	\mathscr{C}_s(\tau):=
	\overline{\bigcup_{t\ge s}S_\epsilon(\tau,\tau-t)
		\mathscr K(\tau-t)}.
	\]
	The sets $\mathscr{C}_s(\tau)$ are closed and decreasing in $s$. Choose $t_n\ge n$ and $\mu_n\in\mathscr K(\tau-t_n)$. Lemma~\ref{UPe4-lemcompatight} gives a subsequence of $S_\epsilon(\tau,\tau-t_n)\mu_n$ converging in $\mathcal P(H)$. Its limit belongs to every $\mathscr{C}_s(\tau)$, so the set in \eqref{attractor-omega-formula} is nonempty.
	
	Since $\mathscr K\in\mathfrak D$, the absorbing property applied to $D=\mathscr K$ gives $T_\tau>0$ such that
	\[
	S_\epsilon(\tau,\tau-t)\mathscr K(\tau-t)
	\subset\mathscr K(\tau),\qquad t\ge T_\tau.
	\]
	The set $\mathscr K(\tau)$ is closed in $\mathcal P(H)$; hence
	$\mathscr A_\epsilon(\tau)\subset\mathscr K(\tau)$. In particular, the attractor candidates are admissible and the family $\mathscr A_\epsilon$ belongs to $\mathfrak D$.
	
	We next prove compactness. Let $\upsilon_n\in\mathscr A_\epsilon(\tau)$. From the definition of $\mathscr{C}_n(\tau)$, select $t_n\ge n$ and $\mu_n\in\mathscr K(\tau-t_n)$ such that
	\[
	d_{\mathcal P(H)}
	\bigl(\upsilon_n,S_\epsilon(\tau,\tau-t_n)\mu_n\bigr)<n^{-1}.
	\]
	Pullback asymptotic compactness gives a convergent subsequence of the second terms $S_\epsilon(\tau,\tau-t_n)\mu_n$, and therefore the corresponding subsequence of $\upsilon_n$ converges. Since $\mathscr A_\epsilon(\tau)$ is closed, it is compact.
	
	To prove pullback attraction, suppose to the contrary that for some $D\in\mathfrak D$, $\delta>0$, $t_n\to\infty$, and $\mu_n\in D(\tau-t_n)$,
	\[
	d_{\mathcal P(H)}\bigl(S_\epsilon(\tau,\tau-t_n)\mu_n,
	\mathscr A_\epsilon(\tau)\bigr)\ge\delta.
	\]
	By Lemma~\ref{UPe4-lemcompatight}, after taking a subsequence,
	$S_\epsilon(\tau,\tau-t_n)\mu_n\to \upsilon$ in $\mathcal P(H)$. Fix $s>0$. For all sufficiently large $n$, the absorbing property at time $\tau-s$ gives
	\[
	S_\epsilon(\tau-s,\tau-t_n)\mu_n\in\mathscr K(\tau-s).
	\]
	By the process identity,
	\[
	S_\epsilon(\tau,\tau-t_n)\mu_n
	\in S_\epsilon(\tau,\tau-s)\mathscr K(\tau-s)
	\subset \mathscr{C}_s(\tau).
	\]
	Thus $\upsilon\in \mathscr{C}_s(\tau)$ for every $s>0$, and hence $\upsilon\in\mathscr A_\epsilon(\tau)$, a contradiction.
	
	We finally establish invariance. Let $r\ge\tau$. If $\upsilon\in\mathscr A_\epsilon(\tau)$, choose $t_n\to\infty$ and $\mu_n\in\mathscr K(\tau-t_n)$ such that
	$S_\epsilon(\tau,\tau-t_n)\mu_n\to \upsilon$. By continuity and the process identity,
	\[
	S_\epsilon(r,\tau)\upsilon
	=\lim_{n\to\infty}S_\epsilon(r,\tau-t_n)\mu_n
	\in\mathscr A_\epsilon(r),
	\]
	so $S_\epsilon(r,\tau)\mathscr A_\epsilon(\tau)
	\subset\mathscr A_\epsilon(r)$. Conversely, let $\hat{\upsilon}\in\mathscr A_\epsilon(r)$ and choose $t_n\to\infty$, with $r-t_n<\tau$, and $\mathfrak{a}_n\in\mathscr K(r-t_n)$ such that $S_\epsilon(r,r-t_n)\mathfrak{a}_n\to \hat{\upsilon}$. Set
	\[
	\upsilon_n:=S_\epsilon(\tau,r-t_n)\mathfrak{a}_n.
	\]
	Pullback asymptotic compactness at time $\tau$ gives a subsequence $\upsilon_n\to \upsilon\in\mathscr A_\epsilon(\tau)$. Continuity then yields $S_\epsilon(r,\tau)\upsilon=\hat{\upsilon}$, proving the reverse inclusion.
	
	If $\widetilde{\mathscr A}$ is another $\mathfrak D$-pullback measure attractor, its attraction of $\mathscr K$ and closedness imply $\mathscr A_\epsilon(\tau)\subset\widetilde{\mathscr A}(\tau)$. Conversely, invariance of $\widetilde{\mathscr A}$ and pullback attraction by $\mathscr A_\epsilon$ give
	\[
	\operatorname{dist}_{w,H}
	\bigl(\widetilde{\mathscr A}(\tau),\mathscr A_\epsilon(\tau)\bigr)
	=\lim_{t\to\infty}
	\operatorname{dist}_{w,H}
	\bigl(S_\epsilon(\tau,\tau-t)\widetilde{\mathscr A}(\tau-t),
	\mathscr A_\epsilon(\tau)\bigr)=0.
	\]
	Since $\mathscr A_\epsilon(\tau)$ is closed, the reverse inclusion follows. This proves uniqueness and, together with Definition~\ref{def-admissible-pullback-attractor}, completes the proof.
\end{proof}

\section{Zero-noise limit of pullback measure attractors}\label{USC-PEs4}
In this section, we study the upper semicontinuity of the pullback measure attractors as the Gaussian and jump intensities vanish simultaneously. 
To this end, we denote by $v^{\epsilon}(t,\tau,\zeta)$ the solution of \eqref{PEeq-2.15}. When $\epsilon=0$, the equation \eqref{PEeq-2.15} reduces to the deterministic nonautonomous primitive equation
\begin{equation}\label{PEeq-USCdet}
	\frac{dv^0}{dt}+Av^0+B(v^0,v^0)=g(t),
	\qquad v^0(\tau)=\zeta\in\mathcal H.
\end{equation}
%Let $U_0(t,\tau)\zeta=v^0(t,\tau,\zeta)$ and define the induced process on measures by
%\begin{equation}\label{PEeq-USCS0}
%	S_0(t,\tau)\mu=(U_0(t,\tau))_\#\mu,
%	\qquad \mu\in\mathcal P_{4}(H).
%\end{equation}
\begin{remark}\label{deter1-dd}
	For the deterministic problem \eqref{PEeq-USCdet}, the existing deterministic well-posedness theory (see \cite{Petcu-CPAA-2004,Bresch-SIMA-2004}) yields a unique global solution \(v^0\) for every $\zeta\in\mathcal H$. Moreover, for any $\tau\in\mathbb R$ and $T>0$,
	\begin{align}\label{ee-v05}
	\sup_{\tau\le t\le\tau+T}
	\left(|v^0(t)|^2+|\partial_zv^0(t)|^2\right)
	+\int_\tau^{\tau+T}
	\left(\|v^0(t)\|^2+\|\partial_zv^0(t)\|^2\right)dt
	\le C_T,
	\end{align}
	where $C_T$ depends on $T$, $|\zeta|$, $|\partial_z\zeta|$, and the corresponding norms of $g$ and $\partial_zg$. No deterministic $H$-to-$\mathcal H$ regularization statement is used.
	
%	we know from  that problem \eqref{PEeq-USCdet} admits a unique global solution \( v^0  \in C([0,T]; H) \cap L^2([0,T]; V) \). Moreover, for any $\tau\in\mathbb R$ and $T>0$, there exists
%	\[
%	C_T=C\!\left(T,|\zeta|,|\partial_z\zeta|,
%	\|g\|_{L^2([0,T];H)},
%	\|\partial_zg\|_{L^2([0,T];H)}\right)>0
%	\]
%	such that
%	\begin{align}\label{ee-v05}
%		\sup_{\tau\leq t\leq \tau+T}\left(|v^0(t)|^{2}+|\partial_z v^0(t)|^{2}\right)+\int^{\tau+T}_{\tau}\left(\|v^0(t)\|^{2}+\|\partial_z v^0(t)\|^{2}\right)dt\leq C_{T}.
%	\end{align}

\end{remark}

To highlight the parameter dependence, let $S_{\epsilon}$ and $S_0$ be the measure evolution processes generated by \eqref{PEeq-2.15} and \eqref{PEeq-USCdet}, respectively, and denote by $\mathscr A_\epsilon$ and $\mathscr A_0$ their $\mathfrak{D}$-pullback measure attractors. Since the conclusions of Section~\ref{PMA-PEs3} apply also to $\epsilon=0$, $S_0$ admits a unique $\mathfrak D$-pullback measure attractor $\mathscr A_0$. Furthermore, the constants in Lemmas~\ref{UPe4-lem4.1}--\ref{UPe4-lem3.5} may be chosen uniformly in $\epsilon\in[0,\epsilon_0)$; hence the family $\mathscr K$ from Lemma~\ref{UPe4-lemABs} serves as a common closed pullback absorbing family for all these processes.

We first prove the uniform convergence on finite time intervals of solutions of \eqref{PEeq-2.15} to those of \eqref{PEeq-USCdet} as \(\epsilon \to 0\).

The deterministic counterpart of the estimates in Theorem~\ref{PEeqthe-2.5} will be used repeatedly. In particular, for every $R,T>0$,
\begin{align}\label{det-uniform-fourth-estimate}
	\sup_{\mathbb E[|\zeta|^4+|\partial_z\zeta|^4]\le R}
	\mathbb E\Bigg[&\sup_{t\in[\tau,\tau+T]}
	\bigl(|v^0(t)|^4+|\partial_zv^0(t)|^4\bigr)\\
	&+\left(\int_\tau^{\tau+T}
	\bigl(\|v^0(t)\|^2+\|\partial_zv^0(t)\|^2\bigr)dt\right)^2\Bigg]
	\le C_{R,T}.
\end{align}
Indeed, this is obtained by repeating the proof of \eqref{ee-5}-\eqref{ee-6} with the stochastic terms deleted. No additional regularization assumption is involved.

\begin{lemma}\label{PEeq-USCfinite}
		Suppose Hypotheses \ref{PEeqAssum-2.2}, \ref{PEeqAssum-2.3}, \ref{PEeqAssum-2.4} and \eqref{PETNs-3.01} hold. Let $\tau\in\mathbb{R}$, $T>0$, and let $\mathfrak{B}(\tau)>0$ be a given constant. Then, for every $\zeta\in L^4(\Omega,\mathscr F_\tau;\mathcal H)$ satisfying $\mathbb E[|\zeta|^4+|\partial_z\zeta|^4]\leq\mathfrak B^4(\tau)$,
	\begin{align}\label{PEeq-USCconv}
		\lim_{\epsilon\to0}\sup_{\substack{\mathbb E\left[|\zeta|^4+|\partial_z\zeta|^4\right]\leq \mathfrak{B}^4(\tau)}}
		\mathbb E\left[\sup_{t\in[\tau,\tau+T]}
		|v^\epsilon(t,\tau,\zeta)-v^0(t,\tau,\zeta)|^2\right]=0.
	\end{align}
	Consequently, for the admissible moment ball
	\[
	\mathbb B_{4,z}(\mathfrak B(\tau))
	:=
	\left\{\mu\in\mathcal P_{4,z}(H):
	\int_H\bigl(|\zeta|^4+|\partial_z\zeta|^4\bigr)\mu(d\zeta)
	\le\mathfrak B^4(\tau)\right\},
	\]
	and every $t\in[\tau,\tau+T]$,
	\begin{align}\label{PEeq-USCmeasureconv}
	\lim_{\epsilon\to0}
	\sup_{\mu\in\mathbb B_{4,z}(\mathfrak B(\tau))}
	d_{\mathcal P(H)}
	\bigl(S_\epsilon(t,\tau)\mu,S_0(t,\tau)\mu\bigr)=0.
	\end{align}
\end{lemma}
\begin{proof}
%	Fix a deterministic initial value $\zeta\in\mathcal H$ satisfying
%	$|\zeta|^4+|\partial_z\zeta|^4\le R$, and 
 We	write
	$Z^\epsilon(t):=v^\epsilon(t,\tau,\zeta)-v^0(t,\tau,\zeta)$.
	Then $Z^\epsilon$ satisfies
	\begin{align}\label{PEeq-USCdiff}
		dZ^\epsilon+AZ^\epsilon dt
		+\bigl(B(v^\epsilon,v^\epsilon)-B(v^0,v^0)\bigr)dt
		=\sqrt\epsilon\,\sigma(t,v^\epsilon)dW(t)+\epsilon\int_E h(v^\epsilon(t-),\xi)
		\widetilde N^{\epsilon^{-1}}(dt,d\xi)
	\end{align}
	with zero initial value.
	Applying \eqref{PEeq-2.103} with $\widetilde v=v^0$ and $\widehat v=v^\epsilon$, together with Young's inequality, yields
	\begin{align}\label{PEeq-USCnonlinear}
		2\left|\left\langle B(v^\epsilon,v^\epsilon)-B(v^0,v^0),Z^\epsilon\right\rangle\right|
		\le \frac12\|Z^\epsilon\|^2
		+C\bigl(\|v^0\|^2+|\partial_zv^0|^4\bigr)|Z^\epsilon|^2.
	\end{align}
	Choose $C_0$ larger than the constant in \eqref{PEeq-USCnonlinear}. Applying It\^{o}'s formula for jump processes directly to
	\[
	e^{-C_0\int_\tau^t
		\bigl(1+\|v^0(r)\|^2+|\partial_zv^0(r)|^4\bigr)dr}|Z^\epsilon(t)|^2
	\]
	and using \textbf{(A.1)} and \textbf{(B.1)}, we get
	\begin{align}\label{PEeq-USCweighted}
		&~~e^{-C_0\int_\tau^t(1+\|v^0(r)\|^2+|\partial_zv^0(r)|^4)dr}|Z^\epsilon(t)|^2\notag+\frac12\int_\tau^t
		e^{-C_0\int_\tau^s(1+\|v^0(r)\|^2+|\partial_zv^0(r)|^4)dr}
		\|Z^\epsilon(s)\|^2ds\notag\\
		&\le C\epsilon\int_\tau^t
		e^{-C_0\int_\tau^s(1+\|v^0(r)\|^2+|\partial_zv^0(r)|^4)dr}
		\bigl(1+|v^\epsilon(s)|^2\bigr)ds\notag\\
		&+2\sqrt\epsilon\int_\tau^t
		e^{-C_0\int_\tau^s(1+\|v^0(r)\|^2+|\partial_zv^0(r)|^4)dr}
		\bigl(Z^\epsilon(s),\sigma(s,v^\epsilon(s))dW(s)\bigr)\\
		&+\int_\tau^t\int_E
		e^{-C_0\int_\tau^s(1+\|v^0(r)\|^2+|\partial_zv^0(r)|^4)dr}
		\Bigl[2\epsilon\bigl(Z^\epsilon(s-),h(v^\epsilon(s-),\xi)\bigr)
		+\epsilon^2|h(v^\epsilon(s-),\xi)|^2\Bigr]
		\widetilde N^{\epsilon^{-1}}(ds,d\xi),\notag
	\end{align}
where the drift term in the first integral on the right-hand side consists of two contributions, namely,
	\[
	\epsilon\|\sigma(s,v^\epsilon(s))\|_{\mathcal L_2(U;H)}^2
	\quad\text{and}\quad
	\epsilon\int_E|h(v^\epsilon(s),\xi)|^2\nu(d\xi).
	\]
	
	For $\Bbbk>0$, we introduce the stopping time
	\begin{equation}\label{PEeq-USCstop}
		\tau_\Bbbk=\inf\left\{t\ge\tau:
		\int_\tau^t\bigl(1+\|v^0(s)\|^2+|\partial_zv^0(s)|^4\bigr)ds>\Bbbk\right\}\wedge (\tau+T).
	\end{equation}
	For the second term on the right-hand side of \eqref{PEeq-USCweighted}, the BDG and Young inequalities give
	\begin{align*}
		&~~2\sqrt\epsilon\,\mathbb E\left[\sup_{\tau\le r\le t\wedge\tau_\Bbbk}
		\left|\int_\tau^r
		e^{-C_0\int_\tau^s(1+\|v^0(\rho)\|^2+|\partial_zv^0(\rho)|^4)d\rho}
		\bigl(Z^\epsilon(s),\sigma(s,v^\epsilon(s))dW(s)\bigr)\right|\right]\\
		&\le\frac18\mathbb E\left[\sup_{\tau\le r\le t\wedge\tau_\Bbbk}
		e^{-C_0\int_\tau^r(1+\|v^0(\rho)\|^2+|\partial_zv^0(\rho)|^4)d\rho}|Z^\epsilon(r)|^2\right]
		+C\epsilon\mathbb E\left[\int_\tau^{\tau+T}(1+|v^\epsilon(s)|^2)ds\right].
	\end{align*}
	For the third term on the right-hand side of \eqref{PEeq-USCweighted}, the BDG inequality for Poisson integrals yields
	{\footnotesize
		\begin{align*}
			&\mathbb E\left[\sup_{\tau\le r\le t\wedge\tau_\Bbbk}\left|\int_\tau^r\int_E
			e^{-C_0\int_\tau^s(1+\|v^0(\rho)\|^2+|\partial_zv^0(\rho)|^4)d\rho}
			\Bigl[2\epsilon(Z^\epsilon(s-),h(v^\epsilon(s-),\xi))
			+\epsilon^2|h(v^\epsilon(s-),\xi)|^2\Bigr]
			\widetilde N^{\epsilon^{-1}}(ds,d\xi)\right|\right]\\
			&\quad\le C\mathbb E\left[\left(\int_\tau^{t\wedge\tau_\Bbbk}\int_E
			e^{-2C_0\int_\tau^s(1+\|v^0(\rho)\|^2+|\partial_zv^0(\rho)|^4)d\rho}
			\Bigl(\epsilon|Z^\epsilon(s)|^2|h(v^\epsilon(s),\xi)|^2
			+\epsilon^3|h(v^\epsilon(s),\xi)|^4\Bigr)\nu(d\xi)ds\right)^{1/2}\right]\\
			&\quad\le\frac{1}{8}\mathbb E\left[\sup_{\tau\le r\le t\wedge\tau_\Bbbk}
			e^{-C_0\int_\tau^r(1+\|v^0(\rho)\|^2+|\partial_zv^0(\rho)|^4)d\rho}|Z^\epsilon(r)|^2\right]\\
			&\qquad+C\epsilon\mathbb E\left[\int_\tau^{\tau+T}(1+|v^\epsilon(s)|^2)ds\right]
			+C\epsilon^{3/2}\mathbb E\left[\left(\int_\tau^{\tau+T}\int_E|h(v^\epsilon(s),\xi)|^4\nu(d\xi)ds\right)^{1/2}\right]\\
			&\quad\le\frac{1}{8}\mathbb E\left[\sup_{\tau\le r\le t\wedge\tau_\Bbbk}
			e^{-C_0\int_\tau^r(1+\|v^0(\rho)\|^2+|\partial_zv^0(\rho)|^4)d\rho}|Z^\epsilon(r)|^2\right]
			+C_{\mathfrak{B}(\tau),T}\epsilon.
	\end{align*}}
	where the last inequality uses $C_h\in L^2(\nu)\cap L^4(\nu)$, $\epsilon\le 1$, and the fourth-moment estimate from Theorem \ref{PEeqthe-2.5}. Substituting the above estimates into \eqref{PEeq-USCweighted} yields
	\begin{align}\label{PEeq-USCstopped}
		&~~\mathbb E\left[\sup_{\tau\le r\le t\wedge\tau_\Bbbk}
		e^{-C_0\int_\tau^r(1+\|v^0(\rho)\|^2+|\partial_zv^0(\rho)|^4)d\rho}|Z^\epsilon(r)|^2\right]\notag\\
		&\quad+\mathbb E\left[\int_\tau^{t\wedge\tau_\Bbbk}
		e^{-C_0\int_\tau^s(1+\|v^0(\rho)\|^2+|\partial_zv^0(\rho)|^4)d\rho}
		\|Z^\epsilon(s)\|^2ds\right]\notag\\
		&\le\frac12\mathbb E\left[\sup_{\tau\le r\le t\wedge\tau_\Bbbk}
		e^{-C_0\int_\tau^r(1+\|v^0(\rho)\|^2+|\partial_zv^0(\rho)|^4)d\rho}|Z^\epsilon(r)|^2\right]
		+C_{\mathfrak{B}(\tau),T}\epsilon.
	\end{align}
	Since the exponential factor is bounded below by $e^{-C_0\Bbbk}$ before $\tau_\Bbbk$, we infer
	\begin{equation}\label{PEeq-USCstopped2}
		\sup_{\mathbb E\left[|\zeta|^4+|\partial_z\zeta|^4\right]\leq \mathfrak{B}^4(\tau)}
		\mathbb E\left[\sup_{r\in[\tau,\tau+T]}
		|Z^\epsilon(r\wedge\tau_\Bbbk)|^2\right]
		\le C_{\Bbbk,\mathfrak{B}(\tau),T}\epsilon.
	\end{equation}
	
We next remove the stopping time. By \eqref{det-uniform-fourth-estimate}, we have
\[
\sup_{\mathbb E(|\zeta|^4+|\partial_z\zeta|^4)\le\mathfrak B^4(\tau)}
\mathbb E\left[\int_\tau^{\tau+T}
\bigl(1+\|v^0(s)\|^2+|\partial_zv^0(s)|^4\bigr)ds\right]
\le C_{\mathfrak B(\tau),T}.
\]
Applying Markov's inequality gives
\begin{align}\label{PEeq-USCstoptail}
	\begin{split}
		&~~\sup_{\mathbb E\left[|\zeta|^4+|\partial_z\zeta|^4\right]
		\leq \mathfrak{B}^4(\tau)}
		\mathbb P(\tau_\Bbbk<\tau+T)\\
		&\le\frac{1}{\Bbbk}\sup_{\substack{\mathbb E(|\zeta|^4+|\partial_z\zeta|^4)\le\mathfrak B^4(\tau)}}\mathbb E\left[\int_\tau^{\tau+T}
		\bigl(1+\|v^0(s)\|^2+|\partial_zv^0(s)|^4\bigr)ds\right]
		\le\frac{C_{\mathfrak{B}(\tau),T}}{\Bbbk}.
	\end{split}
\end{align}
Moreover, $|v^\epsilon-v^0|^4\le8(|v^\epsilon|^4+|v^0|^4)$, and Theorem~\ref{PEeqthe-2.5} together with \eqref{det-uniform-fourth-estimate} implies
	\begin{equation}\label{PEeq-USCdiff4}
		\sup_{\epsilon\in[0,\epsilon_0)}
		\sup_{\mathbb E\left[|\zeta|^4+|\partial_z\zeta|^4\right]\leq \mathfrak{B}^4(\tau)}
		\mathbb E\left[\sup_{t\in[\tau,\tau+T]}|Z^\epsilon(t)|^4\right]
		\le C_{\mathfrak{B}(\tau),T}.
	\end{equation}
	Therefore, by Cauchy-Schwarz we have
	\begin{align*}
		\mathbb E\left[\sup_{t\in[\tau,\tau+T]}|Z^\epsilon(t)|^2\right]
		&\le\mathbb E\left[\sup_{t\in[\tau,\tau+T]}|Z^\epsilon(t\wedge\tau_\Bbbk)|^2\right]+\left(\mathbb E\left[\sup_{t\in[\tau,\tau+T]}|Z^\epsilon(t)|^4\right]\right)^{1/2}
		\mathbb P(\tau_\Bbbk<{\tau+T})^{1/2}.
	\end{align*}
	Combining \eqref{PEeq-USCstopped2}-\eqref{PEeq-USCdiff4}, and passing to the limits first $\epsilon\to0$ and then $\Bbbk\to\infty$, gives \eqref{PEeq-USCconv}.
	
	We now lift this convergence to probability measures. For $t\in [\tau,\tau+T]$,
	\[
	\begin{aligned}
		&\sup_{\mathbb{E}\left[|\zeta|^4+|\partial_z\zeta|^4\right] \leq \mathfrak{B}^4(\tau)} 
		\sup_{\substack{\varphi \in L_b(H) \\ \|\varphi\|_{L_b} \leq 1}} 
		\left|\mathbb{E}\left[\varphi(v^\varepsilon(t, \tau, \zeta))\right]
		- \mathbb{E}\left[\varphi(v^0(t, \tau, \zeta))\right]\right| \\
		&\leq \sup_{\mathbb{E}\left[|\zeta|^4+|\partial_z\zeta|^4\right] \leq \mathfrak{B}^4(\tau)} 
		\sup_{\substack{\varphi \in L_b(H) \\ \|\varphi\|_{L_b} \leq 1}} 
		\left|\mathbb{E}\left[\varphi(v^\varepsilon(t, \tau, \zeta))
		- \varphi(v^0(t, \tau, \zeta))\right]\right| \\
		&\leq \sup_{\mathbb{E}\left[|\zeta|^4+|\partial_z\zeta|^4\right] \leq \mathfrak{B}^4(\tau)} 
		\mathbb{E}\left[|v^\varepsilon(t, \tau, \zeta) - v^0(t, \tau, \zeta)|\right] \\
		&\leq \left( \sup_{\mathbb{E}\left[|\zeta|^4+|\partial_z\zeta|^4\right] \leq \mathfrak{B}^4(\tau)} 
		\mathbb{E}\left[|v^\varepsilon(t, \tau, \zeta) - v^0(t, \tau, \zeta)|^2\right]\right)^{1/2}.
	\end{aligned}
	\]
	For $\mu\in\mathbb B_{4,z}(\mathfrak B(\tau))$, choose an $\mathcal H$-valued random variable $\zeta$ with law $\mu$, independent of the future Wiener process and Poisson random measure, and drive $v^\epsilon$ and $v^0$ by this same initial variable. Then
	$\mathbb E(|\zeta|^4+|\partial_z\zeta|^4)\le\mathfrak B^4(\tau)$, so the preceding estimate applies uniformly in $\mu$. Taking the supremum over $\mu$ gives \eqref{PEeq-USCmeasureconv}. This completes the proof.
\end{proof}

The convergence below is upper semicontinuity in the weak-measure semidistance \eqref{weak-Hausdorff-distance}; it is not a fourth-order Wasserstein convergence.
\begin{theorem}\label{PEeq-USCmain}
	Suppose Hypotheses \ref{PEeqAssum-2.2}, \ref{PEeqAssum-2.3}, \ref{PEeqAssum-2.4} and \eqref{PETNs-3.01} hold. Then, for every $\tau\in\mathbb R$,
	\begin{align}\label{PEeq-USCresult}
		\lim_{\epsilon\to0}
		\operatorname{dist}_{w,H}
		\bigl(\mathscr A_\epsilon(\tau),\mathscr A_0(\tau)\bigr)=0.
	\end{align}
\end{theorem}
\begin{proof}
	By Lemma \ref{UPe4-lemABs}, the family
	\[
	\mathscr K=\{\mathscr K(\tau):\tau\in\mathbb R\}\in\mathfrak D
	\]
	is a common closed $\mathfrak D$-pullback absorbing family for $S_\epsilon$, uniformly for $\epsilon\in[0,\epsilon_0)$. More precisely, for every $D=\{D(t):t\in\mathbb R\}\in\mathfrak D$ and $\tau\in\mathbb R$, there exists $T=T(\tau,D)>0$, independent of $\epsilon$, such that
	\begin{align}\label{PEeq-USCcommonabs}
	\bigcup_{t\ge T}\ \bigcup_{\epsilon\in[0,\epsilon_0)}
	S_\epsilon(\tau,\tau-t)D(\tau-t)
	\subseteq \mathscr K(\tau).
	\end{align}
The direct construction in Theorem~\ref{sDHD-th6.10} already gives
\begin{align}\label{PEeq-USCattractorK}
	\mathscr A_\epsilon(\tau)\subset\mathscr K(\tau),
	\qquad \tau\in\mathbb R,\quad\epsilon\in[0,\epsilon_0).
\end{align}
This inclusion does not require an absorption time uniform with respect to the varying family $\mathscr A_\epsilon$; it follows sectionwise from the $\omega$-limit formula and the fact that $\mathscr K$ absorbs itself.
	
	Let $\delta>0$ be arbitrary. Since $\mathscr K\in\mathfrak D$ and $\mathscr A_0$ is the $\mathfrak D$-pullback measure attractor of $S_0$, there exists $T=T(\delta,\tau,\mathscr K)>0$ such that
	\begin{align}\label{PEeq-USCdetatt}
	\sup_{\mu\in\mathscr K(\tau-T)}
	d_{\mathcal P(H)}
	\bigl(S_0(\tau,\tau-T)\mu,\mathscr A_0(\tau)\bigr)
	<\frac{\delta}{2}.
	\end{align}
	For this fixed $T$, Lemma \ref{PEeq-USCfinite}, applied at the initial time $\tau-T$, gives an $\epsilon_\delta\in(0,\epsilon_0)$ such that, whenever $0<\epsilon<\epsilon_\delta$,
	\begin{align}\label{PEeq-USCunifprocess}
	\sup_{\mu\in\mathscr K(\tau-T)}
	d_{\mathcal P(H)}
	\bigl(S_\epsilon(\tau,\tau-T)\mu,
	S_0(\tau,\tau-T)\mu\bigr)
	<\frac{\delta}{2}.
	\end{align}
	Here we have used that
	\[
	\mathscr K(\tau-T)\subset
	\mathbb B_{4,z}\!\left(
	\bigl(\mathscr M_0(\tau-T)+\mathscr M_z(\tau-T)\bigr)^{1/4}
	\right),
	\]
	so that the uniform finite-time convergence in Lemma~\ref{PEeq-USCfinite}
	applies on $\mathscr K(\tau-T)$.
	
	By the invariance of $\mathscr A_\epsilon$ and \eqref{PEeq-USCattractorK},
	\[
	\mathscr A_\epsilon(\tau)
	=S_\epsilon(\tau,\tau-T)\mathscr A_\epsilon(\tau-T),
	\qquad
	\mathscr A_\epsilon(\tau-T)\subseteq\mathscr K(\tau-T).
	\]
	Therefore, for every $0<\epsilon<\epsilon_\delta$, the triangle inequality, \eqref{PEeq-USCdetatt}, and \eqref{PEeq-USCunifprocess} imply
	\begin{align*}
	\operatorname{dist}_{w,H}
	\bigl(\mathscr A_\epsilon(\tau),\mathscr A_0(\tau)\bigr)
	&=\sup_{\mu\in\mathscr A_\epsilon(\tau-T)}
	d_{\mathcal P(H)}
	\bigl(S_\epsilon(\tau,\tau-T)\mu,
	\mathscr A_0(\tau)\bigr)\\
	&\le
	\sup_{\mu\in\mathscr A_\epsilon(\tau-T)}
	d_{\mathcal P(H)}
	\bigl(S_\epsilon(\tau,\tau-T)\mu,
	S_0(\tau,\tau-T)\mu\bigr)\\
	&\quad+
	\sup_{\mu\in\mathscr A_\epsilon(\tau-T)}
	d_{\mathcal P(H)}
	\bigl(S_0(\tau,\tau-T)\mu,
	\mathscr A_0(\tau)\bigr)\\
	&\le
	\sup_{\mu\in\mathscr K(\tau-T)}
	d_{\mathcal P(H)}
	\bigl(S_\epsilon(\tau,\tau-T)\mu,
	S_0(\tau,\tau-T)\mu\bigr)\\
	&\quad+
	\sup_{\mu\in\mathscr K(\tau-T)}
	d_{\mathcal P(H)}
	\bigl(S_0(\tau,\tau-T)\mu,
	\mathscr A_0(\tau)\bigr)
	<\delta.
	\end{align*}
	Since $\delta>0$ is arbitrary, \eqref{PEeq-USCresult} follows. This completes the proof.
\end{proof}

\section{Moderate deviation principle}\label{MDP-PEs4}
In this section, we consider the MDP for the following stochastic primitive equation with L\'{e}vy noise:
\begin{eqnarray}\label{PEeq-MDP2.15}
	\left\{
	\begin{array}{ll}
		dv^{\epsilon}(t)+Av^{\epsilon}(t)dt+B(v^{\epsilon}(t),v^{\epsilon}(t))dt=g(t)dt+{\epsilon}\int_{E}h(v^{\epsilon}(t-),\xi)\widetilde{N}^{\epsilon^{-1}}(dt,d\xi),\\
		v^{\epsilon}(0)=v_0\in \mathcal{H},
	\end{array}
	\right.
\end{eqnarray}
As $\epsilon\to0$, the deterministic limit of \eqref{PEeq-MDP2.15} is the following primitive equation:
\begin{eqnarray}\label{PEeq-MDPdet2.15}
	\left\{
	\begin{array}{ll}
		dv^{0}(t)+Av^{0}(t)dt+B(v^{0}(t),v^{0}(t))dt=g(t)dt,\\
		v^{0}(0)=v_0\in \mathcal{H},
	\end{array}
	\right.
\end{eqnarray}
We study the asymptotic behavior of the rescaled fluctuation \eqref{PEeq-MDPdetdec4.3}. Under suitable conditions on $h$ and the assumption \eqref{PEeq-MDPdetdec4.4}, we establish a MDP for the solutions of \eqref{PEeq-MDP2.15}.

In the purely jump equation \eqref{PEeq-MDP2.15} we have $\sigma\equiv0$, so Hypothesis~\ref{PEeqAssum-2.2} is trivially satisfied. Hence, under Hypothesis~\ref{PEeqAssum-2.3}, Theorem~\ref{PEeqthe-2.5} (see also \cite{SunGao-2013,ZRR-2021}) yields a unique probabilistically strong solution $v^{\epsilon}\in \mathcal{D}([0,T],H)\cap L^2([0,T],V)$ of \eqref{PEeq-MDP2.15}. In particular, $\mathcal{Y}^{\epsilon}$ satisfies
\begin{align}
	\begin{split}
		&d\mathcal{Y}^{\epsilon}(t)+A\mathcal{Y}^{\epsilon}(t)dt+B(\mathcal{Y}^{\epsilon}(t),v^{0}(t))dt+B(a(\epsilon)\mathcal{Y}^{\epsilon}(t)+v^0(t),\mathcal{Y}^{\epsilon}(t))dt\\
		&=\frac{\epsilon}{a(\epsilon)}\int_{E}h(a(\epsilon)\mathcal{Y}^{\epsilon}(t-)+v^0(t-),\xi)\widetilde{N}^{\epsilon^{-1}}(dt,d\xi),
	\end{split}
\end{align}
with initial value $\mathcal{Y}^{\epsilon}(0)=0$. Lemma~\ref{lem:measurable-solution-maps} below constructs the Borel solution map
$\mathscr G^\epsilon:\mathfrak M\to\mathfrak U$ required by the variational representation.

Throughout this section, we use the state space
\begin{align}\label{PEeq-MDP-state-space}
	\mathfrak U:=\mathcal D([0,T],H)\cap L^2(0,T;V)
\end{align}
with metric
\begin{align}\label{PEeq-MDP-state-metric}
	d_{\mathfrak U}(u_1,u_2)
	:=d_{J_1,H}(u_1,u_2)+\|u_1-u_2\|_{L^2(0,T;V)}.
\end{align}
\begin{lemma}\label{lem:MDP-path-space-Polish}
	The metric space $(\mathfrak U,d_{\mathfrak U})$ is Polish.
\end{lemma}
\begin{proof}
	Both $\mathcal D([0,T];H)$ with the $J_1$ metric and $L^2(0,T;V)$ are Polish. Embed $\mathfrak U$ diagonally into their product. Suppose that $u_n\to u$ in the $J_1$ topology of $\mathcal D([0,T];H)$ and that $u_n\to v$ strongly in $L^2(0,T;V)$. By the continuous embedding $V\hookrightarrow H$, the second convergence also holds strongly in $L^2(0,T;H)$. Hence, after passing to a subsequence, 
	\[
	 u_n(t)\longrightarrow v(t)\quad\text{in }H\quad\text{for a.e. }t\in[0,T].
	\]
	On the other hand, $J_1$ convergence implies $u_n(t)\to u(t)$ in $H$ at every continuity point of the c\`adl\`ag path $u$. Since such a path has at most countably many discontinuities, the two pointwise limits agree for a.e.~$t$, and therefore $u=v$ in $L^2(0,T;H)$. Thus the diagonal image of $\mathfrak U$ is closed in $\mathcal D([0,T];H)\times L^2(0,T;V)$. A closed subspace of a Polish product is Polish, which proves the claim.
\end{proof}

For the measurable reconstruction in Subsection~\ref{Cond4.1-2}, we also use the enhanced path space
\[
\mathfrak U_z
:=
\left\{u\in\mathfrak U:\partial_z u\in\mathfrak U\right\},
\qquad
d_{\mathfrak U_z}(u_1,u_2)
:=
d_{\mathfrak U}(u_1,u_2)
+d_{\mathfrak U}(\partial_z u_1,\partial_z u_2),
\]
where $\partial_z$ is understood in the distributional sense. The distributional vertical derivative is a closed operator. Hence the map
$u\mapsto(u,\partial_z u)$ identifies $\mathfrak U_z$ with a closed subset of
$\mathfrak U\times\mathfrak U$, and therefore $(\mathfrak U_z,d_{\mathfrak U_z})$ is Polish. This enhanced space is used only for the auxiliary residual reconstruction; the MDP itself is still formulated in $\mathfrak U$.

Unless explicitly stated otherwise, the convergences in Condition~\ref{PEeq-MDPMDPcond-4.1} and Theorem~\ref{PEeq-MDPmainresu-4.7} are taken in the topology of $\mathfrak U$; the topology of $\mathfrak U_z$ is used for the auxiliary residual construction.

For \( \delta > 0 \), define
\[
\mathcal{H}^{\delta} = \left\{f : E \to \mathbb{R} :  \forall \mho \in \mathcal{B}(E) \text{~with~} \nu(\mho)<\infty, {\int_{\mho} e^{\delta f^2(\xi)}  \nu(d\xi) } < \infty \right\}
\]

$\bullet$ (\emph{A variational representation}) Define the mapping \(\hbar : \mathbb{R}^+ \to \mathbb{R}^+\) by $
\hbar(x) = x \log x - x + 1$.
For any \(\phi \in \overline{\mathfrak{R}}_+\) and $t\in [0,T]$, the quantity
\[
\mathfrak{L}_{t}(\phi) = \int_{E\times [0,t]} \hbar(\phi(s, x, \omega)) \, \nu(dx)ds
\]
is well defined as a \([0, \infty]\)-valued random variable. Let
$\{\mathfrak C_n\}_{n=1}^{\infty}$ be an increasing sequence of compact subsets of $E$ with \(\bigcup_{n=1}^\infty \mathfrak{C}_n = E\). For every \(n\in \mathbb{N}\), set
\[
\overline{\mathfrak{R}}_{b,n} := \left\{ \phi \in \overline{\mathfrak{R}}_+ : \phi(t, x, \omega) \in 
\begin{cases} 
	\left[ \frac{1}{n}, n \right], & \text{if } x \in \mathfrak{C}_n \\
	\{1\}, & \text{if } x \in \mathfrak{C}_n^c 
\end{cases} 
~~ \text{for all } (t,\omega)\in[0,T]\times\overline{\mathfrak M}\right\},
\]
and let \(\overline{\mathfrak{R}}_b = \bigcup_{n=1}^\infty \overline{\mathfrak{R}}_{b,n}\).

$\bullet$ (\emph{A general moderate deviation result}) We recall the following sufficient criterion for an MDP from \cite{Budhiraja-AOP-2016}. Suppose that $a(\epsilon)$ satisfies \eqref{PEeq-MDPdetdec4.4}. For $\epsilon > 0$, let $\mathscr{G}^{\epsilon}$ be a measurable map from $\mathfrak{M}$ to $\mathfrak{U}$, where $\mathfrak{M}$ is as defined in Section~\ref{PEeq-Pre2.3} and $\mathfrak{U}$ is a Polish space.
The criterion gives sufficient conditions for the LDP to hold for the mapping $\mathscr{G}^\epsilon(\epsilon N^{\epsilon^{-1}})$ as $\epsilon \to 0$, with small logarithmic scale $\epsilon/a^2(\epsilon)$ (equivalently, conventional speed $a^2(\epsilon)/\epsilon$) and a rate function given by a suitable quadratic form. This is the MDP considered below. For any \(\epsilon > 0\) and \(M\in (0,\infty)\), denote
\begin{align}\label{PEeq-MDPdetdec4.5}
\begin{split}
	\mathcal{S}_{+, \epsilon}^M := &\{\phi : E \times [0, T] \to \mathbb{R}^+ \mid \mathfrak{L}_T(\phi) \leq M a^2(\epsilon)\},\\
	\mathcal{S}_\epsilon^M := \{\psi :& E \times [0, T] \to \mathbb{R} \mid \psi = (\phi - 1)/a(\epsilon), \ \phi \in \mathcal S_{+, \epsilon}^M\},	
\end{split}
\end{align}
and
\begin{align}\label{PEeq-MDPdetdec4.6}
	\begin{split}
\mathcal{U}_{+, \epsilon}^M &:= \{\phi \in \overline{\mathfrak{R}}_b : \phi(\cdot, \cdot, \omega) \in \mathcal{S}_{+, \epsilon}^M,\ \overline{\mathbb P}\text{-a.s.}\},\\
\mathcal{U}_\epsilon^M &:= \{\psi \in \overline{\mathfrak{R}} : \psi(\cdot, \cdot, \omega) \in \mathcal{S}_\epsilon^M,\ \overline{\mathbb P}\text{-a.s.}\}.
\end{split}
\end{align}

Given a mapping $\mathscr{G}^{0}:L^2(\nu_T)\rightarrow \mathfrak{U}$ and $\eta \in \mathfrak{U}$, we set 
$$
\aleph_{\eta}^{0}=\left\{\psi\in L^2(\nu_T): \eta= \mathscr{G}^{0}(\psi)\right\},
$$
and define $I$ by 
\begin{align}\label{PEeq-MDPrefun-4.1}
	I(\eta)=\inf_{\psi \in \aleph_{\eta}^{0}}\left\{\frac{1}{2}\|\psi\|_{L^2(\nu_T)}^2\right\}
\end{align}
with the convention $\inf \emptyset=+\infty$.

If \( \phi \in \mathcal{S}_{+, \epsilon}^M \), then \cite[Lemma 3.2]{Budhiraja-AOP-2016} gives \( \kappa_2(1) \in (0, \infty) \) independent of \( \epsilon \) such that \( \psi \mathbf{1}_{\{|\psi| \leq  1/a(\epsilon)\}} \in \mathbb{B}_{\nu}(\sqrt{M\kappa_2(1)}) \), where \( \psi = (\phi - 1)/a(\epsilon) \). Moreover, $\kappa_2$ is a nondecreasing function from $(0,\infty)$ to $(0,\infty)$ such that for each $\beta>0$, 
\[
|x-1|^2 \leq \kappa_2(\beta) \hbar(x) \quad \text{for } |x-1| < \beta \text{ and } x \geq 0.
\]

To establish the main theorem of this section via the abstract criterion \cite[Theorem 2.3]{Budhiraja-AOP-2016}, it suffices to verify the following two claims.

\begin{condition}\label{PEeq-MDPMDPcond-4.1}
	Let $\mathscr{G}^{0}:L^2(\nu_T)\to\mathfrak U$ be measurable. Assume the following two conditions:
	
	$(A_1)$ Given \( M \in (0, \infty) \), suppose that \( q^\epsilon, q \in \mathbb{B}_\nu(M) \) and \(q^\epsilon\rightharpoonup q\) in the weak topology of $L^2(\nu_T)$. Then
	\[
	\mathscr G^0(q^\epsilon)\to\mathscr G^0(q).
	\]
	
	$(A_2)$ Given \( M \in (0, \infty) \), let \( \{\phi^\epsilon\}_{\epsilon > 0} \) be such that for every \( \epsilon > 0 \), \( \phi^\epsilon \in \mathcal{U}_{+, \epsilon}^M \) and for some \( \beta \in (0, 1] \), \( \psi^\epsilon \mathbf{1}_{\{|\psi^\epsilon| \leq \beta / a(\epsilon)\}} \xrightarrow{d} \psi \) in \( \mathbb{B}_\nu((M \kappa_2(1))^{1/2}) \) where \( \psi^\epsilon = (\phi^\epsilon - 1)/a(\epsilon) \). Then 
	\[
	\mathscr{G}^\epsilon (\epsilon N^{\epsilon^{-1} \phi^\epsilon}) \xrightarrow{d} \mathscr{G}^{0}(\psi) 
	\quad\text{in }(\mathfrak U,d_{\mathfrak U}).
	\]
	Here $\xrightarrow{d}$ denotes convergence in distribution.
\end{condition}

We use the following abstract criterion from \cite[Theorem 2.3]{Budhiraja-AOP-2016}.
\begin{theorem}\label{PEeq-MDPMDPthe-4.1}
	Suppose that the functionals \( \mathscr G^\epsilon \) and \(\mathscr G^0\) satisfy Condition \ref{PEeq-MDPMDPcond-4.1}.
	Then $I$, defined by \eqref{PEeq-MDPrefun-4.1}, is a good rate function, and \(\{\mathcal{Y}^\epsilon := \mathscr{G}^\epsilon (\epsilon N^{\epsilon^{-1}})\}_{\epsilon>0}\) satisfies a large deviation principle at logarithmic scale \(b(\epsilon)=\frac{\epsilon}{a^2(\epsilon)}\) (equivalently, with conventional speed $a^2(\epsilon)/\epsilon$) and rate function \( I \).
\end{theorem}

To establish the main results of this section, we impose only the exponential integrability assumption needed for the entropy estimates.
\begin{hypothesis}\label{PEeqAssum-2.5}
	Assume that $C_h$, $L_h$, and $\widetilde C_h$ from Hypothesis~\ref{PEeqAssum-2.3} belong to $\mathcal H^\delta$ for some $\delta>0$.
\end{hypothesis}

\begin{remark}\label{PEeq-assumption-scope}
	Hypothesis~\ref{PEeqAssum-2.5} is imposed only for the MDP analysis. None of the $\mathfrak{D}$-pullback measure attractor or zero-noise results in Sections~\ref{PMA-PEs3} and \ref{USC-PEs4} use Hypothesis~\ref{PEeqAssum-2.5}.
\end{remark}

\begin{theorem}[Main result]\label{PEeq-MDPmainresu-4.7}
	Suppose Hypotheses \ref{PEeqAssum-2.3} and \ref{PEeqAssum-2.5} hold.
	Let $v_0\in\mathcal H$, let $g,\partial_zg\in L^4(0,T;H)$, and assume that $a(\epsilon)$ satisfies \eqref{PEeq-MDPdetdec4.4}.
	Then $\{\mathcal{Y}^\epsilon\}$ satisfies the moderate deviation principle in the Polish space $(\mathfrak U,d_{\mathfrak U})$ defined by \eqref{PEeq-MDP-state-space}-\eqref{PEeq-MDP-state-metric}; equivalently, it satisfies an LDP at logarithmic scale $b(\epsilon)=\epsilon/a^2(\epsilon)$ (or conventional speed $a^2(\epsilon)/\epsilon$), with good rate function
	\[
	I(\eta) = \inf_{\psi} \left\{ \frac{1}{2} \|\psi\|_{L^2(\nu_T)}^2 \right\},
	\]
	where the infimum is taken over all $\psi \in L^2(\nu_T)$ such that $(\eta, \psi)$ satisfy the following skeleton equation
	\begin{align}\label{PEeq-MDPskeequ-4.8}
		\frac{d}{dt} \eta(t) = -A\eta(t) - B(\eta(t), v^0(t)) - B(v^0(t), \eta(t)) + \int_{E}  h(v^0(t), \xi)\psi(\xi, t) \nu(d\xi),
	\end{align}
	with initial value $\eta(0) = 0$.
\end{theorem}
\begin{proof}
	By Theorem~\ref{PEeq-MDPMDPthe-4.1}, it remains to verify Condition \ref{PEeq-MDPMDPcond-4.1}. Conditions $(A_1)$ and $(A_2)$ are verified in Subsections~\ref{Cond4.1} and \ref{Cond4.1-2}, respectively. The conclusion therefore follows from Theorem~\ref{PEeq-MDPMDPthe-4.1}.
\end{proof}

\begin{remark}\label{PEeq-MDPderem-4.8}
By Remark~\ref{deter1-dd}, \eqref{PEeq-MDPdet2.15} admits a unique solution \(v^0=(v^0(t))_{t\in[0,T]}\in C([0,T];H)\cap L^2(0,T;V)\), which satisfies
\begin{align}\label{PEeq-MDPderm-4.10}
	\sup_{t \in [0,T]} \left(|v^0(t)|^2+|\partial_{z}v^0(t)|^2\right) + \int_0^T \left(\|v^0(t)\|^2 +\|\partial_{z}v^0(t)\|^2\right)dt \leq C_{T}.
\end{align}
%and
%\begin{align}\label{PEeq-MDPderm-4.11}
%	\sup_{t \in [0,T]} |\partial_{z}v^0(t)|^2 + \int_0^T \|\partial_{z}v^0(t)\|^2 \, dt \leq C_{T}.
%\end{align}
%\begin{align}\label{PEeq-MDPderm-4.10}
%	\sup_{t \in [0,T]} |v^0(t)|^2 + \int_0^T \|v^0(t)\|^2 \, dt \leq C_{T}(|v_0|^2+1),
%\end{align}
%and
%\begin{align}\label{PEeq-MDPderm-4.11}
%	\sup_{t \in [0,T]} |\partial_{z}v^0(t)|^2 + \int_0^T \|\partial_{z}v^0(t)\|^2 \, dt \leq C_{T}(|\partial_{z}v_0|^2+1)
%\end{align}
\end{remark}

\subsection[Compactness of the skeleton map]{Verification of $(A_1)$ in Condition \ref{PEeq-MDPMDPcond-4.1}}\label{Cond4.1}
We first verify $(A_1)$ in Condition~\ref{PEeq-MDPMDPcond-4.1}. The following well-posedness result for the skeleton equation \eqref{PEeq-MDPskeequ-4.8} follows by the same argument as \cite[Theorem 4.3]{ZRR-2021}.
\begin{proposition}\label{PEeq-MDPskeequpro-4.9}
	Suppose Hypothesis~\ref{PEeqAssum-2.3} holds. Then, for every $\psi\in L^2(\nu_T)$, the skeleton equation \eqref{PEeq-MDPskeequ-4.8} has a unique solution $\eta = (\eta(t))_{t \in [0,T]} \in C([0,T]; H) \cap L^2([0,T]; V)$.
\end{proposition}

By Proposition \ref{PEeq-MDPskeequpro-4.9}, we define the solution map
$\mathscr{G}^{0}: L^2(\nu_{T}) \rightarrow C([0,T];H)\cap L^2([0,T];V)$ by
\begin{align}\label{PEeq-MDPmap0-4.12}
	\mathscr{G}^{0}(\psi)=\eta \text{~~for~} \psi \in  L^2(\nu_{T}), \text{~where~} (\eta,\psi) \text{~solves~} \eqref{PEeq-MDPskeequ-4.8}.
\end{align}
Its continuity, and hence Borel measurability, will follow from Lemma~\ref{PEeq-MDPCOnd1lem-4.10}.

\begin{lemma}\label{PEeq-MDPCOnd1lem-4.10}
	Suppose Hypotheses \ref{PEeqAssum-2.3} and \ref{PEeqAssum-2.5} hold. Let $R>0$ and let $q^{\epsilon},q\in\mathbb B_\nu(R)$ satisfy $q^{\epsilon}\rightharpoonup q$ weakly in $L^2(\nu_T)$. Then $\mathscr G^{0}(q^{\epsilon})\to\mathscr G^{0}(q)$ in $C([0,T];H)\cap L^2(0,T;V)$.
\end{lemma}
\begin{proof}
	Let $\eta^{\epsilon}:=\mathscr{G}^{0}(q^{\epsilon})$ and $\eta:=\mathscr{G}^{0}(q)$. 
	
\textbf{Step 1.} We first establish the following estimates.
%		\begin{align}\label{PEeq-MDPskees-4.13}
%		\sup_{t\in [0,T]}|\eta^{\epsilon}(t)|^2+\int_{0}^{T}\|\eta^{\epsilon}(s)\|^2ds \leq C_{R,T}.
%	\end{align}
	\begin{align}\label{PEeq-MDPskees-4.13}
		\sup_{t\in [0,T]}\left(|\eta^{\epsilon}(t)|^2+|\partial_{z}\eta^{\epsilon}(t)|^2\right)+\int_{0}^{T}\left(\|\eta^{\epsilon}(s)\|^2+\|\partial_{z}\eta^{\epsilon}(s)\|^2\right)ds \leq C_{R,T}.
	\end{align}
%	where $C_{R,T}$ is a positive constant.
Replacing \( \eta(t) \) by \( \eta^{\epsilon}(t) \) in \eqref{PEeq-MDPskeequ-4.8} and then taking the inner product of the resulting equation with \( \eta^{\epsilon}(t) \) in \( H \), we obtain
\begin{align}\label{PEeq-MDPskees-4.14}
	\begin{split}
		&~~|\eta^{\epsilon}(t)|^2+2\int_{0}^{t}\|\eta^{\epsilon}(s)\|^2ds=-2\int_{0}^{t} \left(B(\eta^{\epsilon}(s),v^0(s)),\eta^{\epsilon}(s)\right)ds\\
		&
		+2\int_{0}^{t} \int_{E} \left(h(v^0(s),\xi)q^\epsilon(\xi,s),\eta^{\epsilon}(s)\right)\nu(d\xi)ds:=\mathscr{T}_1(t)+\mathscr{T}_2(t),
	\end{split}
\end{align}	
	where we have used the cancellation property \eqref{PEeq-2.13}. Estimate~\eqref{PEeq-2.11}, Young's inequality, and \eqref{PEeq-MDPderm-4.10} give, for every $t\in[0,T]$,
\begin{align}\label{PEeq-MDPskees-4.15}
	\begin{split}
		|\mathscr{T}_1(t)|
		&\leq C\int_{0}^{t}\left(|\eta^{\epsilon}(s)|\|\eta^{\epsilon}(s)\|\|v^{0}(s)\|+|\partial_{x} \eta^{\epsilon}(s)||\partial_{z}v^{0}(s)|\|\eta^{\epsilon}(s)\|^{1/2}|\eta^{\epsilon}(s)|^{1/2}\right)ds\\
		&\leq \int_{0}^{t}\|\eta^{\epsilon}(s)\|^2ds
		+C\int_{0}^{t} |\eta^{\epsilon}(s)|^2\|v^{0}(s)\|^2ds+C\int_{0}^{t}|\partial_{z}v^{0}(s)|^4 |\eta^{\epsilon}(s)|^2ds\\
		&\leq \int_{0}^{t}\|\eta^{\epsilon}(s)\|^2ds
		+C\int_{0}^{t} |\eta^{\epsilon}(s)|^2\|v^{0}(s)\|^2ds+C\sup_{t\in [0,T]}|\partial_{z}v^{0}(t)|^4\int_{0}^{t} |\eta^{\epsilon}(s)|^2ds\\
		&\leq \int_{0}^{t}\|\eta^{\epsilon}(s)\|^2ds+C_{T}\int_{0}^{t} |\eta^{\epsilon}(s)|^2(1+\|v^{0}(s)\|^2)ds.
	\end{split}
\end{align}
	Condition \textbf{(B.1)} in Hypothesis~\ref{PEeqAssum-2.3}, Young's inequality, and \eqref{PEeq-MDPderm-4.10} yield, for every $t\in[0,T]$,
	\begin{align}\label{PEeq-MDPskees-4.16}
%		\begin{split}
			|\mathscr{T}_2(t)|
			&\leq C\int_{0}^{t}\int_{E} C_{h}(\xi)\left(1+|v^{0}(s)|\right)|q^\epsilon(\xi,s)||\eta^{\epsilon}(s)|\nu(d\xi)ds \notag \\
			&\leq C\int_{0}^{t}\int_{E} \left(\frac{5}{4}+|v^{0}(s)|^2\right)\left(C_{h}^2(\xi)+|q^\epsilon(\xi,s)|^2\right)\left(\frac{1}{4}+|\eta^{\epsilon}(s)|^2\right)\nu(d\xi)ds \notag \\
			&\leq C\sup_{t\in [0,T]}\left(\frac{5}{4}+|v^{0}(t)|^2\right)\int_{0}^{t}\int_{E} \left(C_{h}^2(\xi)+|q^\epsilon(\xi,s)|^2\right)\left(1+|\eta^{\epsilon}(s)|^2\right)\nu(d\xi)ds \notag \\
			&\leq C_{T}\int_{0}^{T}\int_{E}C_{h}^2(\xi)\nu(d\xi)ds+C_{T}\int_{0}^{T}\int_{E}|q^\epsilon(\xi,s)|^2\nu(d\xi)ds \notag \\
			&+C_{T}\int_{E}C_{h}^2(\xi)\nu(d\xi)\int_{0}^{t}|\eta^{\epsilon}(s)|^2ds+C_{T}\int_{0}^{t}\int_{E}|q^\epsilon(\xi,s)|^2|\eta^{\epsilon}(s)|^2\nu(d\xi)ds \notag \\
			&= C_{T}\left(T\|C_{h}\|_{L^2(\nu)}^{2}+\|q^{\epsilon}\|_{L^2(\nu_{T})}^{2}\right)+C_{T}\int_{0}^{t}|\eta^{\epsilon}(s)|^2\left(\|C_{h}\|_{L^2(\nu)}^{2}+\int_{E}|q^\epsilon(\xi,s)|^2\nu(d\xi)\right)ds \notag \\
			&\leq C_{T,R}+C_{T}\int_{0}^{t}|\eta^{\epsilon}(s)|^2\left(1+\int_{E}|q^\epsilon(\xi,s)|^2\nu(d\xi)\right)ds.
%		\end{split}
	\end{align}
	Combining \eqref{PEeq-MDPskees-4.13}--\eqref{PEeq-MDPskees-4.16} yields
	\begin{align*}
			|\eta^{\epsilon}(t)|^2+\int_{0}^{t}\|\eta^{\epsilon}(s)\|^2ds\leq C_{T,R}+C_{T}\int_{0}^{t}|\eta^{\epsilon}(s)|^2\left(1+\int_{E}|q^\epsilon(\xi,s)|^2\nu(d\xi)+\|v^{0}(s)\|^2\right)ds
	\end{align*}
	Applying Gronwall's inequality and \eqref{PEeq-MDPderm-4.10} gives %\eqref{PEeq-MDPskees-4.13}.
	\begin{align}\label{PEeq-MDPskees-4.17}
		\sup_{t\in [0,T]}|\eta^{\epsilon}(t)|^2+\int_{0}^{T}\|\eta^{\epsilon}(s)\|^2ds \leq C_{R,T}. 
	\end{align}
	
	Moreover, we replace $\eta(t)$ by $\eta^{\epsilon}(t)$ in \eqref{PEeq-MDPskeequ-4.8} and take the inner product of the resulting equation with $\partial_{zz} \eta^{\epsilon}(t)$ in $H$. Using integration by parts together with the cancellation property \eqref{PEeq-2.13}, we obtain
	\begin{align}\label{PEeq-MDPskees-4.18}
			&~~|\partial_{z}\eta^{\epsilon}(t)|^2+2\int_{0}^{t}\|\partial_{z}\eta^{\epsilon}(s)\|^2ds=-2\int_{0}^{t} \left(B(\eta^{\epsilon}(s),v^0(s)),\partial_{zz}\eta^{\epsilon}(s)\right)ds\\ \notag
			&-2\int_{0}^{t} \left(B(v^0(s),\eta^{\epsilon}(s)),\partial_{zz}\eta^{\epsilon}(s)\right)ds
			+2\int_{0}^{t} \int_{E} \left(h(v^0(s),\xi)q^\epsilon(\xi,s),\partial_{zz}\eta^{\epsilon}(s)\right)\nu(d\xi)ds:=\sum_{j=3}^{5}\mathscr{T}_{j}(t),
	\end{align}	
For $\mathscr{T}_{3}(t)$, by H\"{o}lder's inequality and interpolation inequality, we have
	\begin{align}\label{PEeq-MDPskees-4.19}
    &|\mathscr{T}_{3}(t)|\leq C\int_{0}^{t}\int_{\mathcal{M}} |\partial_{z}\eta^{\epsilon}(s)| |\partial_{x}v^{0}(s)| |\partial_{z}\eta^{\epsilon}(s)|dxdz ds+C\int_{0}^{t}\int_{\mathcal{M}} |\eta^{\epsilon}(s)| |\partial_{xz}v^{0}(s)| |\partial_{z}\eta^{\epsilon}(s)|dxdz ds\notag\\
    &+C\int_{0}^{t}\left|\int_{\mathcal{M}} \mathcal{W}(\eta^{\epsilon})(s)\left(\partial_{z}(v^0(s)\partial_{z}\eta^{\epsilon}(s))-\partial_{z}v^0(s)\partial_{z}\eta^{\epsilon}(s)\right)dxdz\right| ds\notag\\
    &\leq C\int_{0}^{t}|\partial_{z}\eta^{\epsilon}(s)| \|v^{0}(s)\| \|\partial_{z}\eta^{\epsilon}(s)\| ds+C\int_{0}^{t} |\eta^{\epsilon}(s)|^{1/2} \|\eta^{\epsilon}(s)\|^{1/2} \|\partial_{z}v^{0}(s)\| |\partial_{z}\eta^{\epsilon}(s)|^{1/2} \|\partial_{z}\eta^{\epsilon}(s)\|^{1/2} ds\notag\\
    &+C\int_{0}^{t}\int_{\mathcal{M}} |\partial_{x}\eta^{\epsilon}(s)| |v^0(s)||\partial_{z}\eta^{\epsilon}(s)|dxdz ds+C\int_{0}^{t}\int_{\mathcal{M}} |\mathcal{W}(\eta^{\epsilon})(s)| |\partial_{z}v^0(s)||\partial_{z}\eta^{\epsilon}(s)|dxdz ds\notag\\
    &\leq C\int_{0}^{t}|\partial_{z}\eta^{\epsilon}(s)| \|v^{0}(s)\| \|\partial_{z}\eta^{\epsilon}(s)\| ds+C\int_{0}^{t} |\eta^{\epsilon}(s)|^{1/2} \|\eta^{\epsilon}(s)\|^{1/2} \|\partial_{z}v^{0}(s)\| |\partial_{z}\eta^{\epsilon}(s)|^{1/2} \|\partial_{z}\eta^{\epsilon}(s)\|^{1/2} ds\notag\\
    &+C\int_{0}^{t} \|\eta^{\epsilon}(s)\| |v^0(s)|^{1/2}\|v^0(s)\|^{1/2}|\partial_{z}\eta^{\epsilon}(s)|^{1/2}\|\partial_{z}\eta^{\epsilon}(s)\|^{1/2} ds\notag\\
    &+C\int_{0}^{t}\|\eta^{\epsilon}(s)\| |\partial_{z}v^0(s)|^{1/2}\|\partial_{z}v^0(s)\|^{1/2}|\partial_{z}\eta^{\epsilon}(s)|^{1/2}\|\partial_{z}\eta^{\epsilon}(s)\|^{1/2} ds:=\sum_{j=6}^{9}\mathscr{T}_{j}(t).
	\end{align}	
Using Young's inequality yields
\begin{align*}
%	\begin{split}
		&\mathscr{T}_{6}(t)\leq C\int_{0}^{t}|\partial_{z}\eta^{\epsilon}(s)|^2 \|v^{0}(s)\|^2  ds+\frac{1}{8}\int_{0}^{t}\|\partial_{z}\eta^{\epsilon}(s)\|^2 ds,\\
		&\mathscr{T}_{7}(t)\leq C\int_{0}^{t}    |\partial_{z}\eta^{\epsilon}(s)|^{2}\|\eta^{\epsilon}(s)\|^{2}  ds
		+C\int_{0}^{t}  |\eta^{\epsilon}(s)| \|\partial_{z}v^{0}(s)\|^{2}  ds
		+\frac{1}{4}\int_{0}^{t}\|\partial_{z}\eta^{\epsilon}(s)\|^2 ds,\\
		&\mathscr{T}_{8}(t)\leq C\int_{0}^{t} |\partial_{z}\eta^{\epsilon}(s)|^{2}\|v^0(s)\|^{2} ds+C\int_{0}^{t} |v^0(s)|\|\eta^{\epsilon}(s)\|^{2} ds+\frac{1}{4}\int_{0}^{t}\|\partial_{z}\eta^{\epsilon}(s)\|^2 ds,\\
		&\mathscr{T}_{9}(t)\leq C\int_{0}^{t}|\partial_{z}\eta^{\epsilon}(s)|^{2}\|\partial_{z}v^0(s)\|^{2}ds+C\int_{0}^{t}|\partial_{z}v^0(s)| \|\eta^{\epsilon}(s)\|^{2}ds+\frac{1}{4}\int_{0}^{t}\|\partial_{z}\eta^{\epsilon}(s)\|^2 ds,
%	\end{split}
\end{align*}	
Together with \eqref{PEeq-MDPskees-4.19}, \eqref{PEeq-MDPderm-4.10}, and \eqref{PEeq-MDPskees-4.17}, this yields, for every $t\in [0,T]$,
\begin{align}\label{PEeq-MDPskees-4.20}
	\begin{split}
		|\mathscr{T}_{3}(t)|&\leq C\int_{0}^{t}|\partial_{z}\eta^{\epsilon}(s)|^2 \left(\|v^{0}(s)\|^2+\|\eta^{\epsilon}(s)\|^{2}+\|\partial_{z}v^0(s)\|^{2}\right)  ds+\frac{7}{8}\int_{0}^{t}\|\partial_{z}\eta^{\epsilon}(s)\|^2 ds\\
		&+C\left(\sup_{t \in [0,T]}(|\eta^{\epsilon}(t)|+|v^{0}(t)|+|\partial_{z}v^0(t)|)\right)\int_{0}^{t}  \left(\|\partial_{z}v^{0}(s)\|^{2}+\|\eta^{\epsilon}(s)\|^{2}\right)  ds\\
		&\leq C\int_{0}^{t}|\partial_{z}\eta^{\epsilon}(s)|^2 \left(\|v^{0}(s)\|^2+\|\eta^{\epsilon}(s)\|^{2}+\|\partial_{z}v^0(s)\|^{2}\right)  ds+\frac{7}{8}\int_{0}^{t}\|\partial_{z}\eta^{\epsilon}(s)\|^2 ds+C_{R,T}.
	\end{split}
\end{align}	
Similarly, for every $t\in[0,T]$, $\mathscr T_4(t)$ satisfies
\begin{align}\label{PEeq-MDPskees-4.21}
	|\mathscr{T}_{4}(t)|&\leq C\int_{0}^{t}\int_{\mathcal{M}} |\partial_{z}v^{0}(s)| |\partial_{x}\eta^{\epsilon}(s)| |\partial_{z}\eta^{\epsilon}(s)|dxdz ds+C\int_{0}^{t}\int_{\mathcal{M}} |v^{0}(s)| |\partial_{xz}\eta^{\epsilon}(s)| |\partial_{z}\eta^{\epsilon}(s)|dxdz ds\notag\\
	&+C\int_{0}^{t}\left|\frac{1}{2}\int_{\mathcal{M}} \mathcal{W}(v^{0})(s)\partial_{z}(\partial_{z}\eta^{\epsilon}(s))^2dxdz\right| ds\notag\\
	&\leq C\int_{0}^{t}|\partial_{z}v^{0}(s)|^{1/2}\|\partial_{z}v^{0}(s)\|^{1/2} \|\eta^{\epsilon}(s)\| |\partial_{z}\eta^{\epsilon}(s)|^{1/2}\|\partial_{z}\eta^{\epsilon}(s)\|^{1/2} ds\notag\\
	&+C\int_{0}^{t}|v^{0}(s)|^{1/2} \|v^{0}(s)\|^{1/2} \|\partial_{z}\eta^{\epsilon}(s)\| |\partial_{z}\eta^{\epsilon}(s)|^{1/2} \|\partial_{z}\eta^{\epsilon}(s)\|^{1/2}ds\notag\\
	&+C\int_{0}^{t} |\partial_{z}\eta^{\epsilon}(s)| \|\partial_{z}\eta^{\epsilon}(s)\| \|v^{0}(s)\|ds\notag\\
	&\leq C\int_{0}^{t} |\partial_{z}\eta^{\epsilon}(s)|^2\left((|v^{0}(s)|^2+1)\|v^{0}(s)\|^{2}+\|\partial_{z}v^{0}(s)\|^{2}\right)ds+\frac{5}{8}\int_{0}^{t}\|\partial_{z}\eta^{\epsilon}(s)\|^2 ds\notag\\
	&
	+C\sup_{t \in [0,T]}|\partial_{z}v^{0}(t)|\int_{0}^{t}\|\eta^{\epsilon}(s)\|^2ds\notag\\
	&\leq C\int_{0}^{t} |\partial_{z}\eta^{\epsilon}(s)|^2\left((|v^{0}(s)|^2+1)\|v^{0}(s)\|^{2}+\|\partial_{z}v^{0}(s)\|^{2}\right)ds+\frac{5}{8}\int_{0}^{t}\|\partial_{z}\eta^{\epsilon}(s)\|^2 ds+C_{R,T}.
\end{align}	
It remains to handle $\mathscr{T}_{5}(t)$, by 
 condition \textbf{(B.3)} in Hypothesis \ref{PEeqAssum-2.3}, Young's inequality and \eqref{PEeq-MDPderm-4.10}, we deduce that for any $t\in [0,T]$,
\begin{align}\label{PEeq-MDPskees-4.22}
		|\mathscr{T}_{5}(t)|
		&\leq C\int_{0}^{t} \int_{E} \left|\left(\partial_{z}h(v^0(s),\xi)q^\epsilon(\xi,s),\partial_{z}\eta^{\epsilon}(s)\right)\right|\nu(d\xi)ds\notag\\
		&\leq C\int_{0}^{t}\int_{E} \widetilde{C}_{h}(\xi)\left(1+|\partial_{z}v^{0}(s)|\right)|q^\epsilon(\xi,s)||\partial_{z}\eta^{\epsilon}(s)|\nu(d\xi)ds\notag\\
%		&\leq C\int_{0}^{t}\int_{E} \left(\frac{5}{4}+|v^{0}(s)|^2\right)\left(C_{h}^2(\xi)+|q^\epsilon(\xi,s)|^2\right)\left(\frac{1}{4}+|\eta^{\epsilon}(s)|^2\right)\nu(d\xi)ds\\
		&\leq C\sup_{t\in [0,T]}\left(1+|\partial_{z}v^{0}(t)|^2\right)\int_{0}^{t}\int_{E} \left(\widetilde{C}_{h}^2(\xi)+|q^\epsilon(\xi,s)|^2\right)\left(1+|\partial_{z}\eta^{\epsilon}(s)|^2\right)\nu(d\xi)ds\notag\\
		&\leq C_{T}\int_{E}\widetilde{C}_{h}^2(\xi)\nu(d\xi)+C_{T}\int_{0}^{T}\int_{E}|q^\epsilon(\xi,s)|^2\nu(d\xi)ds\notag\\
		&+C_{T}\int_{E}\widetilde{C}_{h}^2(\xi)\nu(d\xi)\int_{0}^{t}|\partial_{z}\eta^{\epsilon}(s)|^2ds+C_{T}\int_{0}^{t}\int_{E}|q^\epsilon(\xi,s)|^2|\partial_{z}\eta^{\epsilon}(s)|^2\nu(d\xi)ds\notag\\
%		&= C_{T}\left(T\|C_{h}\|_{L^2(\nu)}^{2}+\|q^{\epsilon}\|_{L^2(\nu_{T})}^{2}\right)+C_{T}\int_{0}^{t}|\eta^{\epsilon}(s)|^2\left(\|C_{h}\|_{L^2(\nu)}^{2}+\int_{E}|q^\epsilon(\xi,s)|^2\nu(d\xi)\right)ds\\
		&\leq C_{R,T}+C_{T}\int_{0}^{t}|\partial_{z}\eta^{\epsilon}(s)|^2\left(1+\int_{E}|q^\epsilon(\xi,s)|^2\nu(d\xi)\right)ds.
\end{align}
By \eqref{PEeq-MDPskees-4.18}, \eqref{PEeq-MDPskees-4.20}, \eqref{PEeq-MDPskees-4.21} and \eqref{PEeq-MDPskees-4.22}, we have
\begin{align*}
		&~~|\partial_{z}\eta^{\epsilon}(t)|^2+\frac{1}{2}\int_{0}^{t}\|\partial_{z}\eta^{\epsilon}(s)\|^2ds\\
		&\leq C\int_{0}^{t} |\partial_{z}\eta^{\epsilon}(s)|^2\left((|v^{0}(s)|^2+1)\|v^{0}(s)\|^{2}+\|\eta^{\epsilon}(s)\|^2+\|\partial_{z}v^{0}(s)\|^{2}+\int_{E}|q^\epsilon(\xi,s)|^2\nu(d\xi)+1\right)ds+C_{R,T},
\end{align*}	
Together with Gronwall's inequality, \eqref{PEeq-MDPderm-4.10}, \eqref{PEeq-MDPskees-4.17}, and $q^\epsilon \in \mathbb{B}_{\nu}(R)$, this yields
\begin{align}\label{PEeq-MDPskees-4.23}
	\sup_{t \in [0,T]}|\partial_{z}\eta^{\epsilon}(t)|^2+\int_{0}^{T}\|\partial_{z}\eta^{\epsilon}(s)\|^2ds\leq C_{R,T}.
\end{align}
Thus, by \eqref{PEeq-MDPskees-4.17} and \eqref{PEeq-MDPskees-4.23} we obtain \eqref{PEeq-MDPskees-4.13}.

%-------------------------------------------------------------sole lemma in above

\textbf{Step 2.} We next prove that there exist $C_{R,T,\widetilde{\alpha}}, C_{R,T,\widehat{\alpha}} >0$ such that for $ \widetilde{\alpha}\in (0,\frac{1}{2})$,
\begin{align}\label{PEeq-MDPskeesr-4.24}
	\|\eta^{\epsilon}\|^2_{W^{\widetilde{\alpha},2}([0,T];V')}\leq C_{R,T,\widetilde{\alpha}},
\end{align}
and for $ \widehat{\alpha}\in (0,\frac{1}{4})$,
\begin{align}\label{PEeq-MDPskeesr-4.25}
	\|\partial_{z}\eta^{\epsilon}\|^2_{W^{\widehat{\alpha},2}([0,T];V')}\leq C_{R,T,\widehat{\alpha}}.
\end{align}

From \eqref{PEeq-MDPskeequ-4.8}, we have 
\begin{align}\label{PEeq-MDPskeesr-4.26}
	\begin{split}
		\eta^{\epsilon}(t)&=-\int_{0}^{t } A\eta^{\epsilon}(s)ds-\int_{0}^{t } B(\eta^{\epsilon}(s),v^0(s))ds-\int_{0}^{t }B(v^0(s),\eta^{\epsilon}(s))ds\\
		&+\int_{0}^{t}\int_{E}h(v^0(s),\xi)q^\epsilon(\xi,s) \nu(d\xi)ds:=\sum_{j=1}^{4}\mathfrak{J}_{j}(t), 
	\end{split}
\end{align}	
Since $\|A\eta^{\epsilon}\|_{V'}^2\leq C\|\eta^{\epsilon}\|^2$, Minkowski's and H\"{o}lder's inequalities give, for $0\leq s<t\leq T$,
\begin{align*}%\label{PEeq-MDPskeesr-4.27}
%	\begin{split}
		\|\mathfrak{J}_{1}(t)-\mathfrak{J}_{1}(s)\|_{V'}^{2}&=\|\int_{s}^{t } A\eta^{\epsilon}(r)dr\|_{V'}^{2}\leq \left(\int_{s}^{t } \|A\eta^{\epsilon}(r)\|_{V'}dr\right)^{2}\\
		&\leq  (t-s)\int_{s}^{t } \|A\eta^{\epsilon}(r)\|_{V'}^{2}dr \leq C(t-s)\int_{s}^{t } \|\eta^{\epsilon}(r)\|^2dr,
%	\end{split}
\end{align*}
Combining this estimate with \eqref{PEeq-MDPskees-4.13} yields, for every $\widetilde{\alpha}\in (0,\frac{1}{2})$,
\begin{align}\label{PEeq-MDPskeesr-4.27}
	\begin{split}
		\|\mathfrak{J}_{1}\|^2_{W^{\widetilde{\alpha},2}([0,T];V')}
		&=\int_{0}^{T}\|\mathfrak{J}_{1}(t)\|_{V'}^2dt+\int_{0}^{T}\int_{0}^{T} \frac{\|\mathfrak{J}_{1}(t)-\mathfrak{J}_{1}(s)\|_{V'}^{2}}{|t-s|^{1+2\widetilde{\alpha}}} dtds\\
		&\leq C_{T}\int_{0}^{T} \|\eta^{\epsilon}(s)\|^2ds+C\int_{0}^{T} \|\eta^{\epsilon}(s)\|^2ds \int_{0}^{T}\int_{0}^{T} \frac{1}{|t-s|^{2\widetilde{\alpha}}} dtds\leq C_{R,T,\widetilde{\alpha}}.
	\end{split}
\end{align}

Estimate~\eqref{PEeq-2.0013} and H\"{o}lder's inequality give
\begin{align*}
	&~~\|\mathfrak{J}_{2}(t)-\mathfrak{J}_{2}(s)\|_{V'}^{2}
	\leq \left(\int_{s}^{t } \|B(\eta^{\epsilon}(r),v^0(r))\|_{V'}dr\right)^{2}\notag\\
	&\leq C\left(\int_{s}^{t } \left(\left(|\eta^{\epsilon}(r)|+|\partial_{z} \eta^{\epsilon}(r)|\right)\|v^0(r)\|+|\eta^{\epsilon}(r)|\|\partial_{z}v^0(r)\|+|\eta^{\epsilon}(r)||\partial_{z}v^0(r)|^{1/2}\|\partial_{z}v^0(r)\|^{1/2}\right)dr\right)^2\notag\\
	&\leq C(t-s)\sup_{t \in [0,T]}\left(|\eta^{\epsilon}(t)|^2+|\partial_{z} \eta^{\epsilon}(t)|^2\right)\int_{s}^{t}\left(\|v^0(r)\|^2+\|\partial_{z}v^0(r)\|^2\right)dr\notag\\
	&
	+C\left(\int_{s}^{t }|\eta^{\epsilon}(r)|(|\partial_{z}v^0(r)|+\|\partial_{z}v^0(r)\|)dr\right)^{2}\notag\\
	&\leq C(t-s)\sup_{t \in [0,T]}\left(|\eta^{\epsilon}(t)|^2+|\partial_{z} \eta^{\epsilon}(t)|^2\right)\int_{s}^{t}\left(\|v^0(r)\|^2+\|\partial_{z}v^0(r)\|^2\right)dr\notag\\
	&+C_{T}(t-s)(\sup_{t \in [0,T]}|\eta^{\epsilon}(t)|^2)(\sup_{t \in [0,T]}|\partial_{z}v^0(t)|^2),
\end{align*}
Combining this estimate with \eqref{PEeq-MDPderm-4.10} and \eqref{PEeq-MDPskees-4.13} gives
\begin{align}\label{PEeq-MDPskeesr-4.28}
	\|\mathfrak{J}_{2}\|^2_{W^{\widetilde{\alpha},2}([0,T];V')}\leq C_{R,T,\widetilde{\alpha}}.
\end{align}
Similarly, for $\mathfrak J_3$ and all $s,t\in[0,T]$,
\begin{align*}
	\|\mathfrak{J}_{3}(t)-\mathfrak{J}_{3}(s)\|_{V'}^{2}
%	\leq \left(\int_{s}^{t } \|B(v^0(r),\eta^{\epsilon}(r))\|_{V'}dr\right)^{2}\notag\\
%	&\leq C\left(\int_{s}^{t } \left(\left(|v^0(r)|+|\partial_{z} v^0(r)|\right)\|\eta^{\epsilon}(r)\|+|v^0(r)|\|\partial_{z}\eta^{\epsilon}(r)\|+|v^0(r)||\partial_{z}\eta^{\epsilon}(r)|^{1/2}\|\partial_{z}\eta^{\epsilon}(r)\|^{1/2}\right)dr\right)^2\notag\\
%	&\leq C(t-s)\sup_{t \in [0,T]}\left(|v^0(t)|^2+|\partial_{z} v^0(t)|^2\right)\int_{s}^{t}\left(\|\eta^{\epsilon}(r)\|^2+\|\partial_{z}\eta^{\epsilon}(r)\|^2\right)dr\notag\\
%	&
%	+C\left(\int_{s}^{t }|v^0(r)|(|\partial_{z}\eta^{\epsilon}(r)|+\|\partial_{z}\eta^{\epsilon}(r)\|)dr\right)^{2}\notag\\
	&\leq C(t-s)\sup_{t \in [0,T]}\left(|v^0(t)|^2+|\partial_{z} v^0(t)|^2\right)\int_{s}^{t}\left(\|\eta^{\epsilon}(r)\|^2+\|\partial_{z}\eta^{\epsilon}(r)\|^2\right)dr\notag\\
	&+C_{T}(t-s)(\sup_{t \in [0,T]}|v^0(t)|^2)(\sup_{t \in [0,T]}|\partial_{z}\eta^{\epsilon}(t)|^2)\leq C_{R,T}(t-s),
\end{align*}
and hence
\begin{align}\label{PEeq-MDPskeesr-4.29}
	\|\mathfrak{J}_{3}\|^2_{W^{\widetilde{\alpha},2}([0,T];V')}\leq C_{R,T,\widetilde{\alpha}}.
\end{align}

To prove \eqref{PEeq-MDPskeesr-4.24}, it remains to estimate $\mathfrak J_4$. An $H$-norm bound is sufficient. Condition \textbf{(B.1)} in Hypothesis~\ref{PEeqAssum-2.3} and \eqref{PEeq-MDPderm-4.10} give, for all $s,t\in[0,T]$,
\begin{align*}
	&~~|\mathfrak{J}_{4}(t)-\mathfrak{J}_{4}(s)|^2=|\int_{s}^{t} \int_{E}h(v^0(r),\xi)q^\epsilon(\xi,r) \nu(d\xi)dr|^2\notag\\
	&\leq C\left(\int_{s}^{t} \int_{E}|h(v^0(r),\xi)||q^\epsilon(\xi,r)| \nu(d\xi)dr\right)^{2}\notag\\
	&\leq C\left(\int_{s}^{t} \int_{E}C_{h}(\xi)\left(1+|v^0(r)|\right)|q^\epsilon(\xi,r)| \nu(d\xi)dr\right)^{2}\notag\\
	&\leq C\left(\int_{s}^{t} \int_{E}C_{h}^2(\xi)\left(1+|v^0(r)|^2\right) \nu(d\xi)dr\right)\left(\int_{s}^{t} \int_{E}|q^\epsilon(\xi,r)|^2 \nu(d\xi)dr\right)\notag\\
	&\leq C\sup_{t\in [0,T]}\left(1+|v^0(t)|^2\right)\left(\int_{s}^{t} \int_{E}C_{h}^2(\xi)\nu(d\xi)dr\right) \left(\int_{s}^{t} \int_{E}|q^\epsilon(\xi,r)|^2 \nu(d\xi)dr\right)\notag\\
	&= C(t-s)\sup_{t\in [0,T]}\left(1+|v^0(t)|^2\right)\|C_{h}\|_{L^2({\nu})}^2\|q^\epsilon\|_{L^2({\nu_T})}^2\leq C_{R,T}(t-s).
\end{align*}
It follows from $\|\cdot\|_{V'}^2\leq C|\cdot|^2$ that 
\begin{align}\label{PEeq-MDPskeesr-4.30}
	\|\mathfrak{J}_{4}\|^2_{W^{\widetilde{\alpha},2}([0,T];V')}\leq C_{R,T,\widetilde{\alpha}}.
\end{align}
Combining \eqref{PEeq-MDPskeesr-4.26}--\eqref{PEeq-MDPskeesr-4.30} gives \eqref{PEeq-MDPskeesr-4.24}.

We next prove \eqref{PEeq-MDPskeesr-4.25}. Differentiating \eqref{PEeq-MDPskeequ-4.8} with respect to $z$ gives 
\begin{align}\label{PEeq-MDPskeesr-4.31}
	\begin{split}
		\partial_{z}\eta^{\epsilon}(t)&=-\int_{0}^{t } A\partial_{z}\eta^{\epsilon}(s)ds-\int_{0}^{t } \partial_{z}B(\eta^{\epsilon}(s),v^0(s))ds-\int_{0}^{t }\partial_{z}B(v^0(s),\eta^{\epsilon}(s))ds\\
		&+\int_{0}^{t}\int_{E}\partial_{z}h(v^0(s),\xi)q^\epsilon(\xi,s) \nu(d\xi)ds:=\sum_{j=5}^{8}\mathfrak{J}_{j}(t), 
	\end{split}
\end{align}	
For $\mathfrak J_5$ and $\mathfrak J_8$, the arguments used for \eqref{PEeq-MDPskeesr-4.27} and \eqref{PEeq-MDPskeesr-4.30}, together with \eqref{PEeq-MDPskees-4.13}, condition \textbf{(B.3)} in Hypothesis~\ref{PEeqAssum-2.3}, and \eqref{PEeq-MDPderm-4.10}, give, for $\widehat\alpha\in(0,1/4)$,
\begin{align}\label{PEeq-MDPskeesr-4.32}
	\|\mathfrak{J}_{5}\|^2_{W^{\widehat{\alpha},2}([0,T];V')}+\|\mathfrak{J}_{8}\|^2_{W^{\widehat{\alpha},2}([0,T];V')}\leq C_{R,T,\widehat{\alpha}}.
\end{align}

We still need to handle $\mathfrak{J}_{6}$ and $\mathfrak{J}_{7}$. From \eqref{PEeq-2.11}, we have
$$
\|B(\partial_{z}\eta^{\epsilon},v^0)\|_{V'}\leq C(|\partial_{z}\eta^{\epsilon}|^{1/2}\|\partial_{z}\eta^{\epsilon}\|^{1/2}\|v^0\|+\|\partial_{z}\eta^{\epsilon}\||\partial_{z}v^0|),
$$
and further, by \eqref{PEeq-2.13}, we have
$$
\|B(\eta^{\epsilon},\partial_{z}v^0)\|_{V'}\leq C(|\eta^{\epsilon}|^{1/2}\|\eta^{\epsilon}\|^{1/2}|\partial_{z}v^0|^{1/2}\|\partial_{z}v^0\|^{1/2}+\|\eta^{\epsilon}\||\partial_{z}v^0|^{1/2}\|\partial_{z}v^0\|^{1/2}).
$$
For $\mathfrak{J}_{6}$, by H\"{o}lder's inequality, \eqref{PEeq-MDPderm-4.10} and \eqref{PEeq-MDPskees-4.13},  we deduce that for any $s,t\in [0,T]$,
\begin{align}\label{PEeq-MDPskeesr-4.33}
	&~~\|\mathfrak{J}_{6}(t)-\mathfrak{J}_{6}(s)\|_{V'}^2
	\leq \left(\int_{s}^{t}\left(\|B(\partial_{z}\eta^{\epsilon}(r),v^0(r))\|_{V'}+\|B(\eta^{\epsilon}(r),\partial_{z}v^0(r))\|_{V'}\right)dr\right)^{2} \notag\\
	&\leq C\left(\int_{s}^{t}(|\partial_{z}\eta^{\epsilon}(r)|^{1/2}\|\partial_{z}\eta^{\epsilon}(r)\|^{1/2}\|v^0(r)\|+\|\partial_{z}\eta^{\epsilon}(r)\||\partial_{z}v^0(r)|)dr\right)^2\notag\\
	&+C\left(\int_{s}^{t} (|\eta^{\epsilon}(r)|^{1/2}\|\eta^{\epsilon}(r)\|^{1/2}|\partial_{z}v^0(r)|^{1/2}\|\partial_{z}v^0(r)\|^{1/2}+\|\eta^{\epsilon}(r)\||\partial_{z}v^0(r)|^{1/2}\|\partial_{z}v^0(r)\|^{1/2}) dr\right)^2\notag\\
	&\leq C(t-s)^{1/2}\left(\sup_{t\in [0,T]}|\partial_{z}\eta^{\epsilon}(t)|\int_{s}^{t}\|\partial_{z}\eta^{\epsilon}(r)\|^{2}dr\right)^{1/2}\int_{s}^{t} \|v^0(r)\|^2dr\notag\\
	&+C(t-s)\sup_{t\in [0,T]}|\partial_{z}v^0(t)|^2 \int_{s}^{t}\|\partial_{z}\eta^{\epsilon}(r)\|^2dr\notag\\
	&+C(t-s)\left(\sup_{t\in [0,T]}|\eta^{\epsilon}(t)|\int_{s}^{t}\|\eta^{\epsilon}(r)\|^{2}dr\right)^{1/2}\left(\sup_{t\in [0,T]}|\partial_{z}v^0(t)|\int_{s}^{t}\|\partial_{z}v^0(r)\|^{2}dr\right)^{1/2}\notag\\
	&+C(t-s)^{1/2}\left(\sup_{t\in [0,T]}|\partial_{z}v^0(t)|\int_{s}^{t}\|\partial_{z}v^0(r)\|^{2}dr\right)^{1/2}\int_{s}^{t} \|\eta^{\epsilon}(r)\|^2dr\notag\\
	&\leq C_{R,T}\left((t-s)^{1/2}+(t-s)\right),
\end{align}
and hence
\begin{align}\label{PEeq-MDPskeesr-4.34}
	\|\mathfrak{J}_{6}\|^2_{W^{\widehat{\alpha},2}([0,T];V')}\leq C_{R,T,\widehat{\alpha}}.
\end{align}
Applying the same argument as for \eqref{PEeq-MDPskeesr-4.33} to $\mathfrak J_7$ gives, for all $s,t\in[0,T]$,
\begin{align}\label{PEeq-MDPskeesr-4.35}
	\|\mathfrak{J}_{7}\|^2_{W^{\widehat{\alpha},2}([0,T];V')}\leq C_{R,T,\widehat{\alpha}}.
\end{align}
Combining \eqref{PEeq-MDPskeesr-4.31}--\eqref{PEeq-MDPskeesr-4.32} with \eqref{PEeq-MDPskeesr-4.33}--\eqref{PEeq-MDPskeesr-4.34} yields \eqref{PEeq-MDPskeesr-4.25}.

\textbf{Step 3.} By \eqref{PEeq-MDPskees-4.13}, \eqref{PEeq-MDPskeesr-4.24} and \eqref{PEeq-MDPskeesr-4.25}, together with Lemma~\ref{PEeq-MDPCpctness-4.1}, the solution $\eta^{\epsilon}$ of \eqref{PEeq-MDPskeequ-4.8} is relatively compact in $L^{2}([0,T];H)$. Hence there exist an element $\widehat{\eta}$ and a subsequence $\eta^{\epsilon_k}$ such that, as $k\to \infty$, 

\begin{enumerate}
	\item[(i)] $\widehat{\eta}, \partial_{z}\widehat{\eta} \in  L^\infty([0,T];H) \cap L^2([0,T];V)$,
	\item[(ii)] $\eta^{\epsilon_k} \to \widehat{\eta}$ weakly-* in $L^\infty([0,T];H)$,
	\item[(iii)] $\eta^{\epsilon_k} \to \widehat{\eta}$ weakly in $L^2([0,T];V)$, $\; \eta^{\epsilon_k} \to \widehat{\eta}$ strongly in $L^2([0,T];H)$,
	\item[(iv)] $\partial_{z}\eta^{\epsilon_k} \to \partial_{z}\widehat{\eta}$ weakly-* in $L^\infty([0,T];H)$,
	\item[(v)] $\partial_{z}\eta^{\epsilon_k} \to \partial_{z}\widehat{\eta}$ weakly in $L^2([0,T];V)$, $\; \partial_{z}\eta^{\epsilon_k} \to \partial_{z}\widehat{\eta}$ strongly in $L^2([0,T];H)$.
\end{enumerate}

It remains to identify the limit. Let $\phi\in V$. The strong convergence in $L^2(0,T;H)$, the weak convergence in $L^2(0,T;V)$, and estimate \eqref{PEeq-2.0013} imply
\[
\int_0^T\langle B(\eta^{\epsilon_k},v^0)-B(\widehat\eta,v^0),\phi\rangle dt\to0,
\qquad
\int_0^T\langle B(v^0,\eta^{\epsilon_k})-B(v^0,\widehat\eta),\phi\rangle dt\to0.
\]
Moreover, since $q^{\epsilon_k}\rightharpoonup q$ in $L^2(\nu_T)$ and
$(t,\xi)\mapsto(h(v^0(t),\xi),\phi)$ belongs to $L^2(\nu_T)$, the controlled forcing terms converge weakly to the forcing generated by $q$. Passing to the limit in the variational formulation of \eqref{PEeq-MDPskeequ-4.8} shows that $(\widehat\eta,q)$ satisfies the same skeleton equation as $(\eta,q)$. The uniqueness in Proposition \ref{PEeq-MDPskeequpro-4.9} gives $\widehat\eta=\eta$. Since every convergent subsequence has the same limit, the full sequence converges in the stated weak and strong topologies.
%It remains to show that $\widehat{\eta}=\eta$. Since the argument is standard, we omit it for brevity and refer the reader to \cite[Proposition 5.1]{ZRR-2021} for details.

\textbf{Step 4.}  We upgrade the convergence to $C([0,T];H)\cap L^2(0,T;V)$ without assuming strong convergence of the controls.
Let $\varpi^{\epsilon_k}=\eta^{\epsilon_k} -{\eta}$. Then, by \eqref{PEeq-MDPskeequ-4.8} we have
\begin{align}\label{PEeq-MDPskeesr-4.36}
	\begin{split}
		\frac{d}{dt} \varpi^{\epsilon_k}(t) &= -A\varpi^{\epsilon_k}(t) - B(\varpi^{\epsilon_k}(t), v^0(t)) - B(v^0(t), \varpi^{\epsilon_k}(t))\\
		& + \int_{E}  h(v^0(t), \xi)\left(q^{\epsilon_k}(\xi, t)-q(\xi, t)\right) \nu(d\xi),
	\end{split}
\end{align}
	with initial value $\varpi^{\epsilon_k}(0) = 0$.

Taking the $H$-inner product with $\varpi^{\epsilon_k}$ and using the cancellation of $B(v^0,\varpi^{\epsilon_k})$ gives
\begin{align}\label{PEeq-MDPskeesr-4.37}
	\begin{split}
		|\varpi^{\epsilon_k}(t)|^2+2\int_0^{t}\|\varpi^{\epsilon_k}(s)\|^2ds
		&=-2\int_0^{t} \langle B(\varpi^{\epsilon_k}(s), v^0(s)),\varpi^{\epsilon_k}(s)\rangle ds\\
		&~~+2\int_0^{t}\int_{E}  \left(h(v^0(s), \xi)\left(q^{\epsilon_k}(\xi, s)-q(\xi, s)\right),\varpi^{\epsilon_k}(s)\right) \nu(d\xi) ds.
	\end{split}
\end{align}
%\ref{PEeq-lem2.1} and \eqref{PEeq-2.13}
%\begin{align}
%	\begin{split}
%		&~~\left|2\int_0^{t} \langle B(\varpi^{\epsilon_k}(s), v^0(s)),\varpi^{\epsilon_k}(s)\rangle ds\right|\leq 2\int_0^{t} \left|\langle B(\varpi^{\epsilon_k}(s), \varpi^{\epsilon_k}(s)),v^0(s)\rangle\right| ds\\
%		&\leq C\int_{0}^{t} (|\varpi^{\epsilon_k}(s)|^{1/2}\|\varpi^{\epsilon_k}(s)\|^{3/2}|v^0(s)|^{1/2}\|v^0(s)\|^{1/2} + \|\varpi^{\epsilon_k}(s)\| |\partial_{z} \varpi^{\epsilon_k}(s)| |v^0(s)|^{1/2}\|v^0(s)\|^{1/2})ds\\
%		&\leq \int_0^{t}\|\varpi^{\epsilon_k}(s)\|^2ds+
%	\end{split}
%\end{align}
By \eqref{PEeq-2.11}, for any $t\in [0,T]$,
\begin{align}\label{PEeq-MDPskeesr-4.38}
	\begin{split}
		&~~\left|2\int_0^{t} \langle B(\varpi^{\epsilon_k}(s), v^0(s)),\varpi^{\epsilon_k}(s)\rangle ds\right|\leq 2\int_0^{t} \left|\langle B(\varpi^{\epsilon_k}(s), v^0(s)),\varpi^{\epsilon_k}(s)\rangle\right| ds\\
		&\leq C\int_{0}^{t} (|\varpi^{\epsilon_k}(s)|\|\varpi^{\epsilon_k}(s)\|\|v^0(s)\| +  |\partial_{z} v^0(s)| |\varpi^{\epsilon_k}(s)|^{1/2}\|\varpi^{\epsilon_k}(s)\|^{3/2})ds\\
		&\leq \int_0^{t}\|\varpi^{\epsilon_k}(s)\|^2ds+C\int_{0}^{t} \left(\|v^0(s)\|^2+|\partial_{z} v^0(s)|^{4}\right)|\varpi^{\epsilon_k}(s)|^2ds.
	\end{split}
\end{align}
On the other hand, condition \textbf{(B.1)} in Hypothesis
\ref{PEeqAssum-2.3}, H\"older's inequality, and
$q^{\epsilon_k},q\in\mathbb B_\nu(R)$ imply
\begin{align}\label{PEeq-MDPskeesr-4.39}
	&~~\left|2\int_0^{t}\int_{E}  \left(h(v^0(s), \xi)\left(q^{\epsilon_k}(\xi, s)-q(\xi, s)\right),\varpi^{\epsilon_k}(s)\right) \nu(d\xi) ds\right| \notag \\
	&\leq 2\int_0^{t}\int_{E} \left| \left(h(v^0(s), \xi)\left(q^{\epsilon_k}(\xi, s)-q(\xi, s)\right),\varpi^{\epsilon_k}(s)\right) \right|\nu(d\xi) ds \notag \\
	&\leq 2\sup_{t \in [0,T]}(1+|v^0(t)|)\int_0^t \int_{E} C_{h}(\xi)|q^{\epsilon_k}(\xi, s)-q(\xi, s)|
	|\varpi^{\epsilon_k}(s)|\nu(d\xi) ds\notag \\
	&\leq
	2\|C_h\|_{L^2(\nu)}\sup_{t \in [0,T]}(1+|v^0(t)|)
	\left(\int_0^T\int_E|q^{\epsilon_k}(\xi,s)-q(\xi,s)|^2
	\,\nu(d\xi)\,ds\right)^{1/2}
	\left(\int_0^T|\varpi^{\epsilon_k}(s)|^2\,ds\right)^{1/2}\notag\\
	&\leq
	4R\|C_h\|_{L^2(\nu)}\sup_{t \in [0,T]}(1+|v^0(t)|)
	\|\varpi^{\epsilon_k}\|_{L^2(0,T;H)}\longrightarrow0,
\end{align}
since \textbf{Step~3} gives $\varpi^{\epsilon_k}\to0$ strongly in $L^2(0,T;H)$.
Combining \eqref{PEeq-MDPskeesr-4.37}-\eqref{PEeq-MDPskeesr-4.39} and applying
Gronwall's inequality, we obtain
\begin{align}\label{PEeq-MDPMDPlinm4.400-04.7}
\sup_{0\leq t\leq T}|\varpi^{\epsilon_k}(t)|^2\longrightarrow0.
\end{align}
Moreover, \eqref{PEeq-MDPderm-4.10} gives
\[
\int_0^T\bigl(\|v^0(s)\|^2+|\partial_zv^0(s)|^4\bigr)\,ds<\infty.
\]
Returning to \eqref{PEeq-MDPskeesr-4.37} and using
\eqref{PEeq-MDPskeesr-4.38}--\eqref{PEeq-MDPskeesr-4.39}, we further obtain
\[
\int_0^T\|\varpi^{\epsilon_k}(s)\|^2\,ds\longrightarrow0.
\]
Therefore,
\[
\eta^{\epsilon_k}\longrightarrow\eta
\quad\text{in }C([0,T];H)\cap L^2(0,T;V).
\]
Since every convergent subsequence has the same limit, the whole sequence converges, and condition $(A_1)$ follows. Moreover, norm convergence in $L^2(\nu_T)$ implies boundedness and weak convergence, so the same result shows that $\mathscr G^0$ is norm-continuous, hence Borel measurable. This completes the proof.
%Combining with \eqref{PEeq-MDPskeesr-4.37}-\eqref{PEeq-MDPskeesr-4.39} yields
%\begin{align*}
%	&~~|\varpi^{\epsilon_k}(t)|^2+\int_0^{t}\|\varpi^{\epsilon_k}(s)\|^2ds\\
%	&\leq C_{T}\int_{0}^{t} \left(1+\|v^0(s)\|^2+|\partial_{z} v^0(s)|^{4}\right)|\varpi^{\epsilon_k}(s)|^2ds+ C_{T}\int_0^{T}\int_{E}|q^{\epsilon_k}(\xi, s)-q(\xi, s)|^2\nu(d\xi) ds.
%\end{align*}
%Using Gronwall's inequality, we obtain that for any $t\in [0,T]$,
%\begin{align*}
%%	\begin{split}
%		&~~|\varpi^{\epsilon_k}(t)|^2+\int_0^{t}\|\varpi^{\epsilon_k}(s)\|^2ds\\
%		&\leq C_{T}e^{C_T\int_0^{t}\left(1+\|v^0(s)\|^2+|\partial_{z} v^0(s)|^{4}\right)ds}\int_0^{t}\int_{E}|q^{\epsilon_k}(\xi, s)-q(\xi, s)|^2\nu(d\xi) ds\\
%		&\leq C_{T}e^{C_T\left(T+\int_{0}^{T}\|v^0(t)\|^2dt+\sup\limits_{t\in [0,T]}|\partial_{z} v^0(t)|^{4} \right)}\int_0^{T}\int_{E}|q^{\epsilon_k}(\xi, s)-q(\xi, s)|^2\nu(d\xi) ds
%%	\end{split}
%\end{align*}
%Therefore, by \eqref{PEeq-MDPderm-4.10} and $q^{\epsilon_k}\rightarrow q$ in $\mathbb{B}_{\nu}(R)$ we have
%\begin{align}\label{PEeq-MDPMDPlinm4.400-04.7}
%	\begin{split}
%		&~~\lim_{k\rightarrow \infty}\left(\sup_{t\in [0,T]}|\varpi^{\epsilon_k}(t)|^2+\int_0^{T}\|\varpi^{\epsilon_k}(s)\|^2ds\right)\\
%		&\leq C_{T}\lim_{k\rightarrow \infty}\int_0^{T}\int_{E}|q^{\epsilon_k}(\xi, s)-q(\xi, s)|^2\nu(d\xi) ds=0.
%	\end{split}
%\end{align}
%This completes the proof.
\end{proof}

\subsection{Verification of \texorpdfstring{$(A_2)$}{(A2)} in Condition \ref{PEeq-MDPMDPcond-4.1}}\label{Cond4.1-2}

To verify part $(A_2)$ of Condition \ref{PEeq-MDPMDPcond-4.1}, we need the following results. The detailed proofs of Lemmas \ref{PEeq-MDPMDPlem-4.8} and \ref{PEeq-MDPMDPlem-4.9} are given in \cite[Lemmas 4.3 and 4.4]{Budhiraja-AOP-2016}.
\begin{lemma}\label{PEeq-MDPMDPlem-4.7}
% Let \( f \in L^2(\nu) \cap \mathcal{H}^{\delta} \) for some $\delta>0$, and let \( \Lambda \) be a measurable subset of \([0, T]\). Let \( M \in (0,\infty) \). Then there exists \( \textbf{c}_{f} > 0 \) such that for all \( \epsilon > 0 \),
Let $f\in L^2(\nu)\cap\mathcal H^\delta$ for some $\delta>0$, let
$\Lambda\in\mathcal B([0,T])$, and fix $M\in(0,\infty)$. Then there exists a constant
$\textbf{c}_{f}>0$, independent of $\epsilon$, $\Lambda$, and
$\phi\in\mathcal S_{+,\epsilon}^M$, such that
\[
\sup_{\phi \in \mathcal{S}_{+,\epsilon}^M} \int_{E \times \Lambda} f^2(\xi) \phi(\xi, s) \nu(d\xi) ds \leq \textbf{c}_{f} \big( a^2(\epsilon) + \lambda_T(\Lambda) \big).
\]
\end{lemma}
\begin{proof}
	We first show that
	\begin{align}\label{PEeq-MDPMDPlem-4.7-exp}
	\int_E\bigl(e^{\delta f^2(\xi)}-1\bigr)\,\nu(d\xi)<\infty.
	\end{align}
	Set $E_0=\{\xi:|f(\xi)|\le1\}$ and $E_1=\{\xi:|f(\xi)|>1\}$.
	Since $e^{\delta r}-1\le \delta e^\delta r$ for $0\le r\le1$,
	\[
	\int_{E_0}\bigl(e^{\delta f^2(\xi)}-1\bigr)d\nu
	\le \delta e^\delta\int_{E_0} f^2d\nu(\xi)<\infty.
	\]
	Moreover,
	\[
	\nu(E_1)\le\int_{E_1}f^2(\xi)d\nu(\xi)\le\|f\|_{L^2(\nu)}^2<\infty.
	\]
	The definition of $\mathcal H^\delta$, applied to the finite-measure set $E_1$, gives
	\[
	\int_{E_1}e^{\delta f^2(\xi)}d\nu(\xi)<\infty,
	\]
	and hence \eqref{PEeq-MDPMDPlem-4.7-exp} holds.
	
	The convex conjugacy between $\hbar(r)=r\log r-r+1$ and $e^y-1$ yields
	\begin{align}\label{PEeq-MDPMDPlem-4.7-Young}
	yr\le \hbar(r)+e^y-1,
	\qquad y\ge0,\quad r\ge0.
	\end{align}
	Taking $y=\delta f^2(\xi)$ and $r=\phi(\xi,s)$ in
	\eqref{PEeq-MDPMDPlem-4.7-Young}, integrating over $E\times\Lambda$, and using
	$\mathfrak L_T(\phi)\le Ma^2(\epsilon)$, we obtain
	\begin{align*}
	\int_{E\times\Lambda} f^2(\xi)\phi(\xi,s)\,\nu(d\xi)ds
	&\le \frac1\delta\int_{E\times\Lambda}\hbar(\phi(\xi,s))\,\nu(d\xi)ds+\frac{\lambda_T(\Lambda)}\delta
	\int_E\bigl(e^{\delta f^2(\xi)}-1\bigr)\,\nu(d\xi)\\
&\le \frac{M}{\delta}a^2(\epsilon)
	+\frac{\lambda_T(\Lambda)}\delta
	\int_E\bigl(e^{\delta f^2(\xi)}-1\bigr)\,\nu(d\xi).
	\end{align*}
	Thus the desired conclusion holds with
	\[
	\textbf{c}_{f}
	=\frac1\delta\max\left\{M,
	\int_E\bigl(e^{\delta f^2(\xi)}-1\bigr)\nu(d\xi)\right\}.
	\]
	This completes the proof.
\end{proof}

\begin{lemma}\label{PEeq-MDPMDPlem-4.8}
	Let $f\in L^2(\nu)\cap\mathcal H^\delta$ for some $\delta>0$, let $\Lambda\subset[0,T]$ be measurable, and fix $M<\infty$. Then there exist maps $\ell_f,\rho_f,\varsigma_f:(0,\infty)\to(0,\infty)$ such that
	\[
	\ell_f(\beta)\longrightarrow0\quad \text{as~~} \beta\to\infty,
	\qquad
	\varsigma_f(\epsilon)\longrightarrow0\quad \text{as~~} \epsilon\to 0,
	\]
	and, for every $\epsilon,\beta\in (0,\infty)$,
	\[
	\sup_{\psi\in\mathcal S_\epsilon^M}
	\int_{E\times\Lambda}|f(\xi)\psi(\xi,s)|
	\mathbf1_{\{|\psi|\ge\beta/a(\epsilon)\}}\,\nu(d\xi)ds
	\le \ell_f(\beta)\sqrt{a(\epsilon)}
	+\varsigma_f(\epsilon)\bigl(1+\lambda_T(\Lambda)\bigr),
	\]
	and
	\[
	\sup_{\psi\in\mathcal S_\epsilon^M}
	\int_{E\times\Lambda}|f(\xi)\psi(\xi,s)|\,\nu(d\xi)ds
	\le \rho_f(\beta)\sqrt{\lambda_T(\Lambda)}
	+\ell_f(\beta)\sqrt{a(\epsilon)}
	+\varsigma_f(\epsilon)\bigl(1+\lambda_T(\Lambda)\bigr).
	\]
\end{lemma}

\begin{lemma}\label{PEeq-MDPMDPlem-4.9}
	 Let \( f \in L^2(\nu) \cap \mathcal{H}^{\delta} \) for some $\delta>0$, and suppose $f\geq 0$.  Then for any \( \beta \in (0,\infty) \) and $M\in \mathbb{N}$,
	\[
	\lim_{\epsilon \to 0} \sup_{\psi \in \mathcal{S}_{\epsilon}^M} \int_{E \times [0, T]} |f(\xi) \psi(\xi, s)| \mathbf{1}_{\{|\psi| > \beta/a(\epsilon)\}} \nu(d\xi) ds = 0.
	\]
\end{lemma}

Recall $\mathcal{U}_{+,\epsilon}^M$ from \eqref{PEeq-MDPdetdec4.6}. For fixed $M>0$ and $\phi^\epsilon\in\mathcal U_{+,\epsilon}^M$, consider the controlled equation below. Proposition~\ref{prop-controlled-wp}, based on a Galerkin approximation and Lemma~\ref{PEeq-MDPMDPlem-4.14}, gives its existence and pathwise uniqueness. Girsanov's theorem then identifies this solution with the equation driven by the controlled Poisson random measure in the variational representation.
\begin{align}\label{PEeq-MDPMDPequ-4.40}
	\begin{split}
		&~~du^{\epsilon}(t)+Au^{\epsilon}(t)dt+B(u^{\epsilon}(t),u^{\epsilon}(t))dt\\
		&=g(t)dt+\epsilon\int_{E}h(u^{\epsilon}(t-),\xi)\widetilde{N}^{\epsilon^{-1}\phi^\epsilon}(dt,d\xi)+\int_{E} h(u^{\epsilon}(t),\xi)(\phi^\epsilon(\xi,t)-1)\nu(d\xi)dt,
	\end{split}
\end{align} 
with initial value $u^{\epsilon}(0)=v_0$.

\begin{lemma}\label{PEeq-MDPMDPlem-4.14}
Suppose Hypotheses \ref{PEeqAssum-2.3} and \ref{PEeqAssum-2.5} hold. Let $v_0\in \mathcal{H}$ and $g, \partial_{z}g\in L^4([0,T];H)$. Then there exists $\epsilon_{0}>0$ such that the following estimates hold:
\begin{align}\label{PEeq-MDPMDPequ-4.41}
	\sup_{\epsilon\in (0,\epsilon_{0})}\left(\mathbb{E}\left[\sup_{t\in [0,T]}|u^{\epsilon}(t)|^2\right]+\int_{0}^{T}\mathbb{E}\left[\|u^{\epsilon}(t)\|^2\right]dt\right)\leq C_{\epsilon_{0},T},
\end{align}
and
\begin{align}\label{PEeq-MDPMDPequ-4.42}
	\sup_{\epsilon\in (0,\epsilon_{0})}\left(\mathbb{E}\left[\sup_{t\in [0,T]}|\partial_{z}u^{\epsilon}(t)|^2\right]+\int_{0}^{T}\mathbb{E}\left[\|\partial_{z}u^{\epsilon}(t)\|^2\right]dt\right)\leq C_{\epsilon_{0},T},
\end{align}
where $C_{\epsilon_{0},T} > 0$ is a constant that depends on $\epsilon_{0}$ and $T$.
\end{lemma}
\begin{proof}
	The computation is initially performed on the finite-dimensional Galerkin system obtained by projecting \eqref{PEeq-MDPMDPequ-4.40} onto the span of the first $n$ eigenfunctions of $A$. In this setting, all stochastic integrals are defined in finite dimensions, and It\^{o}'s formula may be applied without further justification. The constants appearing in the estimates below are chosen independently of $n$; consequently, the estimates extend to the limiting solution by virtue of weak lower semicontinuity. Making use of the skew-symmetry relation $\left(B(u_n^{\epsilon},u_n^{\epsilon}),u_n^{\epsilon}\right)=0$, together with the standard $H$-energy estimate, the compensator identity, the BDG inequality, and Lemma~\ref{PEeq-MDPMDPlem-4.8}, we obtain \eqref{PEeq-MDPMDPequ-4.41} uniformly in $n$. We further provide an additional estimate for the vertical velocity, which is particular to the primitive-equation setting. To ease notation, we omit the subscript $n$ in what follows.

	Applying It\^{o}'s formula to \eqref{PEeq-MDPMDPequ-4.40} yields
\begin{align}\label{PEeq-MDPMDPequ-4.43}
	\begin{split}
		|\partial_{z}u^{\epsilon}(t)|^2+2\int_0^t \|\partial_{z}u^{\epsilon}(s)\|^2ds
		&=|\partial_{z}v_0|^2+ 2\int_0^t \left(g(s),\partial_{zz}u^{\epsilon}(s)\right)ds\\
		&+2\epsilon\int_0^t\int_{E}\left(\partial_{z}h(u^{\epsilon}(s-),\xi),\partial_{z}u^{\epsilon}(s)\right)\widetilde{N}^{\epsilon^{-1}\phi^\epsilon}(ds,d\xi)\\
		&+2\int_0^t\int_{E} \left(\partial_{z}h(u^{\epsilon}(s),\xi)(\phi^\epsilon(\xi,s)-1),\partial_{z}u^{\epsilon}(s)\right)\nu(d\xi)ds\\
		&+
		\epsilon^2\int_0^t\int_{E} |\partial_{z}h(u^{\epsilon}(s-),\xi)|^2{N}^{\epsilon^{-1}\phi^\epsilon}(ds,d\xi)=:\Upsilon_{1}+\sum_{j=2}^{5}\Upsilon_{j}(t).
	\end{split}
\end{align}
where we have used the cancellation property $\int_0^t \left(B(u^{\epsilon}(s),u^{\epsilon}(s)),\partial_{zz}u^{\epsilon}(s)\right) ds=0$, as given in \eqref{PEeq-2.13}.
Using Young's inequality yields
\begin{align}\label{PEeq-MDPMDPequ-4.45}
	\left|\Upsilon_{2}(t)\right|
	\leq \frac{\textbf{c}_0}{2}\int_0^t |\partial_{z}u^{\epsilon}(s)|^2ds
	+\frac{2}{\textbf{c}_0}\int_0^{t} |\partial_{z}g(s)|^2ds.
\end{align}
Let $\psi^\epsilon(\xi,s)=\frac{\phi^\epsilon(\xi,s)-1}{a(\epsilon)} \in \mathcal{U}_{\epsilon}^M$. Then condition \textbf{(B.3)} in Hypothesis~\ref{PEeqAssum-2.3} gives
\begin{align}\label{PEeq-MDPMDPequ-4.46}
		\left|\Upsilon_{4}(t)\right|
		&\leq 2a(\epsilon)\int_0^t |\partial_{z}u^{\epsilon}(s)|\int_{E}|\partial_{z}h(u^{\epsilon}(s),\xi)||\psi^\epsilon(\xi,s)|\nu(d\xi)ds\notag \\
		&\leq 2a(\epsilon)\int_0^t |\partial_{z}u^{\epsilon}(s)|(1+|\partial_{z}u^{\epsilon}(s)|)\int_{E}\widetilde{C}_{h}(\xi)|\psi^\epsilon(\xi,s)|\nu(d\xi)ds\notag \\
		&\leq 2a(\epsilon)\int_0^t (1+|\partial_{z}u^{\epsilon}(s)|)^2\int_{E}\widetilde{C}_{h}(\xi)|\psi^\epsilon(\xi,s)|\nu(d\xi)ds\notag \\
		&\leq 4a(\epsilon)\int_0^{t}\int_{E}\widetilde{C}_{h}(\xi)|\psi^\epsilon(\xi,s)|\nu(d\xi)ds+4a(\epsilon)\int_0^t |\partial_{z}u^{\epsilon}(s)|^2\int_{E}\widetilde{C}_{h}(\xi)|\psi^\epsilon(\xi,s)|\nu(d\xi)ds\notag \\
			&=:\Upsilon_{6,\epsilon}(t)+\int_0^t |\partial_{z}u^{\epsilon}(s)|^2d\Upsilon_{6,\epsilon}(s),
\end{align}	
where 
$$
\Upsilon_{6,\epsilon}(t):=4a(\epsilon)\int_0^{t}\int_{E}\widetilde{C}_{h}(\xi)|\psi^\epsilon(\xi,s)|\nu(d\xi)ds.
$$
Combining \eqref{PEeq-MDPMDPequ-4.43}--\eqref{PEeq-MDPMDPequ-4.46}, we obtain, for every $t\in [0,T]$,
 \begin{align}\label{PEeq-MDPMDPequ-4.47}
 	\begin{split}
 		&~|\partial_{z}u^{\epsilon}(t)|^2+\frac{3}{2}\int_0^t \|\partial_{z}u^{\epsilon}(s)\|^2ds
 		\leq \Upsilon_{1}+ \frac{2}{\textbf{c}_0}\int_0^{t} |\partial_{z}g(s)|^2ds+\sup_{t\in [0,T]}\left|\Upsilon_{3}(t)\right|\\
 		&+\Upsilon_{5}(T)+\Upsilon_{6,\epsilon}(T)+4a(\epsilon)\int_0^t |\partial_{z}u^{\epsilon}(s)|^2\int_{E}\widetilde{C}_{h}(\xi)|\psi^\epsilon(\xi,s)|\nu(d\xi)ds.
 	\end{split}
 \end{align}
Gronwall's inequality gives
\begin{align}
	\begin{split}
		&~~|\partial_{z}u^{\epsilon}(t)|^2+\int_0^t \|\partial_{z}u^{\epsilon}(s)\|^2ds\\
		&\leq \left(\Upsilon_{1}+\frac{2}{\textbf{c}_0}\left(\frac{T}{2}+\frac{1}{2}\int_0^{T} |\partial_{z}g(s)|^4ds\right)+\sup_{t\in [0,T]}\left|\Upsilon_{3}(t)\right|+\Upsilon_{5}(T)+\Upsilon_{6,\epsilon}(T)\right) e^{\Upsilon_{6,\epsilon}(T)}.
	\end{split}
\end{align}
It follows from Lemma \ref{PEeq-MDPMDPlem-4.8} that
\[\Upsilon_{\epsilon,6}(T)\leq 4a(\epsilon)\left(\rho_{\widetilde{C}_{h}}(\beta) \sqrt{T}+\ell_{\widetilde{C}_{h}}(\beta)\sqrt{a(\epsilon)} + \varsigma_{\widetilde{C}_{h}}(\epsilon) \big( 1 + T \big)\right):=\overline{\Gamma}_{\epsilon},\]
and hence
\begin{align*}
	&~~\Upsilon_{1}+\frac{2}{\textbf{c}_0}\left(\frac{T}{2}+\frac{1}{2}\int_0^{T} |\partial_{z}g(s)|^4ds\right)+\Upsilon_{6,\epsilon}(T)\\
	&\leq C_T+4a(\epsilon)\left(\rho_{\widetilde{C}_{h}}(\beta) \sqrt{T}+\ell_{\widetilde{C}_{h}}(\beta)\sqrt{a(\epsilon)} + \varsigma_{\widetilde{C}_{h}}(\epsilon) \big( 1 + T \big)\right).
\end{align*}
Therefore, by \eqref{PEeq-MDPMDPequ-4.47} we obtain
 \begin{align}\label{PEeq-MDPMDPequ-4.48}
	\begin{split}
		&~\mathbb{E}\left[\sup_{t\in [0,T]}|\partial_{z}u^{\epsilon}(t)|^2\right]+\int_0^{T} \mathbb{E}\left[\|\partial_{z}u^{\epsilon}(t)\|^2\right]dt\\
		&\leq \left(C_T+4a(\epsilon)\left(\rho_{\widetilde{C}_{h}}(\beta) \sqrt{T}+\ell_{\widetilde{C}_{h}}(\beta)\sqrt{a(\epsilon)} + \varsigma_{\widetilde{C}_{h}}(\epsilon) \big( 1 + T \big)\right)+\mathbb{E}\left[\sup_{t\in [0,T]}\left|\Upsilon_{3}(t)\right|+\Upsilon_{5}(T)\right]\right)\\
		&\quad \times e^{4a(\epsilon)\left(\rho_{\widetilde{C}_{h}}(\beta) \sqrt{T}+\ell_{\widetilde{C}_{h}}(\beta)\sqrt{a(\epsilon)} + \varsigma_{\widetilde{C}_{h}}(\epsilon) \big( 1 + T \big)\right)}.
	\end{split}
\end{align}
By the BDG inequality, condition \textbf{(B.3)} in Hypothesis \ref{PEeqAssum-2.3}, and Lemma \ref{PEeq-MDPMDPlem-4.7} we have
\begin{align}\label{PEeq-MDPMDPequ-4.49}
	\mathbb{E}\left[\sup_{t\in [0,T]}\left|\Upsilon_{3}(t)\right|\right]
	&\leq 2\epsilon \mathbb{E}\left[\left(\int_0^T\int_{E}|\partial_{z}u^{\epsilon}(t)|^2|\partial_{z}h(u^{\epsilon}(t-),\xi)|^2{N}^{\epsilon^{-1}\phi^\epsilon}(dt,d\xi)\right)^{1/2}\right]\notag \\
	&\leq 2\epsilon \mathbb{E}\left[\sup_{t\in [0,T]}|\partial_{z}u^{\epsilon}(t)| \left(\int_0^T\int_{E}|\partial_{z}h(u^{\epsilon}(t-),\xi)|^2{N}^{\epsilon^{-1}\phi^\epsilon}(dt,d\xi)\right)^{1/2}\right]\notag \\
	&\leq \frac{1}{2}\mathbb{E}\left[\sup_{t\in [0,T]}|\partial_{z}u^{\epsilon}(t)|^2\right]+8\epsilon \mathbb{E}\left[\left(1+\sup_{t\in [0,T]}|\partial_{z}u^{\epsilon}(t)|^2\right)\int_0^T\int_{E}\widetilde{C}_{h}^2(\xi)\phi^\epsilon(\xi,t)\nu(d\xi)dt\right]\notag \\
	&\leq \frac{1}{2}\mathbb{E}\left[\sup_{t\in [0,T]}|\partial_{z}u^{\epsilon}(t)|^2\right]+ 8\epsilon\textbf{c}_{\widetilde{C}_{h}} \big( a^2(\epsilon) + T \big) \mathbb{E}\left[\left(1+\sup_{t\in [0,T]}|\partial_{z}u^{\epsilon}(t)|^2\right)\right],
\end{align}
and
\begin{align}\label{PEeq-MDPMDPequ-4.50}
\mathbb{E}\left[\Upsilon_{5}(T)\right]&\leq 2\epsilon\mathbb{E}\left[\left(1+\sup_{t\in [0,T]}|\partial_{z}u^{\epsilon}(t)|^2\right)\int_0^T\int_{E}\widetilde{C}_{h}^2(\xi)\phi^\epsilon(\xi,t)\nu(d\xi)dt\right]\notag \\
&\leq 	2\epsilon\textbf{c}_{\widetilde{C}_{h}} \big( a^2(\epsilon) + T \big) \mathbb{E}\left[\left(1+\sup_{t\in [0,T]}|\partial_{z}u^{\epsilon}(t)|^2\right)\right].
\end{align}
Since the deterministic upper bound $\overline \Gamma_\epsilon$ displayed above converges to zero, we reduce $\epsilon_0$, if necessary, so that
\[
\sup_{0<\epsilon\le\epsilon_0}e^{\overline \Gamma_\epsilon}\le\frac{4}{3},
\qquad
10\epsilon_0 c_{\widetilde C_h}
\left(T+\sup_{0<\epsilon\le\epsilon_0}a^2(\epsilon)\right)\le\frac{1}{8}.
\]
Substituting \eqref{PEeq-MDPMDPequ-4.49}--\eqref{PEeq-MDPMDPequ-4.50} into \eqref{PEeq-MDPMDPequ-4.48} gives
\begin{align*}
	\mathbb{E}\left[\sup_{t\in [0,T]}|\partial_{z}u^{\epsilon}(t)|^2\right]+\int_0^{T} \mathbb{E}\left[\|\partial_{z}u^{\epsilon}(t)\|^2\right]dt 
	&\le C_{\epsilon_0,T}+\frac{4}{3}\left(\frac12+\frac18\right)\mathbb{E}\left[\sup_{t\in [0,T]}|\partial_{z}u^{\epsilon}(t)|^2\right]\\
	&=C_{\epsilon_0,T}+\frac{5}{6}\mathbb{E}\left[\sup_{t\in [0,T]}|\partial_{z}u^{\epsilon}(t)|^2\right].
\end{align*}
which implies \eqref{PEeq-MDPMDPequ-4.42}. This completes the proof.
\end{proof}

\begin{proposition}\label{prop-controlled-wp}
		Suppose Hypotheses~\ref{PEeqAssum-2.3} and~\ref{PEeqAssum-2.5} hold, $v_0\in\mathcal H$, and $g,\partial_zg\in L^4(0,T;H)$.  For every $M>0$ there exists $\epsilon_0>0$ such that, for $0<\epsilon<\epsilon_0$ and every predictable $\phi^\epsilon\in\mathcal U_{+,\epsilon}^M$, equation~\eqref{PEeq-MDPMDPequ-4.40} has a unique adapted c\`adl\`ag solution
		\[
		u^\epsilon\in\mathcal D([0,T];H)\cap L^2(0,T;V),
		\qquad
		\partial_z u^\epsilon\in L^\infty(0,T;H)\cap L^2(0,T;V),
		\quad\mathbb P\text{-a.s.}
		\]
		The estimates \eqref{PEeq-MDPMDPequ-4.41}-\eqref{PEeq-MDPMDPequ-4.42} hold uniformly over the admissible controls.
\end{proposition}
\begin{proof}
	We distinguish the compactness topology used for existence from the stronger MDP topology $\mathfrak U$.

	\textbf{Step 1.}
	Let $u_n^\epsilon$ be the spectral Galerkin solution of \eqref{PEeq-MDPMDPequ-4.40}, stopped when
	\[
	|u_n^\epsilon|^2+|\partial_zu_n^\epsilon|^2+
	\int_0^t(\|u_n^\epsilon\|^2+\|\partial_zu_n^\epsilon\|^2)ds
	\]
	reaches $R$. The estimates of Lemma~\ref{PEeq-MDPMDPlem-4.14}, together with the entropy inequality and Lemma~\ref{PEeq-MDPMDPlem-4.8}, are uniform in $n$ and $R$ and yield
	\eqref{PEeq-MDPMDPequ-4.41}--\eqref{PEeq-MDPMDPequ-4.42}. Hence the explosion times tend to $T$ almost surely.

	\textbf{Step 2.}
	Set
	\[
	\hat{\mathfrak Z}=\mathcal D([0,T];V^*)\cap L^2(0,T;H),\qquad
	d_{\hat{\mathfrak Z}}=d_{J_1,V^*}+\|\cdot\|_{L^2(0,T;H)}.
	\]
	As in Lemma~\ref{lem:MDP-path-space-Polish}, $\hat{\mathfrak Z}$ is Polish. Fix the control $\phi^\epsilon$. Since $\phi^\epsilon\in\overline{\mathfrak R}_b$, choose $m$ such that $\phi^\epsilon\in\overline{\mathfrak R}_{b,m}$. Put $K_m=[0,T]\times\mathfrak C_m$ and view $\phi^\epsilon|_{K_m}$ as a random element of the closed ball
	\[
	\mathbb B_m^{\rm ctrl}:=
	\left\{q\in L^2(K_m,\nu_T):\|q\|_{L^2(K_m,\nu_T)}\le m\sqrt{T\nu(\mathfrak C_m)}\right\},
	\]
	endowed with the weak $L^2$ topology. This ball is compact and metrizable, hence Polish; moreover $\phi^\epsilon=1$ on $K_m^c$, so its restriction determines the full control. Write the Galerkin equation as
	\[
	u_n^\epsilon(t)=P_nv_0+\int_0^tF_n^\epsilon(s)\,ds+M_n^\epsilon(t),
	\]
	where $F_n^\epsilon$ contains $-Au_n^\epsilon$, the nonlinear drift, $g$, and the controlled compensator, while $M_n^\epsilon$ is the projected compensated Poisson integral. Estimate~\eqref{PEeq-2.0013}, the bounds of Step~1, and the fact that $\phi^\epsilon-1$ is bounded and supported in $\mathfrak C_m$ imply
	\[
	\sup_n\mathbb E\left[\int_0^T\|F_n^\epsilon(s)\|_{V^*}^2ds\right]<\infty.
	\]
	Indeed, every term in \eqref{PEeq-2.0013} is square integrable in time by the $L^\infty(0,T;H)$ bounds for $u_n^\epsilon$ and $\partial_zu_n^\epsilon$ and the $L^2(0,T;V)$ bounds for both quantities.

	For stopping times $\tau_n\le T$, the isometry for compensated Poisson integrals gives
	\[
	\mathbb E\left[\left|M_n^\epsilon((\tau_n+\theta)\wedge T)-M_n^\epsilon(\tau_n)\right|_{V^*}^2\right]
	\le C_{m,\epsilon}\theta.
	\]
	The finite-variation part satisfies the analogous estimate by Cauchy--Schwarz. Hence Aldous' condition holds in $V^*$. Moreover, the same increment estimates integrated over deterministic pairs $(s,t)$ yield, for every $\alpha\in(0,1/2)$,
	\[
	\sup_n\mathbb E\left[\|u_n^\epsilon\|_{W^{\alpha,2}(0,T;V^*)}^2\right]<\infty.
	\]
	Applying Lemma~\ref{PEeq-MDPCpctness-4.1} with $B_0=V$, $B=H$, and $B_1=V^*$ gives tightness in $L^2(0,T;H)$, whereas the uniform $L^\infty(0,T;H)$ bound, the compact embedding $H\hookrightarrow V^*$, and Aldous' criterion give tightness in $\mathcal D([0,T];V^*)$. Since $\hat{\mathfrak Z}$ is the closed diagonal subspace of the product of these two Polish spaces, the laws of $u_n^\epsilon$ are tight in $\hat{\mathfrak Z}$.

	Let $\overline N$ be the canonical Poisson random measure used in the thinning construction \eqref{Possine}. The joint law of
	\[
	\mathsf Y^\epsilon:=(\overline N,\phi^\epsilon|_{K_m})
	\]
	on the Polish input space $\overline{\mathfrak M}\times\mathbb B_m^{\rm ctrl}$ is independent of the Galerkin index $n$. Consequently, the joint laws of $(u_n^\epsilon,\mathsf Y^\epsilon)$ are tight in
	$\hat{\mathfrak Z}\times\overline{\mathfrak M}\times\mathbb B_m^{\rm ctrl}$. On a Skorokhod representation space, along a subsequence,
	\[
	(u_n^\epsilon,\overline N_n,\phi_n^\epsilon)
	\longrightarrow(u^\epsilon,\overline N,\phi^\epsilon)
	\]
	almost surely, where the convergence is strong in $\hat{\mathfrak Z}$ for the solution, vague in $\overline{\mathfrak M}$ for the canonical Poisson input, and weak in $L^2(K_m,\nu_T)$ for the control. In particular, $u_n^\epsilon\to u^\epsilon$ strongly in $L^2(0,T;H)$. The uniform energy estimates give, after a further subsequence,
	\[
	u_n^\epsilon\rightharpoonup u^\epsilon
	\quad\text{in }L^2(\widetilde\Omega\times(0,T);V),
	\qquad
	\partial_zu_n^\epsilon\rightharpoonup\partial_zu^\epsilon
	\quad\text{in }L^2(\widetilde\Omega\times(0,T);V),
	\]
	where the second limit is identified by the closedness of the distributional vertical derivative.

	For the controlled drift, fix $e\in V$ and split the difference of the compensator pairings as follows:
	\begin{align*}
	&\int_0^T\!\int_{\mathfrak C_m}
	\bigl(h(u_n^\epsilon(s),\xi)-h(u^\epsilon(s),\xi),e\bigr)
	(\phi_n^\epsilon-1)\,\nu(d\xi)ds\\
	&\quad+
	\int_0^T\!\int_{\mathfrak C_m}
	\bigl(h(u^\epsilon(s),\xi),e\bigr)
	(\phi_n^\epsilon-\phi^\epsilon)\,\nu(d\xi)ds.
	\end{align*}
		The first term tends to zero by \textbf{(B.2)}, the strong $L^2(0,T;H)$ convergence, and the uniform $L^2(K_m,\nu_T)$ bound for the controls. The second tends to zero by the weak convergence of $\phi_n^\epsilon$ because $(s,\xi)\mapsto(h(u^\epsilon(s),\xi),e)$ belongs to $L^2(K_m,\nu_T)$ by \textbf{(B.1)}. Hence the controlled compensator is identified without requiring pointwise convergence of the controls. The nonlinear term is identified by \eqref{PEeq-2.0013}.

		For two
		eigenvectors $e_i,e_j$ and the scalar martingales
		$M_n^{\epsilon,i}=(M_n^\epsilon,e_i)$, the predictable covariation is
		\begin{align}\label{controlled-martingale-bracket}
		 \langle M_n^{\epsilon,i},M_n^{\epsilon,j}\rangle_t
		 =\epsilon\int_0^t\!\int_E
		 &(P_nh(u_n^\epsilon(s),\xi),e_i)
		 (P_nh(u_n^\epsilon(s),\xi),e_j)\notag\\[-2mm]
		 &\hspace{34mm}\times\phi_n^\epsilon(\xi,s)\,\nu(d\xi)ds.
		\end{align}
		On $K_m$, split the difference from the candidate limit by first replacing
		$u_n^\epsilon$ with $u^\epsilon$ and then replacing
		$\phi_n^\epsilon$ with $\phi^\epsilon$.  The first difference converges to
		zero by \textbf{(B.1)}--\textbf{(B.2)}, strong $L^2(0,T;H)$ convergence,
		and $1/m\leq\phi_n^\epsilon\leq m$.  For the second difference, the fixed
		coefficient
		\[
		 (s,\xi)\longmapsto
		 (h(u^\epsilon(s),\xi),e_i)(h(u^\epsilon(s),\xi),e_j)
		\]
		belongs to $L^2(K_m,\nu_T)$: this follows from
		$C_h\in L^4(\nu)$ and the almost-sure $L^\infty(0,T;H)$ energy bound.
		It may therefore be paired with the weakly convergent controls.  On
		$K_m^c$ the controls equal one, and the strong convergence of the states
		treats the remaining difference directly.  Thus
		\eqref{controlled-martingale-bracket} converges to the required bracket;
		the same argument identifies the compensators of the jump
		characteristics.  The martingale characterization then identifies the
		compensated Poisson integral. Thus the limit solves
		\eqref{PEeq-MDPMDPequ-4.40}. Lower semicontinuity yields
		\eqref{PEeq-MDPMDPequ-4.41}--\eqref{PEeq-MDPMDPequ-4.42}, and the equation
		provides an $H$-valued c\`adl\`ag modification.

		Finally, the Galerkin output is temporally compatible with the input
		$\mathsf Y^\epsilon$.  Indeed, $u_n^\epsilon$ is adapted to the canonical
		filtration and $\phi^\epsilon$ is predictable.  If
		$\mathscr F_t^{u_n,\mathsf Y}$ and $\mathscr F_t^{\mathsf Y}$ denote the
		$\sigma$-fields generated up to time $t$, then, for every bounded Borel
		function $F$ of the full input,
		\[
		 \mathbb E\!\left[F(\mathsf Y^\epsilon)\mid
		 \mathscr F_t^{u_n,\mathsf Y}\right]
		 =\mathbb E\!\left[F(\mathsf Y^\epsilon)\mid
		 \mathscr F_t^{\mathsf Y}\right].
		\]
		These are the compatibility identities of
		\cite[Section~2]{Kurtz-ECP-2014}.  For bounded continuous cylinder
		functions and continuity times $t$ they pass to the joint weak limit.
		Right continuity extends the identities to every $t$, and a monotone-class
		argument extends them to all bounded Borel $F$. Hence the limiting pair
		$(u^\epsilon,\mathsf Y^\epsilon)$ is temporally compatible.

	\textbf{Step 3.}
	For two solutions driven by the same controlled random measure, let
	$\widehat u=u_1^\epsilon-u_2^\epsilon$. Using
	\[
	B(u_1^\epsilon,u_1^\epsilon)-B(u_2^\epsilon,u_2^\epsilon)
	=B(u_1^\epsilon,\widehat u)+B(\widehat u,u_2^\epsilon),
	\]
	the first term cancels, while It\^{o}'s formula, \eqref{PEeq-2.103}, and \textbf{(B.2)} give after localization
	\[
	d|\widehat u|^2+\|\widehat u\|^2dt
	\le\Gamma_\epsilon(t)|\widehat u|^2dt+dM_t,
	\]
	with
	\[
	\Gamma_\epsilon(t)=C(1+\|u_2^\epsilon\|^2+|\partial_zu_2^\epsilon|^4)
	+Ca(\epsilon)\!\int_E L_h|\psi^\epsilon|\,d\nu
	+C\epsilon\!\int_E L_h^2\phi^\epsilon\,d\nu .
	\]
	\textbf{Step 1} and the entropy estimates imply $\Gamma_\epsilon\in L^1(0,T)$ almost surely; moreover
	\[
	\int_0^T|\partial_zu_2^\epsilon|^4dt
	\le C\sup_{t\le T}|\partial_zu_2^\epsilon(t)|^2
	\int_0^T\|\partial_zu_2^\epsilon(t)\|^2dt<\infty .
	\]
	The BDG inequality and the stochastic Gronwall lemma therefore give pathwise uniqueness.

	\textbf{Step 4.}
	Let $\Pi_{\mathsf Y^\epsilon}:=\mathscr L(\mathsf Y^\epsilon)$, and let $\mathcal S_{\Gamma,\mathcal C,\Pi_{\mathsf Y^\epsilon}}$ be the set of joint laws of output--input pairs $(u,\mathsf Y^\epsilon)$ satisfying the variational form of \eqref{PEeq-MDPMDPequ-4.40}, the fixed input law $\Pi_{\mathsf Y^\epsilon}$, and the temporal compatibility identities from \textbf{Step 2}. By \textbf{Step 2} and \textbf{3}, this set is nonempty and jointly compatible solutions driven by the same input are pathwise unique. Hence \cite[Theorem~1.5 and Lemma~2.10]{Kurtz-ECP-2014} give a strong solution and joint uniqueness in law. Consequently, the solution is represented by a Borel function of $(\overline N,\phi^\epsilon)$, whose evaluation on the original stochastic basis yields the required probabilistically strong solution.
\end{proof}

\begin{lemma}\label{lem:measurable-solution-maps}
	For every fixed $\epsilon\in(0,\epsilon_0)$, there exist Borel maps with the following properties:
	\begin{enumerate}
		\item[$(1)$] A map $\mathscr G^\epsilon:\mathfrak M\to\mathfrak U$ represents the uncontrolled fluctuation equation:
		\[
		\mathcal Y^\epsilon=\mathscr G^\epsilon(\epsilon N^{\epsilon^{-1}})
		\quad\mathbb P\text{-a.s.}
		\]
		Moreover, for every bounded predictable control
		$\phi^\epsilon\in\mathcal U_{+,\epsilon}^M$, the same map evaluated at
		$\epsilon N^{\epsilon^{-1}\phi^\epsilon}$ represents the unique controlled fluctuation associated with \eqref{PEeq-MDPMDPequ-4.40}.
		\item[$(2)$] There is a Borel residual map
		\[
		\mathscr R^\epsilon:\mathfrak U_z\times
		\mathbb B_\nu((M\kappa_2(1))^{1/2})
		\longrightarrow\mathfrak U_z
		\]
		such that $\mathscr R^\epsilon(Q,q)$ is the unique solution of the residual equation obtained from \eqref{PEeq-MDPMDPequ-4.76} by replacing
		$(\mathscr Q^\epsilon,
		\psi^\epsilon\mathbf1_{\{|\psi^\epsilon|\le\beta/a(\epsilon)\}})$
		with $(Q,q)$.
	\end{enumerate}
\end{lemma}
\begin{proof}
	For part~$(1)$, we apply \cite[Theorem~1.5]{Kurtz-ECP-2014} to the
	uncontrolled equation, with the canonical Poisson random measure as input
	and the output taking values in the Polish space $\mathfrak U$.
	Proposition~\ref{prop-controlled-wp}, with $\phi^\epsilon\equiv1$, gives
	compatible weak existence and pathwise uniqueness among jointly compatible
	solutions. Consequently, the strong solution is a Borel function of the
	input. Composing the resulting solution map with
	$u\mapsto(u-v^0)/a(\epsilon)$ gives $\mathscr G^\epsilon$.
	If $\phi^\epsilon\in\mathcal U_{+,\epsilon}^M$, then
	$\phi^\epsilon\in\overline{\mathfrak R}_{b,m}$ for some $m$. The
	Poisson--Girsanov change of measure shows that the same map, evaluated at
	$\epsilon N^{\epsilon^{-1}\phi^\epsilon}$, gives a solution of the
	controlled equation. The required identification then follows from the
	uniqueness in Proposition~\ref{prop-controlled-wp}.
	
	For part~$(2)$, fix
	\[
	(Q,q)\in\mathfrak U_z\times
	\mathbb B_\nu((M\kappa_2(1))^{1/2}).
	\]
	The corresponding residual equation is a deterministic primitive equation
	with prescribed path $Q$ and forcing
	\[
	F_q(t)=\int_E h(v^0(t),\xi)q(\xi,t)\,\nu(d\xi).
	\]
	By \textbf{(B.1)} and \textbf{(B.3)}, $F_q$ has the required $H$ and
	vertical regularity. A standard spectral Galerkin argument, based on the
	nonlinear estimates and cancellations in Lemma~\ref{PEeq-lem2.1}, gives a
	unique solution $Y\in\mathfrak U_z$. Here the vertical regularity is obtained
	by testing the original Galerkin equation with $-\partial_{zz}Y_n$; no
	separate $H$-valued equation or vertical-average condition is imposed on
	$\partial_zY$.
	It remains to verify measurability. The graph of the solution relation is
	Borel. Indeed, the variational equation may be tested against a countable
	dense subset of $V$ at rational times; by the continuity of its integral
	form in $V'$, these identities determine the equation at every time. The
	dependence on $Q$ is Borel in the strong path topology, while the dependence
	on $q$ is continuous on the weakly topologized control ball. Since the
	residual equation has a unique solution for every $(Q,q)$, the
	Lusin--Novikov theorem yields a Borel map
	$
	\mathscr R^\epsilon:
	\mathfrak U_z\times
	\mathbb B_\nu((M\kappa_2(1))^{1/2})
	\longrightarrow\mathfrak U_z$.
	This proves part~$(2)$.
\end{proof}

\begin{lemma}\label{PEeq-MDPMDPlem-4.15}
	Suppose Hypotheses \ref{PEeqAssum-2.3} and \ref{PEeqAssum-2.5} hold. Let $u^{\epsilon}$ and $v^0$ solve \eqref{PEeq-MDPMDPequ-4.40} and \eqref{PEeq-MDPdet2.15}, respectively. Then
	\begin{align}\label{PEeq-MDPMDPequ-4.51}
		\lim\limits_{\epsilon\rightarrow 0}\left(\mathbb{E}\left[\sup_{t\in [0,T]}|u^{\epsilon}(t)-v^0(t)|^2\right]+\int_{0}^{T}\mathbb{E}\left[\|u^{\epsilon}(t)-v^0(t)\|^2\right]dt\right)=0.
	\end{align}
\end{lemma}
\begin{proof}
	Set $\widehat{y}(t)=u^{\epsilon}(t)-v^0(t)$. Then $\widehat{y}$ satisfies
	\begin{align}\label{PEeq-MDPMDPequ-4.52}
		\begin{split}
			&~~d\widehat{y}(t)+A\widehat{y}(t)dt+B(u^{\epsilon}(t),\widehat{y}(t))dt+B(\widehat{y}(t),v^0(t))dt\\
			&=\epsilon\int_{E}h(u^{\epsilon}(t-),\xi)\widetilde{N}^{\epsilon^{-1}\phi^\epsilon}(dt,d\xi)+\int_{E} h(u^{\epsilon}(t),\xi)(\phi^\epsilon(\xi,t)-1)\nu(d\xi)dt
		\end{split}
	\end{align}
	with initial value $\widehat{y}(0)=0$.
	
	Applying It\^{o}'s formula to $|\widehat{y}(t)|^2$ and using \eqref{PEeq-MDPMDPequ-4.52} and \eqref{PEeq-2.13} gives
%	\begin{align}\label{PEeq-MDPMDPequ-4.53}
%		\begin{split}
%			&~|\widehat{y}(t)|^2+2\int_0^t \|\widehat{y}(s)\|^2ds\\
%			&=\underbrace{-2\int_0^t \left(B(u^{\epsilon}(s),\widehat{y}(s)),\widehat{y}(s)\right)ds}_{=0}-2\int_0^t\left(B(\widehat{y}(s),v^0(s)),\widehat{y}(s)\right)ds\\
%			&+2\epsilon\int_0^t\int_{E}\left(h(u^{\epsilon}(s-),\xi),\widehat{y}(s)\right)\widetilde{N}^{\epsilon^{-1}\phi^\epsilon}(ds,d\xi)+\epsilon^2\int_0^t\int_{E}|h(u^{\epsilon}(s-),\xi)|^2{N}^{\epsilon^{-1}\phi^\epsilon}(ds,d\xi)\\
%			&+2\int_0^t\int_{E} \left(h(u^{\epsilon}(s),\xi)(\phi^\epsilon(\xi,s)-1),\widehat{y}(s)\right)\nu(d\xi)ds.
%		\end{split}
%	\end{align}
	\begin{align}\label{PEeq-MDPMDPequ-4.53}
	\begin{split}
		&~|\widehat{y}(t)|^2+2\int_0^t \|\widehat{y}(s)\|^2ds\\
		&=-2\int_0^t\left(B(\widehat{y}(s),v^0(s)),\widehat{y}(s)\right)ds+2\int_0^t\int_{E} \left(h(u^{\epsilon}(s),\xi)(\phi^\epsilon(\xi,s)-1),\widehat{y}(s)\right)\nu(d\xi)ds\\
		&+2\epsilon\int_0^t\int_{E}\left(h(u^{\epsilon}(s-),\xi),\widehat{y}(s)\right)\widetilde{N}^{\epsilon^{-1}\phi^\epsilon}(ds,d\xi)+\epsilon^2\int_0^t\int_{E}|h(u^{\epsilon}(s-),\xi)|^2{N}^{\epsilon^{-1}\phi^\epsilon}(ds,d\xi)\\
		&:= \sum_{j=7}^{10}\Upsilon_{j}(t).
	\end{split}
\end{align}	
	For $\Upsilon_{7}(t)$, by \eqref{PEeq-2.11} we get
	\begin{align}\label{PEeq-MDPMDPequ-4.54}
		\begin{split}
			|\Upsilon_{7}(t)|
			&\leq C\int_0^t \left(|\widehat{y}(s)|\|v^0(s)\|\|\widehat{y}(s)\|+|\partial_{z}v^0(s)||\widehat{y}(s)|^{1/2}\|\widehat{y}(s)\|^{3/2}\right)ds\\
			&\leq \frac{1}{2}\int_0^t\|\widehat{y}(s)\|^2 ds + C\int_0^t \left(\|v^0(s)\|^2+|\partial_{z}v^0(s)|^4\right)|\widehat{y}(s)|^2ds.
		\end{split}
	\end{align}
	Let $\psi^\epsilon(\xi,s)=\frac{\phi^\epsilon(\xi,s)-1}{a(\epsilon)} \in \mathcal{U}_{\epsilon}^M$. Then, for $\Upsilon_{8}(t)$, by conditions \textbf{(B.1)} and \textbf{(B.2)} in Hypothesis \ref{PEeqAssum-2.3} we have
	\begin{align}\label{PEeq-MDPMDPequ-4.55}
%		\begin{split}
			\left|\Upsilon_{8}(t)\right|
			&\leq 2a(\epsilon)\int_0^t |\widehat{y}(s)|\int_{E}|h(u^{\epsilon}(s),\xi)-h(v^{0}(s),\xi)||\psi^\epsilon(\xi,s)|\nu(d\xi)ds\notag \\
			&+2a(\epsilon)\int_0^t |\widehat{y}(s)|\int_{E}|h(v^{0}(s),\xi)||\psi^\epsilon(\xi,s)|\nu(d\xi)ds\notag \\
			&\leq 2a(\epsilon)\int_0^t |\widehat{y}(s)|^2\int_{E}L_{h}(\xi)|\psi^\epsilon(\xi,s)|\nu(d\xi)ds\notag \\
			&+2a(\epsilon)\int_0^t |\widehat{y}(s)|(1+|v^{0}(s)|)\int_{E}C_{h}(\xi)|\psi^\epsilon(\xi,s)|\nu(d\xi)ds\notag \\
			&\leq 2a(\epsilon)\int_0^t |\widehat{y}(s)|^2\int_{E}L_{h}(\xi)|\psi^\epsilon(\xi,s)|\nu(d\xi)ds\notag \\
			&+a(\epsilon)\int_0^t (1+|\widehat{y}(s)|^2)(1+|v^{0}(s)|)\int_{E}C_{h}(\xi)|\psi^\epsilon(\xi,s)|\nu(d\xi)ds\notag \\
			&\leq a(\epsilon)\int_0^t |\widehat{y}(s)|^2 \left(2\int_{E}L_{h}(\xi)|\psi^\epsilon(\xi,s)|\nu(d\xi)+\left(1+\sup_{t\in [0,T]}|v^{0}(t)|\right)\int_{E}C_{h}(\xi)|\psi^\epsilon(\xi,s)|\nu(d\xi)\right)ds\notag \\
			&+\underbrace{a(\epsilon)\left(1+\sup_{t\in [0,T]}|v^{0}(t)|\right)\int_0^t \int_{E}C_{h}(\xi)|\psi^\epsilon(\xi,s)|\nu(d\xi)ds}_{\Upsilon_{11}(t)},
%		\end{split}
	\end{align}	
	where the third inequality uses $2|\widehat{y}(s)|\leq 1+|\widehat{y}(s)|^2$ in the third inequality.

Set
\begin{align*}
	\mathfrak{Z}(s)&=a(\epsilon)\left(2\int_{E}L_{h}(\xi)|\psi^\epsilon(\xi,s)|\nu(d\xi)+\left(1+\sup_{t\in [0,T]}|v^{0}(t)|\right)\int_{E}C_{h}(\xi)|\psi^\epsilon(\xi,s)|\nu(d\xi)\right)\\
	&+C(\|v^0(s)\|^2+|\partial_{z}v^0(s)|^4).
\end{align*}
Combining \eqref{PEeq-MDPMDPequ-4.53}-\eqref{PEeq-MDPMDPequ-4.55} derives
\begin{align*}
|\widehat{y}(t)|^2+\frac{3}{2}\int_0^t \|\widehat{y}(s)\|^2ds
\leq \sup_{t\in [0,T]}\left(\left|\Upsilon_{9}(t)\right| +\Upsilon_{10}(t)+\Upsilon_{11}(t)\right)
+\int_0^t \mathfrak{Z}(s)|\widehat{y}(s)|^2ds,
\end{align*}
which, together with Gronwall's inequality, yields
\begin{align}\label{PEeq-MDPMDPequ-4.56}
|\widehat{y}(t)|^2+\int_0^t \|\widehat{y}(s)\|^2ds\leq \left(\sup_{t\in [0,T]}\left|\Upsilon_{9}(t)\right|+\Upsilon_{10}(T)+\Upsilon_{11}(T) \right)e^{\int_0^{T}\mathfrak{Z}(t)dt}.
\end{align}
By \eqref{PEeq-MDPderm-4.10} and Lemma \ref{PEeq-MDPMDPlem-4.8}, we obtain
{\small
\begin{align*}
{\int_0^{T}\mathfrak{Z}(t)dt}
&\leq {a(\epsilon)\left(2\int_0^{T}\int_{E}L_{h}(\xi)|\psi^\epsilon(\xi,t)|\nu(d\xi)dt+\left(1+\sup_{t\in [0,T]}|v^{0}(t)|\right)\int_0^{T}\int_{E}C_{h}(\xi)|\psi^\epsilon(\xi,t)|\nu(d\xi)dt\right)}\\
& +{C\left(\int_0^{T}\|v^0(t)\|^2dt+\sup_{t\in [0,T]}|\partial_{z}v^0(t)|^2\int_0^{T}\|\partial_{z}v^0(t)\|^2dt\right)}\\
&\leq C_Ta(\epsilon)\left((\rho_{{C}_{h}}(\beta)+\rho_{{L}_{h}}(\beta)) \sqrt{T}+(\ell_{{C}_{h}}(\beta)+\ell_{L_{h}}(\beta))\sqrt{a(\epsilon)} + (\varsigma_{{C}_{h}}(\epsilon)+\varsigma_{L_{h}}(\epsilon)) \big( 1 + T \big)\right)+C_T.
\end{align*}
}
Hence,
\begin{align}\label{PEeq-MDPMDPequ-4.57}
e^{\int_0^{T}\mathfrak{Z}(t)dt} \leq C_{T}.
\end{align}

	As in \eqref{PEeq-MDPMDPequ-4.49} and \eqref{PEeq-MDPMDPequ-4.50}, by the BDG inequality and \eqref{PEeq-MDPMDPequ-4.41} we derive
	\begin{align}\label{PEeq-MDPMDPequ-4.58}
		\begin{split}
			\mathbb{E}\left[\sup_{t\in [0,T]}\left|\Upsilon_{9}(t)\right|\right]
			&\leq \frac{1}{2}\mathbb{E}\left[\sup_{t\in [0,T]}|\widehat{y}(t)|^2\right]+ 8\epsilon\textbf{c}_{{C}_{h}} \big( a^2(\epsilon) + T \big) \mathbb{E}\left[\left(1+\sup_{t\in [0,T]}|u^{\epsilon}(t)|^2\right)\right]\\
			&\leq \frac{1}{2}\mathbb{E}\left[\sup_{t\in [0,T]}|\widehat{y}(t)|^2\right]+ C_{T}\textbf{c}_{{C}_{h}} \big( a^2(\epsilon) + T \big)\epsilon,
		\end{split}
	\end{align}
	and
	\begin{align}\label{PEeq-MDPMDPequ-4.59}
		\mathbb{E}\left[\Upsilon_{10}(T)\right]\leq C_{T}\textbf{c}_{{C}_{h}} \big( a^2(\epsilon) + T \big)\epsilon.
	\end{align}
	By Lemma \ref{PEeq-MDPMDPlem-4.8} and \eqref{PEeq-MDPderm-4.10} we have
	\begin{align}\label{PEeq-MDPMDPequ-4.60}
			\mathbb{E}\left[\Upsilon_{11}(T)\right]\leq C_{T}a(\epsilon)\left(\rho_{{C}_{h}}(\beta) \sqrt{T}+\ell_{{C}_{h}}(\beta)\sqrt{a(\epsilon)} + \varsigma_{{C}_{h}}(\epsilon) \big( 1 + T \big)\right)
	\end{align}
	
	Together with \eqref{PEeq-MDPMDPequ-4.56}-\eqref{PEeq-MDPMDPequ-4.60}, we have
	\begin{align*}
		&~~\mathbb{E}\left[\sup_{t \in [0,T]}|\widehat{y}(t)|^2\right]
		+\int_0^T \mathbb{E}\left[\|\widehat{y}(t)\|^2\right] dt\\
		&\leq C_{T}\left(\textbf{c}_{C_{h}} \big( a^2(\epsilon) + T \big)+\left(\rho_{C_{h}}(\beta) \sqrt{T}+\ell_{C_{h}}(\beta)\sqrt{a(\epsilon)} + \varsigma_{{C}_{h}}(\epsilon) \big( 1 + T \big)\right)\right)\left(\epsilon+a(\epsilon)\right)\rightarrow 0 \text{~~as~~} \epsilon \rightarrow 0,
	\end{align*}
	as desired. This completes the proof.
\end{proof}

Denote
$$
\mathscr{G}^{\epsilon}(\epsilon N^{\epsilon^{-1}\phi^{\epsilon}}):=\mathscr{Z}^{\epsilon}=\frac{u^{\epsilon}-v^0}{a(\epsilon)},
$$
then $\mathscr{Z}^{\epsilon}$ satisfies 
\begin{align}\label{PEeq-MDPMDPequ-4.62}
	\begin{split}
		&~~d\mathscr{Z}^{\epsilon}(t)+A\mathscr{Z}^{\epsilon}(t)dt+B(\mathscr{Z}^{\epsilon}(t),v^{0}(t))dt+B(a(\epsilon)\mathscr{Z}^{\epsilon}(t)+v^0(t),\mathscr{Z}^{\epsilon}(t))dt\\
		&=\frac{\epsilon}{a(\epsilon)}\int_{E}h(a(\epsilon)\mathscr{Z}^{\epsilon}(t-)+v^0(t-),\xi)\widetilde{N}^{\epsilon^{-1}\phi^{\epsilon}}(dt,d\xi)\\
		&+\frac{1}{a(\epsilon)}\int_{E}h(a(\epsilon)\mathscr{Z}^{\epsilon}(t)+v^0(t),\xi)(\phi^{\epsilon}(\xi,t)-1)\nu(d\xi)dt,
	\end{split}
\end{align}
with initial value $\mathscr{Z}^{\epsilon}(0)=0$.

\begin{lemma}\label{PEeq-MDP-controlled-convergence}
	Fix \( M \in (0,\infty) \), and let \( \{\phi^\epsilon\}_{\epsilon>0} \) be such that \( \phi^\epsilon \in \mathcal{U}_{+, \epsilon}^M \) for every \(\epsilon > 0 \). Let  
	\[
	\psi^\epsilon = (\phi^\epsilon - 1)/a(\epsilon) \quad \text{and} \quad \beta \in (0, 1].
	\]
	Then the family 
	$\left\{ \mathscr{Z}^{\epsilon}, \, \psi^\epsilon \mathbf{1}_{\{|\psi^\epsilon| \leq \beta/a(\epsilon)\}} \right\}_{\epsilon>0}$
	is tight in 
	$\left(\mathcal{D}([0, T], H)\cap L^2([0,T],V)\right) \times \mathbb{B}_{\nu}((M \kappa_2(1))^{1/2})$,
	and any limit point \( (\mathscr{Z}, \psi) \) solves equation \eqref{PEeq-MDPskeequ-4.8}.
\end{lemma}
\begin{proof}
	The proof is divided into the following two steps. The controlled fluctuation $\mathscr Z^\epsilon=(u^\epsilon-v^0)/a(\epsilon)$ satisfies \eqref{PEeq-MDPMDPequ-4.62}, which is equivalently written as
	\begin{align*}
	\begin{split}
	&~~d\mathscr Z^\epsilon+A\mathscr Z^\epsilon dt
	+B(\mathscr Z^\epsilon,v^0)dt+B(v^0,\mathscr Z^\epsilon)dt +a(\epsilon)B(\mathscr Z^\epsilon,\mathscr Z^\epsilon)dt\\
	&=\frac{\epsilon}{a(\epsilon)}\int_Eh(u^\epsilon(t-),\xi)
	\widetilde N^{\epsilon^{-1}\phi^\epsilon}(dt,d\xi) +\int_Eh(u^\epsilon(t),\xi)\psi^\epsilon(\xi,t)\nu(d\xi)dt.
	\end{split}
	\end{align*}
	with $\mathscr Z^\epsilon(0)=0$.
	
	\textbf{Step 1.} Let 
	$\mathscr{X}^{\epsilon}$, $\mathscr{H}^{\epsilon}$ and $\mathscr{S}^{\epsilon}$ be the solutions of the following equations, respectively:
	\begin{align}
		d\mathscr{X}^{\epsilon}(t)+A\mathscr{X}^{\epsilon}(t)&dt=\frac{\epsilon}{a(\epsilon)}\int_{E} h(u^{\epsilon}(t-),\xi)\widetilde{N}^{\epsilon^{-1}\phi^\epsilon}(dt,d\xi), \quad \mathscr{X}^{\epsilon}(0)=0,\label{PEeq-MDPMDPequ-4.63}\\
		d\mathscr{H}^{\epsilon}(t)+A\mathscr{H}^{\epsilon}(t)dt&=\int_{E} h(u^{\epsilon}(t),\xi)\psi^\epsilon(\xi,t)\mathbf{1}_{\{|\psi^\epsilon| > \beta/a(\epsilon)\}} \nu(d\xi) dt, \quad \mathscr{H}^{\epsilon}(0)=0,\label{PEeq-MDPMDPequ-4.64}
	\end{align} 
	and
	\begin{align}\label{PEeq-MDPMDPequ-4.65}
		d\mathscr{S}^{\epsilon}(t)+A\mathscr{S}^{\epsilon}(t)dt=\int_{E} (h(u^{\epsilon}(t),\xi)-h(v^{0}(t),\xi))\psi^\epsilon(\xi,t)\mathbf{1}_{\{|\psi^\epsilon| \leq  \beta/a(\epsilon)\}} \nu(d\xi) dt, \quad \mathscr{S}^{\epsilon}(0)=0.
	\end{align} 

Applying It\^{o}'s formula to $|\mathscr{X}^{\epsilon}(t)|^2+|\partial_{z}\mathscr{X}^{\epsilon}(t)|^2$ and using \eqref{PEeq-MDPMDPequ-4.63} gives
\begin{align}\label{PEeq-MDPMDPequ-4.66}
		&~\mathbb{E}\left[\sup_{t \in [0,T]} \left(|\mathscr{X}^{\epsilon}(t)|^2+|\partial_{z}\mathscr{X}^{\epsilon}(t)|^2\right)\right]+2\mathbb{E}\left[\int_{0}^{T} \|\mathscr{X}^{\epsilon}(s)\|^2+\|\partial_{z}\mathscr{X}^{\epsilon}(s)\|^2ds\right]\notag \\
		&\leq \frac{2\epsilon}{a(\epsilon)}\mathbb{E}\left[\sup_{t \in [0,T]}\left|\int_{0}^{t}\int_{E} \left(h(u^{\epsilon}(s-),\xi),\mathscr{X}^{\epsilon}(s)\right)\widetilde{N}^{\epsilon^{-1}\phi^\epsilon}(ds,d\xi)\right|\right]\notag \\
		&+\frac{2\epsilon}{a(\epsilon)}\mathbb{E}\left[\sup_{t \in [0,T]}\left|\int_{0}^{t}\int_{E} \left(h(u^{\epsilon}(s-),\xi),\partial_{zz}\mathscr{X}^{\epsilon}(s)\right)\widetilde{N}^{\epsilon^{-1}\phi^\epsilon}(ds,d\xi)\right|\right]\notag \\
		&+\frac{\epsilon^2}{a^2(\epsilon)}\mathbb{E}\left[\int_{0}^{T}\int_{E}|h(u^{\epsilon}(s),\xi)|^2{N}^{\epsilon^{-1}\phi^\epsilon}(ds,d\xi)\right]+\frac{\epsilon^2}{a^2(\epsilon)}\mathbb{E}\left[\int_{0}^{T}\int_{E}|\partial_{z}h(u^{\epsilon}(s),\xi)|^2{N}^{\epsilon^{-1}\phi^\epsilon}(ds,d\xi)\right]\notag \\
		&:=\sum_{j=12}^{15}\Upsilon_{j}.
\end{align}
	By the BDG inequality, Lemma \ref{PEeq-MDPMDPlem-4.7} and \eqref{PEeq-MDPMDPequ-4.41}-\eqref{PEeq-MDPMDPequ-4.42}, we infer
	\begin{align}\label{PEeq-MDPMDPequ-4.67}
		\begin{split}
			\Upsilon_{12}+\Upsilon_{13}
			&\leq \frac{1}{2}\mathbb{E}\left[\sup_{t \in [0,T]} \left(|\mathscr{X}^{\epsilon}(t)|^2+|\partial_{z}\mathscr{X}^{\epsilon}(t)|^2\right)\right]\\
			&+\frac{C\epsilon^2}{a^2(\epsilon)}\mathbb{E}\left[\int_{0}^{T}\int_{E} \left(|h(u^{\epsilon}(s),\xi)|^2+|\partial_{z}h(u^{\epsilon}(s),\xi)|^2\right){N}^{\epsilon^{-1}\phi^\epsilon}(ds,d\xi)\right]\\
			&\leq \frac{1}{2}\mathbb{E}\left[\sup_{t \in [0,T]} \left(|\mathscr{X}^{\epsilon}(t)|^2+|\partial_{z}\mathscr{X}^{\epsilon}(t)|^2\right)\right]
			+\frac{C_{\epsilon_{0},T}\epsilon}{a^2(\epsilon)}\left(\textbf{c}_{C_{h}}+\textbf{c}_{\widetilde{C}_{h}}\right) \big( a^2(\epsilon) + T \big),
		\end{split}
	\end{align}
and, by the compensator identity once more,
\begin{align}\label{PEeq-MDPMDPequ-4.68}
	\Upsilon_{14}+\Upsilon_{15}\leq \frac{C_{\epsilon_{0},T}\epsilon}{a^2(\epsilon)}\left(\textbf{c}_{C_{h}}+\textbf{c}_{\widetilde{C}_{h}}\right) \big( a^2(\epsilon) + T \big).
\end{align}
Combining \eqref{PEeq-MDPMDPequ-4.66}-\eqref{PEeq-MDPMDPequ-4.68} yields
	 \begin{align}\label{PEeq-MDPMDPequ-4.69}
	 	\begin{split}
	 		&~~\mathbb{E}\left[\sup_{t \in [0,T]} \left(|\mathscr{X}^{\epsilon}(t)|^2+|\partial_{z}\mathscr{X}^{\epsilon}(t)|^2\right)\right]+\int_{0}^{T}\mathbb{E}\left[ \|\mathscr{X}^{\epsilon}(t)\|^2+\|\partial_{z}\mathscr{X}^{\epsilon}(t)\|^2\right]dt \\
	 		&\leq \frac{C_{\epsilon_{0},T}\epsilon}{a^2(\epsilon)}\left(\textbf{c}_{C_{h}}+\textbf{c}_{\widetilde{C}_{h}}\right) \big( a^2(\epsilon) + T \big) \rightarrow 0 \text{~~as~~} \epsilon\rightarrow 0.
	 	\end{split}
	 \end{align}
	
Since $\psi^{\epsilon}={(\phi^{\epsilon}-1)}/{a(\epsilon)}$, testing \eqref{PEeq-MDPMDPequ-4.64} with $\mathscr H^\epsilon$ and, after vertical differentiation, with $\partial_z\mathscr H^\epsilon$ yields, for every $t\in[0,T]$, 
\begin{align*}
	&|\mathscr{H}^{\epsilon}(t)|^2+|\partial_{z}\mathscr{H}^{\epsilon}(t)|^2+2\int_{0}^{t} \left(\|\mathscr{H}^{\epsilon}(s)\|^2+\|\partial_{z}\mathscr{H}^{\epsilon}(s)\|^2\right)ds \notag \\
	&\leq 2\int_0^{T} \int_{E}|h(u^{\epsilon}(s),\xi)| |\psi^\epsilon(\xi,s)|\mathbf{1}_{\{|\psi^\epsilon| > \beta/a(\epsilon)\}}|\mathscr{H}^{\epsilon}(s)| \nu(d\xi) ds \notag \\
	&+2\int_0^{T} \int_{E}|\partial_{z}h(u^{\epsilon}(s),\xi)| |\psi^\epsilon(\xi,s)|\mathbf{1}_{\{|\psi^\epsilon| > \beta/a(\epsilon)\}}|\partial_{z}\mathscr{H}^{\epsilon}(s)| \nu(d\xi) ds \notag \\
	&\leq \frac{1}{2}\sup_{t \in [0,T]}\left(|\mathscr{H}^{\epsilon}(t)|^2+|\partial_{z}\mathscr{H}^{\epsilon}(t)|^2\right)
	+C_{\epsilon_{0},T}\left(\int_0^{T} \int_{E}\left(C_{h}(\xi)+\widetilde{C}_{h}(\xi)\right)|\psi^\epsilon(\xi,s)|\mathbf{1}_{\{|\psi^\epsilon| > \beta/a(\epsilon)\}}\nu(d\xi) ds\right)^2,
\end{align*}
Together with Lemma~\ref{PEeq-MDPMDPlem-4.9}, this yields
\begin{align}\label{PEeq-MDPMDPequ-4.70}
	\begin{split}
		&~~\mathbb{E}\left[\sup_{t \in [0,T]}\left(|\mathscr{H}^{\epsilon}(t)|^2+|\partial_{z}\mathscr{H}^{\epsilon}(t)|^2\right)\right]+\int_{0}^{T} \mathbb{E}\left[\|\mathscr{H}^{\epsilon}(t)\|^2+\|\partial_{z}\mathscr{H}^{\epsilon}(t)\|^2\right]dt\\
		&\leq 
		C_{\epsilon_{0},T}\left(\sup_{\psi^\epsilon\in \mathcal{S}_{\epsilon}^M}\int_0^{T} \int_{E}\left(C_{h}(\xi)+\widetilde{C}_{h}(\xi)\right)|\psi^\epsilon(\xi,t)|\mathbf{1}_{\{|\psi^\epsilon| > \beta/a(\epsilon)\}}\nu(d\xi) dt\right)^2 \rightarrow 0 \text{~~as~~} \epsilon\rightarrow 0.
	\end{split}
\end{align}

Similarly, by the chain rule, Lemma \ref{PEeq-MDPMDPlem-4.8}, Lemma \ref{PEeq-MDPMDPlem-4.15} and condition \textbf{(B.2)} in Hypothesis \ref{PEeqAssum-2.3}, we deduce
\begin{align}\label{PEeq-MDPMDPequ-4.71}
	\begin{split}
		&~~\mathbb{E}\left[\sup_{t \in [0,T]}|\mathscr{S}^{\epsilon}(t)|^2\right]+\int_{0}^{T} \mathbb{E}\left[\|\mathscr{S}^{\epsilon}(t)\|^2\right]dt\\
		&\leq 
		C_{\epsilon_{0},T}\mathbb{E}\left[\sup_{t\in [0,T]}|u^{\epsilon}(t)-v^0(t)|^2\right]\left(\sup_{\psi^\epsilon \in \mathcal{S}_{\epsilon}^M}\int_0^{T} \int_{E}L_{h}(\xi)|\psi^\epsilon(\xi,t)|\nu(d\xi) dt\right)^2 \rightarrow 0 \text{~~as~~} \epsilon\rightarrow 0.
	\end{split}
\end{align}
To estimate the vertical derivative, we work first at the Galerkin level and test the equation for $\mathscr S^\epsilon$ directly by $-\partial_{zz}\mathscr S^\epsilon$.  Passing to the limit by weak lower semicontinuity gives
\begin{align}\label{PEeq-MDPMDPequ-4.72}
	\begin{split}
		&~~\frac{1}{2}|\partial_{z}\mathscr{S}^{\epsilon}(t)|^2+\int_{0}^{t} \|\partial_{z}\mathscr{S}^{\epsilon}(s)\|^2ds\\
		&=\int_{0}^{t}\int_{E}\left(h(u^{\epsilon}(s),\xi)-h(v^{0}(s),\xi),\partial_{zz}\mathscr{S}^{\epsilon}(s)\right)\psi^\epsilon(\xi,s)\mathbf{1}_{\{|\psi^\epsilon| \leq  \beta/a(\epsilon)\}} \nu(d\xi) ds.
	\end{split}
\end{align}
Using \textbf{(B.2)}, the Cauchy-Schwarz inequality, and
$|\partial_{zz}w|\le C\|\partial_zw\|$, we estimate the right-hand side of \eqref{PEeq-MDPMDPequ-4.72},
\begin{align}\label{PEeq-MDPMDPequ-4.00073}
	\begin{split}
		&~~\left|\int_{E}\left(h(u^{\epsilon}(s),\xi)-h(v^{0}(s),\xi),\partial_{zz}\mathscr{S}^{\epsilon}(s)\right)\psi^\epsilon(\xi,s)\mathbf{1}_{\{|\psi^\epsilon| \leq  \beta/a(\epsilon)\}} \nu(d\xi)\right|\\
		&
		\le C|u^\epsilon(s)-v^0(s)|\|\partial_z\mathscr S^\epsilon(s)\| \int_{E} L_h(\xi)\left|\psi^\epsilon(\xi,s)\right| \mathbf{1}_{\{|\psi^\epsilon| \leq  \beta/a(\epsilon)\}} \nu(d\xi)\\
		&\le \frac12\|\partial_z\mathscr S^\epsilon(s)\|^2
		+C|u^\epsilon(s)-v^0(s)|^2\left(\int_{E} L_h(\xi)\left|\psi^\epsilon(\xi,s)\right| \mathbf{1}_{\{|\psi^\epsilon| \leq  \beta/a(\epsilon)\}} \nu(d\xi)\right)^2.
	\end{split}
\end{align}
This estimate is first justified for the Galerkin approximations, where $\partial_{zz}\mathscr S_n^\epsilon\in H$, and then passed to the limit by weak lower semicontinuity.  The truncated control estimate from \cite[Lemma 3.2]{Budhiraja-AOP-2016} yields pathwise
\begin{align}\label{PEeq-MDPMDPequ-4.73}
	\begin{split}
		&\int_0^T \left(\int_{E} L_h(\xi)\left|\psi^\epsilon(\xi,s)\right| \mathbf{1}_{\{|\psi^\epsilon| \leq  \beta/a(\epsilon)\}} \nu(d\xi)\right)^2ds \\
		&\leq \|L_h\|_{L^2(\nu)}^2
		\int_0^T\!\int_E|\psi^\epsilon(\xi,s)|^2
		\mathbf1_{\{|\psi^\epsilon|\le\beta/a(\epsilon)\}}\,d\nu ds\le M\kappa_2(1)\|L_h\|_{L^2(\nu)}^2 
	\end{split}
\end{align}

Combining \eqref{PEeq-MDPMDPequ-4.72}-\eqref{PEeq-MDPMDPequ-4.73} yields
\begin{align}\label{PEeq-MDPMDPequ-4.74}
		\sup_{t \in [0,T]}|\partial_z\mathscr S^\epsilon(t)|^2
		+\int_0^T\|\partial_z\mathscr S^\epsilon(s)\|^2ds
		\le C\sup_{t \in [0,T]}|u^\epsilon(t)-v^0(t)|^2.
\end{align}
Taking expectations and using Lemma~\ref{PEeq-MDPMDPlem-4.15}, then combining with \eqref{PEeq-MDPMDPequ-4.71}, gives
\begin{align}\label{PEeq-MDPMDPequ-4.75}
	\lim_{\epsilon \to 0} \left(\mathbb{E}\left[\sup_{t \in [0,T]}\left(|\mathscr{S}^{\epsilon}(t)|^2+|\partial_{z}\mathscr{S}^{\epsilon}(t)|^2\right)\right]+\int_{0}^{T} \mathbb{E}\left[\left(\|\mathscr{S}^{\epsilon}(t)\|^2+\|\partial_{z}\mathscr{S}^{\epsilon}(t)\|^2\right)\right]dt\right)=0.
\end{align}
	
	\textbf{Step 2.}
	Let
	$\mathscr Q^\epsilon=\mathscr X^\epsilon+\mathscr H^\epsilon+\mathscr S^\epsilon$
	and
	$\mathfrak Y^\epsilon=\mathscr Z^\epsilon-\mathscr Q^\epsilon$.
	By \eqref{PEeq-MDPMDPequ-4.62}--\eqref{PEeq-MDPMDPequ-4.65},
	\begin{align}\label{PEeq-MDPMDPequ-4.76}
		\begin{split}
			&d\mathfrak Y^\epsilon(t)+A\mathfrak Y^\epsilon(t)dt
			+B(\mathfrak Y^\epsilon(t)+\mathscr Q^\epsilon(t),v^0(t))dt\\
			&\quad
			+a(\epsilon)B(\mathfrak Y^\epsilon(t)+\mathscr Q^\epsilon(t),
			\mathfrak Y^\epsilon(t)+\mathscr Q^\epsilon(t))dt
			+B(v^0(t),\mathfrak Y^\epsilon(t)+\mathscr Q^\epsilon(t))dt\\
			&=\int_E h(v^0(t),\xi)
			\psi^\epsilon(\xi,t)
			\mathbf1_{\{|\psi^\epsilon|\le\beta/a(\epsilon)\}}
			\nu(d\xi)dt,
		\end{split}
	\end{align}
	with $\mathfrak Y^\epsilon(0)=0$.

	The processes appearing in
	\eqref{PEeq-MDPMDPequ-4.69},
	\eqref{PEeq-MDPMDPequ-4.70}, and
	\eqref{PEeq-MDPMDPequ-4.75} have c\`adl\`ag $H$-valued vertical derivatives. Since
	$d_{J_1,H}(f,0)\le\sup_{t\le T}|f(t)|$, these estimates and the triangle inequality imply
	\[
	\mathscr Q^\epsilon\longrightarrow0
	\quad\text{in probability in }\mathfrak U_z.
	\]
	It follows from \cite[Lemma~3.2]{Budhiraja-AOP-2016} that
	$\{\psi^\epsilon\mathbf1_{\{|\psi^\epsilon|\le\beta/a(\epsilon)\}}\}_{\epsilon>0}$
	is tight in
	$\mathbb B_\nu((M\kappa_2(1))^{1/2})$ endowed with its weak topology.
	Hence the family
	$
	\left(
	\mathscr Q^\epsilon,
	\psi^\epsilon\mathbf1_{\{|\psi^\epsilon|\le\beta/a(\epsilon)\}}
	\right)$
	is tight in the Polish space
	$
	\mathfrak U_z\times
	\mathbb B_\nu((M\kappa_2(1))^{1/2})$.

	Take an arbitrary sequence $\epsilon_k\to0$ and, after extraction, let
	$(0,\psi)$ be a limit point. The Skorokhod representation theorem yields a probability space
	$(\widetilde\Omega,\widetilde{\mathscr F},
	\widetilde{\mathbb P})$ and random variables
	$(\widetilde{\mathscr Q}^{\epsilon_k},
	\widetilde\psi^{\epsilon_k})$ and
	$(0,\widetilde\psi)$ such that
	\[
	(\widetilde{\mathscr Q}^{\epsilon_k},
	\widetilde\psi^{\epsilon_k})
	\stackrel{d}{=}
	\left(
	\mathscr Q^{\epsilon_k},
	\psi^{\epsilon_k}
	\mathbf1_{\{|\psi^{\epsilon_k}|\le\beta/a(\epsilon_k)\}}
	\right),
	\]
	$(0,\widetilde\psi)\stackrel d=(0,\psi)$, and
	\[
	(\widetilde{\mathscr Q}^{\epsilon_k},
	\widetilde\psi^{\epsilon_k})
	\longrightarrow(0,\widetilde\psi)
	\quad\widetilde{\mathbb P}\text{-a.s. in }
	\mathfrak U_z\times
	\mathbb B_\nu((M\kappa_2(1))^{1/2}).
	\]
	For readability, the subsequence index $k$ is suppressed below.

	Set
	\[
	\widetilde{\mathfrak Y}^{\epsilon}
	:=\mathscr R^\epsilon(
	\widetilde{\mathscr Q}^{\epsilon},
	\widetilde\psi^\epsilon),
	\qquad
	\widetilde{\mathscr Z}^{\epsilon}
	:=\widetilde{\mathscr Q}^{\epsilon}
	+\widetilde{\mathfrak Y}^{\epsilon}.
	\]
	Because $\mathscr R^\epsilon$ is Borel on
	$\mathfrak U_z\times
	\mathbb B_\nu((M\kappa_2(1))^{1/2})$,
	the mapping theorem and equality in law show that
	$\widetilde{\mathscr Z}^{\epsilon}$ has the same law as
	$\mathscr Z^\epsilon$. Likewise,
	$\widetilde{\mathscr Z}:=\mathscr G^0(\widetilde\psi)$
	has the same law as
	$\mathscr Z=\mathscr G^0(\psi)$.

In view of the above arguments, the proof of the lemma is reduced to establishing the following claim:
	\begin{align}\label{PEeq-MDPMDPequ-4.77}
		\sup_{t \in [0,T]} |\widetilde{\mathscr{Z}}^{\epsilon}(t)-\widetilde{\mathscr{Z}}(t)|^2
		+\int_{0}^{T} \|\widetilde{\mathscr{Z}}^{\epsilon}(t)-\widetilde{\mathscr{Z}}(t)\|^2dt \rightarrow 0, ~~\widetilde{\mathbb{P}}-\text{a.s., ~as ~} \epsilon \rightarrow 0. 
	\end{align}
   Accordingly, we devote the rest of this proof to establishing \eqref{PEeq-MDPMDPequ-4.77}. Consider the following equations:
   \begin{align}\label{PEeq-MDPMDPequ-4.78}
   	d\widetilde{\mathscr{N}}^{\epsilon}(t)+A\widetilde{\mathscr{N}}^{\epsilon}(t)dt=\int_{E} h(v^0(t),\xi)\widetilde{\psi}^{\epsilon}(\xi,t)\nu(d\xi)dt, ~~~\widetilde{\mathscr{N}}^{\epsilon}(0)=0,
   \end{align}
   and
	\begin{align*}
		d\widetilde{\mathscr{N}}(t)+A\widetilde{\mathscr{N}}(t)dt=\int_{E} h(v^0(t),\xi)\widetilde{\psi}(\xi,t)\nu(d\xi)dt, ~~~\widetilde{\mathscr{N}}(0)=0.
	\end{align*}
	By an argument completely analogous to that used in \eqref{PEeq-MDPMDPlinm4.400-04.7}, we obtain 
	 \begin{align}\label{PEeq-MDPMDPequ-4.79}
	 	\lim_{\epsilon \to 0}\left(\sup_{t \in [0,T]} |\widetilde{\mathscr{N}}^{\epsilon}(t)-\widetilde{\mathscr{N}}(t)|^2+\int_0^T \|\widetilde{\mathscr{N}}^{\epsilon}(t)-\widetilde{\mathscr{N}}(t)\|^2dt\right)=0.
	 \end{align}
	Moreover, applying the chain rule to 
	$|\widetilde{\mathscr{N}}^{\epsilon}(t)|^2 + |\partial_{z} \widetilde{\mathscr{N}}^{\epsilon}(t)|^2$, 
	and invoking Lemma \ref{PEeq-MDPMDPlem-4.8}, conditions \textbf{(B.1)} and \textbf{(B.3)} 
	in Hypothesis~\ref{PEeqAssum-2.3}, together with \eqref{PEeq-MDPderm-4.10}, give, for each fixed $\widetilde\omega\in\widetilde\Omega$ and every $t\in[0,T]$,
	{\small
	\begin{align*}
		&~|\widetilde{\mathscr{N}}^{\epsilon}(t)|^2 + |\partial_{z} \widetilde{\mathscr{N}}^{\epsilon}(t)|^2
		+2\int_0^{t} \left(\|\widetilde{\mathscr{N}}^{\epsilon}(s)\|^2 + \|\partial_{z} \widetilde{\mathscr{N}}^{\epsilon}(s)\|^2\right)ds\notag\\
		&= 2\int_0^{t}\int_{E} \left(h(v^0(s),\xi), \widetilde{\mathscr{N}}^{\epsilon}(s)+\partial_{zz} \widetilde{\mathscr{N}}^{\epsilon}(s)\right)\widetilde{\psi}^{\epsilon}(\xi,s)\nu(d\xi)ds\notag\\
		&\leq C\sup_{t \in [0,T]}(|\widetilde{\mathscr{N}}^{\epsilon}(t)|+|\partial_{z} \widetilde{\mathscr{N}}^{\epsilon}(t)|)\sup_{t \in [0,T]}(1+|v^0(t)|+|\partial_{z} v^0(t)|)\int_0^{T}\int_{E}(C_{h}+\widetilde{C}_{h})|\widetilde{\psi}^{\epsilon}(\xi,t)|\nu(d\xi)dt\notag\\
		&\leq \frac{1}{2}\sup_{t\in[0,T]}\bigl(|\widetilde{\mathscr N}^{\epsilon}(t)|^2+|\partial_z\widetilde{\mathscr N}^{\epsilon}(t)|^2\bigr)\notag\\
		&\quad+\sup_{t\in[0,T]}\bigl(1+|v^0(t)|^2+|\partial_zv^0(t)|^2\bigr)\notag\\
		&\qquad\times\left(\sup_{\widetilde\psi^\epsilon\in\mathcal S_\epsilon^M}\int_0^T\int_E(C_h+\widetilde C_h)|\widetilde\psi^\epsilon(\xi,t)|\nu(d\xi)dt\right)^2,
	\end{align*}}
	and hence
	\begin{align}\label{PEeq-MDPMDPequ-4.80}
		\sup_{t \in [0,T]}(|\widetilde{\mathscr{N}}^{\epsilon}(t)|^2+|\partial_{z} \widetilde{\mathscr{N}}^{\epsilon}(t)|^2)
		+\int_0^{T} \left(\|\widetilde{\mathscr{N}}^{\epsilon}(t)\|^2 + \|\partial_{z} \widetilde{\mathscr{N}}^{\epsilon}(t)\|^2\right)dt\leq C(\widetilde{\omega})<\infty, ~~\widetilde{\mathbb{P}}-\text{a.s.}
	\end{align}
	From \eqref{PEeq-MDPMDPequ-4.79} and \eqref{PEeq-MDPMDPequ-4.80}, we have
	\begin{align}\label{PEeq-MDPMDPequ-4.080}
		\sup_{t \in [0,T]}(|\widetilde{\mathscr{N}}(t)|^2+|\partial_{z} \widetilde{\mathscr{N}}(t)|^2)
		+\int_0^{T} \left(\|\widetilde{\mathscr{N}}(t)\|^2 + \|\partial_{z} \widetilde{\mathscr{N}}(t)\|^2\right)dt\leq C(\widetilde{\omega})<\infty, ~~\widetilde{\mathbb{P}}-\text{a.s.}
	\end{align}
	
   Let $\widetilde{\Phi}=\widetilde{\mathscr{Z}}-\widetilde{\mathscr{N}}$ and $\widetilde{\Phi}^\epsilon=\widetilde{\mathscr{Z}}^\epsilon-\widetilde{\mathscr{Q}}^\epsilon-\widetilde{\mathscr{N}}^\epsilon$. Then $\widetilde{\Phi}$ and $\widetilde{\Phi}^\epsilon$ respectively satisfy the following equations:
   \begin{align}\label{PEeq-MDPMDPequ-4.81}
   	\left\{
   	\begin{aligned}
   		& d\widetilde{\Phi}(t)+A\widetilde{\Phi}(t)dt+B(\widetilde{\Phi}(t)+\widetilde{\mathscr{N}}(t),v^0(t))dt
   		+B(v^0(t),\widetilde{\Phi}(t)+\widetilde{\mathscr{N}}(t))dt=0,
   		\\
   		&\widetilde{\Phi}(0)=0,
   	\end{aligned}
   	\right.
   \end{align}
	and
	 \begin{align}\label{PEeq-MDPMDPequ-4.82}
		\left\{
		\begin{aligned}
				&d\widetilde{\Phi}^\epsilon(t)+A\widetilde{\Phi}^\epsilon(t)dt+B(\widetilde{\Phi}^\epsilon(t)+\widetilde{\mathscr{Q}}^\epsilon(t)+\widetilde{\mathscr{N}}^\epsilon(t),v^0(t))dt+B(v^0(t),\widetilde{\Phi}^\epsilon(t)+\widetilde{\mathscr{Q}}^\epsilon(t)+\widetilde{\mathscr{N}}^\epsilon(t))dt\\
			&+a(\epsilon)B(\widetilde{\Phi}^\epsilon(t)+\widetilde{\mathscr{Q}}^\epsilon(t)+\widetilde{\mathscr{N}}^\epsilon(t),\widetilde{\Phi}^\epsilon(t)+\widetilde{\mathscr{Q}}^\epsilon(t)+\widetilde{\mathscr{N}}^\epsilon(t))dt=0,
			\\
			&\widetilde{\Phi}^\epsilon(0)=0.
		\end{aligned}
		\right.
	\end{align}
	
The almost sure convergence
$\widetilde{\mathscr Q}^{\epsilon}\to0$ in $\mathfrak U_z$ obtained from the Skorokhod representation gives directly
\begin{align}\label{PEeq-MDPMDPequ-4.83}
	\lim_{\epsilon \to 0} \left(\sup_{t \in [0,T]}\left(|\widetilde{\mathscr{Q}}^{\epsilon}(t)|^2+|\partial_{z}\widetilde{\mathscr{Q}}^{\epsilon}(t)|^2\right)+\int_{0}^{T} \left(\|\widetilde{\mathscr{Q}}^{\epsilon}(t)\|^2+\|\partial_{z}\widetilde{\mathscr{Q}}^{\epsilon}(t)\|^2\right)dt\right)=0, \quad \widetilde{\mathbb{P}}\text{-a.s.}
\end{align}

In view of \eqref{PEeq-MDPMDPequ-4.83} and \eqref{PEeq-MDPMDPequ-4.79}, the proof of \eqref{PEeq-MDPMDPequ-4.77} reduces to establishing the following convergence:
\begin{align}\label{PEeq-MDPMDPequ-4.84}
	\sup_{t \in [0,T]} |\widetilde{\Phi}^{\epsilon}(t)-\widetilde{\Phi}(t)|^2
	+\int_{0}^{T} \|\widetilde{\Phi}^{\epsilon}(t)-\widetilde{\Phi}(t)\|^2dt \rightarrow 0, ~~\widetilde{\mathbb{P}}-\text{a.s., ~as ~} \epsilon \rightarrow 0. 
\end{align}	
	Using \eqref{PEeq-MDPderm-4.10}, \eqref{PEeq-MDPMDPequ-4.80}, and \eqref{PEeq-MDPMDPequ-4.83}, the same argument as in \eqref{PEeq-MDPskees-4.13} gives, for every fixed $\widetilde\omega\in\widetilde\Omega$,
	\begin{align}\label{PEeq-MDPMDPequ-4.085}
		\sup_{t \in [0,T]}(|\widetilde{\Phi}^{\epsilon}(t)|^2+|\partial_{z} \widetilde{\Phi}^{\epsilon}(t)|^2)
		+\int_0^{T} \left(\|\widetilde{\Phi}^{\epsilon}(t)\|^2 + \|\partial_{z} \widetilde{\Phi}^{\epsilon}(t)\|^2\right)dt\leq C(\widetilde{\omega})<\infty, ~~\widetilde{\mathbb{P}}-\text{a.s.}
	\end{align}
	
	Denote $\overline{\widetilde{\Phi}^\epsilon}={\widetilde{\Phi}}^\epsilon-{\widetilde{\Phi}}$ and $\overline{\widetilde{\mathscr{N}}^\epsilon}=\widetilde{\mathscr{N}}^\epsilon-\widetilde{\mathscr{N}}$. Then, by \eqref{PEeq-MDPMDPequ-4.81} and \eqref{PEeq-MDPMDPequ-4.82} we have
\begin{align}\label{PEeq-MDPMDPequ-4.86}
	\left\{
	\begin{aligned}
		&d\overline{\widetilde{\Phi}^\epsilon}(t)+A\overline{\widetilde{\Phi}^\epsilon}(t)dt+B(\overline{\widetilde{\Phi}^\epsilon}(t)+\overline{\widetilde{\mathscr{N}}^\epsilon}(t)+\widetilde{\mathscr{Q}}^\epsilon(t),v^0(t))dt+B(v^0(t),\overline{\widetilde{\Phi}^\epsilon}(t)+\overline{\widetilde{\mathscr{N}}^\epsilon}(t)+\widetilde{\mathscr{Q}}^\epsilon(t))dt\\
		&+a(\epsilon)B(\widetilde{\Phi}^\epsilon(t)+\widetilde{\mathscr{Q}}^\epsilon(t)+{\widetilde{\mathscr{N}}^\epsilon}(t),\widetilde{\Phi}^\epsilon(t)+\widetilde{\mathscr{Q}}^\epsilon(t)+\widetilde{\mathscr{N}}^\epsilon(t))dt=0,
		\\
		&\overline{\widetilde{\Phi}^\epsilon}(0)=0.
	\end{aligned}
	\right.
\end{align}
Applying the chain rule to $|\overline{\widetilde{\Phi}^\epsilon}(t)|^2$, we have
	\begin{align}\label{PEeq-MDPMDPequ-4.87}
		\begin{split}
			|\overline{\widetilde{\Phi}^\epsilon}(t)|^2&+2\int_0^t \|\overline{\widetilde{\Phi}^\epsilon}(s)\|^2ds=-2\int_0^t \left(B(\overline{\widetilde{\Phi}^\epsilon}(s)+\overline{\widetilde{\mathscr{N}}^\epsilon}(s)+\widetilde{\mathscr{Q}}^\epsilon(s),v^0(s)),\overline{\widetilde{\Phi}^\epsilon}(s)\right)ds\\
			&\qquad \qquad -2\int_0^t \left(B(v^0(s),\overline{\widetilde{\Phi}^\epsilon}(s)+\overline{\widetilde{\mathscr{N}}^\epsilon}(s)+\widetilde{\mathscr{Q}}^\epsilon(s)),\overline{\widetilde{\Phi}^\epsilon}(s)\right)ds\\
			&-2a(\epsilon)\int_0^t\left(B(\widetilde{\Phi}^\epsilon(s)+\widetilde{\mathscr{Q}}^\epsilon(s)+{\widetilde{\mathscr{N}}^\epsilon}(s),\widetilde{\Phi}^\epsilon(s)+\widetilde{\mathscr{Q}}^\epsilon(s)+\widetilde{\mathscr{N}}^\epsilon(s)),\overline{\widetilde{\Phi}^\epsilon}(s)\right)ds\\
			&\qquad \qquad\qquad \qquad:=\Upsilon_{16}(t)+\Upsilon_{17}(t)+\Upsilon_{18}(t).
		\end{split}
	\end{align}
Fix $\widetilde{\omega}\in \widetilde{\Omega}$. We now estimate the right-hand side of \eqref{PEeq-MDPMDPequ-4.87}. For $\Upsilon_{16}(t)$, using \eqref{PEeq-2.11} and Young's inequality, we obtain
\begingroup\small
\begin{align}\label{PEeq-MDPMDPequ-4.88}
	|\Upsilon_{16}(t)|
	&\leq 2\int_0^t \left(\left|\left(B(\overline{\widetilde{\Phi}^\epsilon}(s),v^0(s)),\overline{\widetilde{\Phi}^\epsilon}(s)\right)\right|+\left|\left(B(\overline{\widetilde{\mathscr{N}}^\epsilon}(s)+\widetilde{\mathscr{Q}}^\epsilon(s),v^0(s)),\overline{\widetilde{\Phi}^\epsilon}(s)\right)\right|\right)ds\notag\\
	&\leq C \int_0^t \left(|\overline{\widetilde{\Phi}^\epsilon}(s)| \|\overline{\widetilde{\Phi}^\epsilon}(s)\| \|v^0(s)\| + \|\overline{\widetilde{\Phi}^\epsilon}(s)\|^{3/2} |\partial_{z} v^0(s)| |\overline{\widetilde{\Phi}^\epsilon}(s)|^{1/2}\right)ds\notag\\
	&+C\int_0^t |\overline{\widetilde{\mathscr{N}}^\epsilon}(s)+\widetilde{\mathscr{Q}}^\epsilon(s)|^{1/2} \|\overline{\widetilde{\mathscr{N}}^\epsilon}(s)+\widetilde{\mathscr{Q}}^\epsilon(s)\|^{1/2} \|v^0(s)\| |\overline{\widetilde{\Phi}^\epsilon}(s)|^{1/2} \|\overline{\widetilde{\Phi}^\epsilon}(s)\|^{1/2} ds\notag\\
	&+C\int_0^t \|\overline{\widetilde{\mathscr{N}}^\epsilon}(s)+\widetilde{\mathscr{Q}}^\epsilon(s)\| |\partial_{z} v^0(s)| |\overline{\widetilde{\Phi}^\epsilon}(s)|^{1/2} \|\overline{\widetilde{\Phi}^\epsilon}(s)\|^{1/2} ds\notag \\
	&\leq \frac{1}{4}\int_0^t \|\overline{\widetilde{\Phi}^\epsilon}(s)\|^2ds
	+C \int_0^t \left(\|v^0(s)\|^2+|\partial_{z} v^0(s)|^4\right) |\overline{\widetilde{\Phi}^\epsilon}(s)|^2ds\notag\\
	&+\frac{1}{8}\int_0^t \|\overline{\widetilde{\Phi}^\epsilon}(s)\|^2ds
	+C\int_0^t \|v^0(s)\|^2|\overline{\widetilde{\Phi}^\epsilon}(s)|^2ds
	+C\int_0^t \|v^0(s)\| |\overline{\widetilde{\mathscr{N}}^\epsilon}(s)+\widetilde{\mathscr{Q}}^\epsilon(s)| \|\overline{\widetilde{\mathscr{N}}^\epsilon}(s)+\widetilde{\mathscr{Q}}^\epsilon(s)\|ds\notag\\
	&+\frac{1}{8}\int_0^t \|\overline{\widetilde{\Phi}^\epsilon}(s)\|^2ds
	+C\int_0^t  |\overline{\widetilde{\Phi}^\epsilon}(s)|^2ds
	+C \int_0^t \|\overline{\widetilde{\mathscr{N}}^\epsilon}(s)+\widetilde{\mathscr{Q}}^\epsilon(s)\|^2 |\partial_{z} v^0(s)|^2 ds\notag\\
	&\leq \frac{1}{2}\int_0^t \|\overline{\widetilde{\Phi}^\epsilon}(s)\|^2ds
	+C \int_0^t \left(1+\|v^0(s)\|^2+|\partial_{z} v^0(s)|^4\right) |\overline{\widetilde{\Phi}^\epsilon}(s)|^2ds\notag\\
	&+C\sup_{s\in[0,T]}|\overline{\widetilde{\mathscr N}^\epsilon}(s)+\widetilde{\mathscr Q}^\epsilon(s)|\left(\int_0^t\|v^0(s)\|^2ds\right)^{1/2}\left(\int_0^t\|\overline{\widetilde{\mathscr N}^\epsilon}(s)+\widetilde{\mathscr Q}^\epsilon(s)\|^2ds\right)^{1/2}\notag\\
	&+C\sup_{s \in [0,T]}|\partial_{z} v^0(s)|^2\int_0^t\|\overline{\widetilde{\mathscr{N}}^\epsilon}(s)+\widetilde{\mathscr{Q}}^\epsilon(s)\|^2ds.
\end{align}
\endgroup
For $\Upsilon_{17}(t)$, combining \eqref{PEeq-2.13}, an analogous computation gives
\begin{align}\label{PEeq-MDPMDPequ-4.90}
\left|\Upsilon_{17}(t)\right|
&\leq C	\int_0^t \left|\left(B(v^0(s),\overline{\widetilde{\mathscr{N}}^\epsilon}(s)+\widetilde{\mathscr{Q}}^\epsilon(s))),\overline{\widetilde{\Phi}^\epsilon}(s)\right)\right| ds\notag\\
&\leq C\int_0^t |v^0(s)|^{1/2}\|v^0(s)\|^{1/2} \|\overline{\widetilde{\mathscr{N}}^\epsilon}(s)+\widetilde{\mathscr{Q}}^\epsilon(s)\|  |\overline{\widetilde{\Phi}^\epsilon}(s)|^{1/2} \|\overline{\widetilde{\Phi}^\epsilon}(s)\|^{1/2}ds\notag\\
&+ C\int_0^t \|v^0(s)\| |\partial_{z}(\overline{\widetilde{\mathscr{N}}^\epsilon}(s)+\widetilde{\mathscr{Q}}^\epsilon(s))||\overline{\widetilde{\Phi}^\epsilon}(s)|^{1/2} \|\overline{\widetilde{\Phi}^\epsilon}(s)\|^{1/2}ds\notag\\
&\leq \frac{1}{8}\int_0^t \|\overline{\widetilde{\Phi}^\epsilon}(s)\|^2ds
+C\int_0^t \|v^0(s)\|^2 |\overline{\widetilde{\Phi}^\epsilon}(s)|^2ds
+C\int_0^t |v^0(s)|^2\|\overline{\widetilde{\mathscr{N}}^\epsilon}(s)+\widetilde{\mathscr{Q}}^\epsilon(s)\|^2ds\notag\\
&+\frac{1}{8}\int_0^t \|\overline{\widetilde{\Phi}^\epsilon}(s)\|^2ds
+C\int_0^t \|v^0(s)\|^2|\overline{\widetilde{\Phi}^\epsilon}(s)|^2ds
+C\int_0^t \|v^0(s)\| |\partial_{z}(\overline{\widetilde{\mathscr{N}}^\epsilon}(s)+\widetilde{\mathscr{Q}}^\epsilon(s))|^2ds\notag\\
&\leq \frac{1}{4}\int_0^t \|\overline{\widetilde{\Phi}^\epsilon}(s)\|^2ds
+C\int_0^t (1+\|v^0(s)\|^2) |\overline{\widetilde{\Phi}^\epsilon}(s)|^2ds
+C\sup_{s \in [0,T]} |v^0(s)|^2 \int_0^t \|\overline{\widetilde{\mathscr{N}}^\epsilon}(s)+\widetilde{\mathscr{Q}}^\epsilon(s)\|^2ds\notag\\
&+C\sup_{s\in [0,T]}|\partial_{z}(\overline{\widetilde{\mathscr{N}}^\epsilon}(s)+\widetilde{\mathscr{Q}}^\epsilon(s))| \left(\int_0^t \|v^0(s)\|^2 ds\right)^{1/2}
\left(\int_0^t \|(\overline{\widetilde{\mathscr{N}}^\epsilon}(s)+\widetilde{\mathscr{Q}}^\epsilon(s))\|^2 ds\right)^{1/2}.
\end{align}
It remains to deal with $\Upsilon_{18}(t)$. Using H\"{o}lder's inequality, \eqref{PEeq-2.0013}, Young's inequality and Poincar\'{e}'s inequality, we obtain
\begingroup\small
\begin{align}\label{PEeq-MDPMDPequ-4.91}
	\left|\Upsilon_{18}(t)\right|
	&\leq |a(\epsilon)|\int_0^{t} \|B(\widetilde{\Phi}^\epsilon(s)+\widetilde{\mathscr{Q}}^\epsilon(s)+{\widetilde{\mathscr{N}}^\epsilon}(s),\widetilde{\Phi}^\epsilon(s)+\widetilde{\mathscr{Q}}^\epsilon(s)+\widetilde{\mathscr{N}}^\epsilon(s))\|_{V'}\|\overline{\widetilde{\Phi}^\epsilon}(s)\|ds\notag\\
	&\leq C|a(\epsilon)|\int_0^{t}|\widetilde{\Phi}^\epsilon(s)+\widetilde{\mathscr{Q}}^\epsilon(s)+{\widetilde{\mathscr{N}}^\epsilon}(s)|\|\widetilde{\Phi}^\epsilon(s)+\widetilde{\mathscr{Q}}^\epsilon(s)+{\widetilde{\mathscr{N}}^\epsilon}(s)\| \|\overline{\widetilde{\Phi}^\epsilon}(s)\|ds\notag\\
	&+C|a(\epsilon)|\int_0^{t}|\partial_{z}(\widetilde{\Phi}^\epsilon(s)+\widetilde{\mathscr{Q}}^\epsilon(s)+{\widetilde{\mathscr{N}}^\epsilon}(s))|
	\|\widetilde{\Phi}^\epsilon(s)+\widetilde{\mathscr{Q}}^\epsilon(s)+{\widetilde{\mathscr{N}}^\epsilon}(s)\| \|\overline{\widetilde{\Phi}^\epsilon}(s)\|ds\notag\\
	&+C|a(\epsilon)|\int_0^{t} |\widetilde{\Phi}^\epsilon(s)+\widetilde{\mathscr{Q}}^\epsilon(s)+{\widetilde{\mathscr{N}}^\epsilon}(s)| \|\partial_{z}(\widetilde{\Phi}^\epsilon(s)+\widetilde{\mathscr{Q}}^\epsilon(s)+{\widetilde{\mathscr{N}}^\epsilon}(s))\|\|\overline{\widetilde{\Phi}^\epsilon}(s)\|ds\notag\\
	&+C|a(\epsilon)|\int_0^{t}|\widetilde{\Phi}^\epsilon(s)+\widetilde{\mathscr{Q}}^\epsilon(s)+{\widetilde{\mathscr{N}}^\epsilon}(s)| |\partial_{z}(\widetilde{\Phi}^\epsilon(s)+\widetilde{\mathscr{Q}}^\epsilon(s)+{\widetilde{\mathscr{N}}^\epsilon}(s))|^{1/2}\notag\\
	&\qquad \times \|\partial_{z}(\widetilde{\Phi}^\epsilon(s)+\widetilde{\mathscr{Q}}^\epsilon(s)+{\widetilde{\mathscr{N}}^\epsilon}(s))\|^{1/2}\|\overline{\widetilde{\Phi}^\epsilon}(s)\|ds\notag\\
	&\leq \frac{1}{4}\int_0^t \|\overline{\widetilde{\Phi}^\epsilon}(s)\|^2ds
	+C|a(\epsilon)|^2\sup_{s\in [0,T]}\left(|\widetilde{\Phi}^\epsilon(s)+\widetilde{\mathscr{Q}}^\epsilon(s)+{\widetilde{\mathscr{N}}^\epsilon}(s)|^2+|\partial_{z}(\widetilde{\Phi}^\epsilon(s)+\widetilde{\mathscr{Q}}^\epsilon(s)+{\widetilde{\mathscr{N}}^\epsilon}(s))|^2
\right)\notag\\
&\qquad \times \int_0^{t}\left(\|\widetilde{\Phi}^\epsilon(s)+\widetilde{\mathscr{Q}}^\epsilon(s)+{\widetilde{\mathscr{N}}^\epsilon}(s)\|^2+\|\partial_{z}(\widetilde{\Phi}^\epsilon(s)+\widetilde{\mathscr{Q}}^\epsilon(s)+{\widetilde{\mathscr{N}}^\epsilon}(s))\|^2
\right)ds\notag\\
&+C|a(\epsilon)|^2\sup_{s\in [0,T]}\left(|\widetilde{\Phi}^\epsilon(s)+\widetilde{\mathscr{Q}}^\epsilon(s)+{\widetilde{\mathscr{N}}^\epsilon}(s)|^2\right) \int_0^t |\partial_z(\widetilde\Phi^\epsilon+\widetilde{\mathscr Q}^\epsilon+\widetilde{\mathscr N}^\epsilon)(s)|
\,\|\partial_z(\widetilde\Phi^\epsilon+\widetilde{\mathscr Q}^\epsilon+\widetilde{\mathscr N}^\epsilon)(s)\|ds\notag\\
&\leq \frac{1}{4}\int_0^t \|\overline{\widetilde{\Phi}^\epsilon}(s)\|^2ds+C|a(\epsilon)|^2\sup_{s\in [0,T]}\Bigl(
|\widetilde{\Phi}^\epsilon(s)+\widetilde{\mathscr{Q}}^\epsilon(s)+{\widetilde{\mathscr{N}}^\epsilon}(s)|^2+|\partial_{z}(\widetilde{\Phi}^\epsilon(s)+\widetilde{\mathscr{Q}}^\epsilon(s)+{\widetilde{\mathscr{N}}^\epsilon}(s))|^2\Bigr)\notag\\
&\qquad\times \int_0^{t}\Bigl(
\|\widetilde{\Phi}^\epsilon(s)+\widetilde{\mathscr{Q}}^\epsilon(s)+{\widetilde{\mathscr{N}}^\epsilon}(s)\|^2+\|\partial_{z}(\widetilde{\Phi}^\epsilon(s)+\widetilde{\mathscr{Q}}^\epsilon(s)+{\widetilde{\mathscr{N}}^\epsilon}(s))\|^2\Bigr)ds.
\end{align}
\endgroup

By \eqref{PEeq-MDPMDPequ-4.87}--\eqref{PEeq-MDPMDPequ-4.91}, we obtain
\begin{align*}
	&~~|\overline{\widetilde{\Phi}^\epsilon}(t)|^2+\int_0^t \|\overline{\widetilde{\Phi}^\epsilon}(s)\|^2ds
	\leq C \int_0^t \left(1+\|v^0(s)\|^2+|\partial_{z} v^0(s)|^4\right) |\overline{\widetilde{\Phi}^\epsilon}(s)|^2ds\notag\\
	&+C\sup_{s \in [0,T]}\left(|\overline{\widetilde{\mathscr{N}}^\epsilon}(s)+\widetilde{\mathscr{Q}}^\epsilon(s)|+|\partial_{z}(\overline{\widetilde{\mathscr{N}}^\epsilon}(s)+\widetilde{\mathscr{Q}}^\epsilon(s))|\right) \left(\int_0^t\|v^0(s)\|^2ds\right)^{1/2}\notag\\
	&\quad \times\left(\int_0^t\|\overline{\widetilde{\mathscr{N}}^\epsilon}(s)+\widetilde{\mathscr{Q}}^\epsilon(s)\|^2ds\right)^{1/2}+C\sup_{s \in [0,T]}(|v^0(s)|^2+|\partial_{z} v^0(s)|^2)\int_0^t\|\overline{\widetilde{\mathscr{N}}^\epsilon}(s)+\widetilde{\mathscr{Q}}^\epsilon(s)\|^2ds\notag\\
	&+C|a(\epsilon)|^2\sup_{s\in [0,T]}\left(|\widetilde{\Phi}^\epsilon(s)+\widetilde{\mathscr{Q}}^\epsilon(s)+{\widetilde{\mathscr{N}}^\epsilon}(s)|^2+|\partial_{z}(\widetilde{\Phi}^\epsilon(s)+\widetilde{\mathscr{Q}}^\epsilon(s)+{\widetilde{\mathscr{N}}^\epsilon}(s))|^2
	\right)\notag\\
	&\quad \times \int_0^{t}\left(\|\widetilde{\Phi}^\epsilon(s)+\widetilde{\mathscr{Q}}^\epsilon(s)+{\widetilde{\mathscr{N}}^\epsilon}(s)\|^2+\|\partial_{z}(\widetilde{\Phi}^\epsilon(s)+\widetilde{\mathscr{Q}}^\epsilon(s)+{\widetilde{\mathscr{N}}^\epsilon}(s))\|^2
	\right)ds,
\end{align*}

By \eqref{PEeq-MDPMDPequ-4.87}-\eqref{PEeq-MDPMDPequ-4.91} we derive
\begin{align*}
&~~|\overline{\widetilde{\Phi}^\epsilon}(t)|^2+\int_0^t \|\overline{\widetilde{\Phi}^\epsilon}(s)\|^2ds
\leq C \int_0^t \left(1+\|v^0(s)\|^2+|\partial_{z} v^0(s)|^4\right) |\overline{\widetilde{\Phi}^\epsilon}(s)|^2ds\notag\\
&+C\sup_{s \in [0,T]}\left(|\overline{\widetilde{\mathscr{N}}^\epsilon}(s)+\widetilde{\mathscr{Q}}^\epsilon(s)|+|\partial_{z}(\overline{\widetilde{\mathscr{N}}^\epsilon}(s)+\widetilde{\mathscr{Q}}^\epsilon(s))|\right) \left(\int_0^t\|v^0(s)\|^2ds\right)^{1/2}\notag\\
&\quad \times\left(\int_0^t\|\overline{\widetilde{\mathscr{N}}^\epsilon}(s)+\widetilde{\mathscr{Q}}^\epsilon(s)\|^2ds\right)^{1/2}+C\sup_{s \in [0,T]}(|v^0(s)|^2+|\partial_{z} v^0(s)|^2)\int_0^t\|\overline{\widetilde{\mathscr{N}}^\epsilon}(s)+\widetilde{\mathscr{Q}}^\epsilon(s)\|^2ds\notag\\
&+C|a(\epsilon)|^2\sup_{s\in [0,T]}\left(|\widetilde{\Phi}^\epsilon(s)+\widetilde{\mathscr{Q}}^\epsilon(s)+{\widetilde{\mathscr{N}}^\epsilon}(s)|^2+|\partial_{z}(\widetilde{\Phi}^\epsilon(s)+\widetilde{\mathscr{Q}}^\epsilon(s)+{\widetilde{\mathscr{N}}^\epsilon}(s))|^2
\right)\notag\\
&\quad \times \int_0^{t}\left(\|\widetilde{\Phi}^\epsilon(s)+\widetilde{\mathscr{Q}}^\epsilon(s)+{\widetilde{\mathscr{N}}^\epsilon}(s)\|^2+\|\partial_{z}(\widetilde{\Phi}^\epsilon(s)+\widetilde{\mathscr{Q}}^\epsilon(s)+{\widetilde{\mathscr{N}}^\epsilon}(s))\|^2
\right)ds,
\end{align*}
which, together with \eqref{PEeq-MDPderm-4.10}, \eqref{PEeq-MDPMDPequ-4.80}-\eqref{PEeq-MDPMDPequ-4.080} and \eqref{PEeq-MDPMDPequ-4.085}, yields
{\small
\begin{align}\label{PEeq-MDPMDPequ-4.92}
&~~|\overline{\widetilde{\Phi}^\epsilon}(t)|^2+\int_0^t \|\overline{\widetilde{\Phi}^\epsilon}(s)\|^2ds
\leq C \int_0^t \left(1+\|v^0(s)\|^2+|\partial_{z} v^0(s)|^4\right) |\overline{\widetilde{\Phi}^\epsilon}(s)|^2ds\notag\\
&+C_TC(\widetilde{\omega})\sup_{t \in [0,T]}\left(1+|\widetilde{\mathscr{Q}}^\epsilon(t)|^2+|\partial_{z}\widetilde{\mathscr{Q}}^\epsilon(t)|^2\right) \left(\int_0^T\|\overline{\widetilde{\mathscr{N}}^\epsilon}(t)+\widetilde{\mathscr{Q}}^\epsilon(t)\|^2dt\right)^{1/2}+C_T\int_0^T\|\overline{\widetilde{\mathscr{N}}^\epsilon}(t)+\widetilde{\mathscr{Q}}^\epsilon(t)\|^2dt\notag\\
&+CC(\widetilde{\omega})|a(\epsilon)|^2\sup_{t\in [0,T]}\left(1+|\widetilde{\mathscr{Q}}^\epsilon(t)|^2+|\partial_{z}\widetilde{\mathscr{Q}}^\epsilon(t)|^2
\right) \int_0^{T}\left(\|\widetilde{\mathscr{Q}}^\epsilon(t)\|^2+\|\partial_{z}\widetilde{\mathscr{Q}}^\epsilon(t)\|^2
\right)dt,
\end{align}
}
where we used $|\overline{\widetilde{\mathscr{N}}^\epsilon}+\widetilde{\mathscr{Q}}^\epsilon|\leq |\overline{\widetilde{\mathscr{N}}^\epsilon}|+|\widetilde{\mathscr{Q}}^\epsilon|\leq |{\widetilde{\mathscr{N}}^\epsilon}|+|{\widetilde{\mathscr{N}}}|+|\widetilde{\mathscr{Q}}^\epsilon|$.

Moreover, it follows from \eqref{PEeq-MDPMDPequ-4.79} and \eqref{PEeq-MDPMDPequ-4.83} that
\[
\int_0^T\|\overline{\widetilde{\mathscr{N}}^\epsilon}(t)+\widetilde{\mathscr{Q}}^\epsilon(t)\|^2dt \rightarrow 0, \quad \widetilde{\mathbb{P}}\text{-a.s.},\quad \text{as } \epsilon \rightarrow 0.
\]
Since $a(\epsilon) \rightarrow 0$ as $\epsilon\rightarrow 0$, combining these facts with \eqref{PEeq-MDPderm-4.10}, \eqref{PEeq-MDPMDPequ-4.83}, \eqref{PEeq-MDPMDPequ-4.92} and Gronwall's inequality yields
$$
\lim_{\epsilon \to 0} \left(\sup_{t \in [0,T]}|\overline{\widetilde{\Phi}^\epsilon}(t)|^2+\int_0^T \|\overline{\widetilde{\Phi}^\epsilon}(t)\|^2dt\right)=0, \quad \widetilde{\mathbb{P}}\text{-a.s.}
$$
Together with \eqref{PEeq-MDPMDPequ-4.83}, this proves \eqref{PEeq-MDPMDPequ-4.77} along the represented subsequence. Since the original sequence $\epsilon_k\to0$ was arbitrary, the subsequence characterization of convergence in distribution gives
\[
\mathscr G^\epsilon(\epsilon N^{\epsilon^{-1}\phi^\epsilon})
\xrightarrow{d}\mathscr G^0(\psi)
\quad\text{in }(\mathfrak U,d_{\mathfrak U}).
\]
Therefore, Condition~$(A_2)$ holds in the stated strong path topology. This completes the proof.
\end{proof}

\subsection*{ Conflict of Interests}
The authors declare that they have no known competing financial interests or personal relationships that could have appeared to influence the
work reported in this paper.
\subsection*{ Author Contributions} All authors contributed to proving the main results, drafting the manuscript, and
reviewing the manuscript.
\subsection*{ Data Availability } 
No datasets were generated or analysed during the current study.

\end{document}